\documentclass[11pt, reqno]{amsart}
\usepackage{amsmath,amssymb,amsthm,mathrsfs,enumerate,bm,xcolor,multirow,pbox}
\usepackage{graphicx,color,framed,tikz,caption,subcaption}
\usepackage{enumitem}
\setlist{leftmargin=5mm}
\usepackage[colorlinks,linkcolor=red,citecolor=blue,urlcolor=blue]{hyperref}
\allowdisplaybreaks[4]
\numberwithin{equation}{section}
\newcommand{\N}{\mathbb{N}}
\newcommand{\R}{\mathbb{R}}

\newcommand{\E}{\mathbb{E}}
\newcommand{\Prob}{\mathbb{P}}

\newcommand{\pnorm}[2]{\lVert#1\rVert_{#2}}
\newcommand{\bigpnorm}[2]{\big\lVert#1\big\rVert_{#2}}

\newcommand{\abs}[1]{\lvert#1\rvert}
\newcommand{\bigabs}[1]{\big\lvert#1\big\rvert}
\newcommand{\biggabs}[1]{\bigg\lvert#1\bigg\rvert}
\newcommand{\iprod}[2]{\left\langle#1,#2\right\rangle}
\renewcommand{\epsilon}{\varepsilon}

\newcommand{\floor}[1]{\left\lfloor #1 \right\rfloor}
\newcommand{\ceil}[1]{\left\lceil #1 \right\rceil}

\newcommand{\equald}{\stackrel{d}{=}}
\renewcommand{\hat}{\widehat}
\renewcommand{\tilde}{\widetilde}
\DeclareMathOperator{\dv}{div}
\DeclareMathOperator{\tr}{tr}
\DeclareMathOperator{\var}{Var}
\DeclareMathOperator{\rank}{rank}

\newcommand{\beq}{\begin{equation}}
\newcommand{\eeq}{\end{equation}}
\newcommand{\beqa}{\begin{equation} \begin{aligned}}
\newcommand{\eeqa}{\end{aligned} \end{equation}}
\newcommand{\beqas}{\begin{equation*} \begin{aligned}}
\newcommand{\eeqas}{\end{aligned} \end{equation*}}

\newcommand{\bit}{\begin{itemize}}
	\newcommand{\eit}{\end{itemize}}
\newcommand{\bmat}{\begin{bmatrix}}
	\newcommand{\emat}{\end{bmatrix}}

\DeclareMathOperator*{\argmin}{arg\,min}

\theoremstyle{definition}\newtheorem{problem}{Problem}[section]
\theoremstyle{definition}\newtheorem{definition}[problem]{Definition}
\theoremstyle{remark}

\theoremstyle{remark}\newtheorem{remark}[problem]{Remark}
\theoremstyle{definition}\newtheorem{example}[problem]{Example}
\theoremstyle{plain}\newtheorem{theorem}[problem]{Theorem}
\theoremstyle{plain}
\theoremstyle{plain}\newtheorem{lemma}[problem]{Lemma}
\theoremstyle{plain}\newtheorem{proposition}[problem]{Proposition}
\theoremstyle{plain}\newtheorem{corollary}[problem]{Corollary}
\theoremstyle{plain}

\AtBeginDocument{%
	\def\MR#1{}
}

\begin{document}

\title[High dimensional asymptotics of likelihood ratio tests]{High dimensional asymptotics of likelihood ratio tests in the Gaussian sequence model under convex constraints}
\thanks{The research of Q. Han is partially supported by DMS-1916221. The research of B. Sen is partially supported by DMS-2015376.}

\author[Q. Han]{Qiyang Han}

\address[Q. Han]{
Department of Statistics, Rutgers University, Piscataway, NJ 08854, USA.
}
\email{qh85@stat.rutgers.edu}

\author[B. Sen]{Bodhisattva Sen}

\address[B. Sen]{
	Department of Statistics, Columbia University, New York, NY 10027, USA.
}
\email{bodhi@stat.columbia.edu}

\author[Y. Shen]{Yandi Shen}

\address[Y. Shen]{
Department of Statistics, University of Washington, Seattle, WA 98105, USA.
}
\email{ydshen@uw.edu}

\date{\today}


\begin{abstract}
In the Gaussian sequence model $Y=\mu+\xi$, we study the likelihood ratio test (LRT) for testing $H_0: \mu=\mu_0$ versus $H_1: \mu \in K$, where $\mu_0 \in K$, and $K$ is a closed convex set in $\R^n$. In particular, we show that under the null hypothesis, normal approximation holds for the log-likelihood ratio statistic for a general pair $(\mu_0,K)$, in the high dimensional regime where the estimation error of the associated least squares estimator diverges in an appropriate sense. The normal approximation further leads to a precise characterization of the power behavior of the LRT in the high dimensional regime. These characterizations show that the power behavior of the LRT is in general non-uniform with respect to the Euclidean metric, and illustrate the conservative nature of existing minimax optimality and sub-optimality results for the LRT. A variety of examples, including testing in the orthant/circular cone, isotonic regression, Lasso, and testing parametric assumptions versus shape-constrained alternatives, are worked out to demonstrate the versatility of the developed theory.
\end{abstract}

\maketitle



\section{Introduction}

\subsection{The likelihood ratio test}

Consider the Gaussian sequence model
\begin{align}\label{model:sequence}
Y = \mu + \xi, 
\end{align}
where $\mu \in \R^n$ is unknown and $\xi=(\xi_1,\ldots,\xi_n)$ is an $n$-dimensional standard Gaussian vector. In a variety of applications, prior knowledge on the mean vector $\mu$ can be naturally translated into the constraint $\mu\in K$, where $K$ is a closed convex set in $\R^n$. Two such important examples that will be considered in this paper are: (i) Lasso in the constrained form \cite{tibshirani1996regression}, where $K$ is an $\ell_1$-norm ball, and (ii) isotonic regression \cite{chatterjee2015risk}, where $K$ is the cone consisting of monotone sequences. We also refer the readers to \cite{baraud2002nonasymptotic, juditsky2002nonparametric, baraud2005testing, chatterjee2014new,guntuboyina2017nonparametric} and many references therein for a diverse list of further concrete examples of $K$. In this paper, we will be interested in the following `goodness-of-fit' testing problem:
\begin{align}\label{eq:testing_generic}
H_0: \mu = \mu_0\qquad \textrm{versus}\qquad H_1: \mu \in K,
\end{align}
where $\mu_0\in K\subset \R^n$, and $K$ is an arbitrary closed and convex subset of $\R^n$. Throughout the manuscript, the asymptotics will take place as $n\to \infty$, and the explicit dependence of $\mu,\mu_0,K$ and related quantities on the dimension $n$ will be suppressed for ease of notation.

Given observation $Y$ generated from model (\ref{model:sequence}), arguably the most natural and generic test for (\ref{eq:testing_generic}) is the likelihood ratio test (LRT). Under the Gaussian model (\ref{model:sequence}), the log-likelihood ratio statistic (LRS) for (\ref{eq:testing_generic}) takes the form
\begin{align}\label{def:TY}
T(Y) &\equiv \pnorm{Y-\mu_0}{}^2 - \pnorm{Y-\hat{\mu}_K}{}^2 \nonumber \\
&= \pnorm{\mu+\xi-\mu_0}{}^2 - \pnorm{\mu+\xi-\Pi_K(\mu+\xi)}{}^2 \geq 0.
\end{align}
Here $\hat{\mu}_K\equiv \Pi_K(Y) \equiv\arg\min_{\nu \in K} \pnorm{Y-\nu}{}^2$ is the metric projection of $Y$ onto the constraint set $K$ with respect to the canonical Euclidean $\ell_2$ norm $\|\cdot\|$ on $\R^n$. As $K$ is both closed and convex, $\Pi_K$ is well-defined, and the resulting $\hat{\mu}_K$ is both the least squares estimator (LSE) and the maximum likelihood estimator for the mean vector $\mu$ under the Gaussian model (\ref{model:sequence}). The risk behavior of $\hat{\mu}_K$ is completely characterized in the recent work \cite{chatterjee2014new}.

The LRT for (\ref{eq:testing_generic}) and its generalizations thereof have gained extensive attention in the literature, see e.g.~\cite{chernoff1954distribution,bartholomew1959test,bartholomew1959test2,bartholomew1961ordered,bartholomew1961test,kudo1963multivariate,barlow1972statistical,kudo1975generalized,robertson1978likelihood,warrack1984likelihood,shapiro1985asymptotic,raubertas1986hypothesis,robertson1988order,shapiro1988towards,menendez1991anomalies,menendez1992dominance,menendez1992testing,durot2001goodness,meyer2003test,sen2017testing,wei2019geometry} for an incomplete list. In our setting, an immediate way to use the LRS $T(Y)$ in (\ref{def:TY}) to form a test is to simulate the critical values of $T(Y)$ under $H_0$. More precisely, for any confidence level $\alpha \in (0,1)$, we may determine through simulations an acceptance region $\mathcal{I}_\alpha\subset \R$ such that the LRS satisfies $\Prob\big(T(Y)\in \mathcal{I}_\alpha\big)=1-\alpha$ under $H_0$, and then formulate the LRT accordingly. In some special cases, including the classical setting where $K$ is a subspace, the distribution of $T(Y)$ under the null is even known in closed form, so the simulation step can be skipped.

Clearly, the almost effortless LRT as described above already gives an exact type I error control at the prescribed level for a generic pair $(\mu_0,K)$. The equally important question of its power behavior, however, is more complicated and requires a much deeper level of investigation. In the classical setting of parametric models and certain semiparametric models, the power behavior of the LRT can be precisely determined, at least asymptotically, for contiguous alternatives in the corresponding parameter spaces, cf.~\cite{van2000asymptotic,van2002semiparametric}. An important and basic ingredient for the success of power analysis in these settings is the existence of a limiting distribution of the LRT under $H_0$ that can be `perturbed' in a large number of directions of alternatives.

Unfortunately, the distribution of the LRS $T(Y)$ in (\ref{def:TY}) under the null, for both finite-sample and asymptotic regimes, is only understood in very few cases. One such case is, as mentioned above, the classical setting where $K$ is a subspace of dimension $\dim(K)$. Then the null distribution of $T(Y)$ is a chi-squared distribution with $\dim(K)$ degrees of freedom. Another case is $\mu_0=0$ and $K$ is a closed convex cone. In this case, the null distribution of $T(Y)$ is the \emph{chi-bar squared distribution}, cf. \cite{bartholomew1961ordered,kudo1963multivariate,barlow1972statistical,kudo1975generalized,shapiro1985asymptotic,robertson1988order}, which can be expressed as a finite mixture of chi-squared distributions. Apart from these special cases, little next to nothing is known about the distribution of the LRS $T(Y)$ for a general pair of $(\mu_0,K)$ under the null $H_0$, owing in large part to the fact that the null distribution of $T(Y)$ highly depends on the exact location of $\mu_0$ with respect to $K$ and is thus intractable in general. Consequently, the lack of such a general description of the limiting distribution of $T(Y)$ causes a fundamental difficulty in obtaining precise characterizations of the power behavior of the LRT for a general pair $(\mu_0,K)$. On the other hand, such generality is of great interest as it allows us to consider several significant examples, for instance testing general signals in isotonic regression and with constrained Lasso. See Section \ref{section:example} for more details.

\subsection{Normal approximation and power characterization}

The unifying theme of this paper starts from the simple observation that in the classical setting where $K$ is a subspace, as long as $\dim(K)$ diverges as $n\rightarrow\infty$, the distribution of $T(Y)$ has a progressively Gaussian shape, under proper normalization. Such a normal approximation in `high dimensions' also holds for  the more complicated chi-bar squared distribution; see \cite{dykstra1991asymptotic,goldstein2017gaussian} for a different set of conditions. One may therefore hope that normal approximation of $T(Y)$ under the null would hold in a far more general context than just these cases. More importantly, such a distributional approximation would ideally form the basis for  power analysis of the LRT.

\subsubsection{Normal approximation}

The first main result of this paper (see Theorem~\ref{thm:lrt_clt_global_testing}) shows that, although the exact distribution of $T(Y)$ under $H_0$ is highly problem-specific and depends crucially on the pair $(\mu_0,K)$ as described above, Gaussian approximation of $T(Y)$ indeed holds in a fairly general context, after proper normalization. More concretely, we show that under $H_0$, 
\begin{align}\label{eq:normal_approx_intro}
\frac{T(Y) - m_{\mu_0}}{\sigma_{\mu_0}} \approx \mathcal{N}(0,1)\quad \hbox{ in total variation}
\end{align}
holds in the high dimensional regime where the estimation error $\E_{\mu_0} \pnorm{\hat{\mu}_K-\mu_0}{}^2$ diverges in some appropriate sense; see Theorem \ref{thm:lrt_clt_global_testing} and the discussion afterwards for an explanation. Here and below, $\mathcal{N}(0,1)$ denotes the standard normal distribution, and we reserve the notation
\begin{align}\label{eq:lrt_mean_var}
m_\mu \equiv \E_\mu (T(Y)) \qquad \text{and}\qquad  \sigma^2_\mu \equiv \var_\mu(T(Y))
\end{align}
for the mean and variance of the LRS $T(Y)$ under (\ref{model:sequence}) with mean $\mu$, so that $m_{\mu_0}$ and $\sigma^2_{\mu_0}$ in (\ref{eq:normal_approx_intro}) are the corresponding quantities under $H_0$. In a similar spirit, we use the subscript $\mu$ in $\Prob_\mu$ and other probabilistic notations to indicate that the evaluation is under (\ref{model:sequence}) with mean $\mu$.

When the normal approximation (\ref{eq:normal_approx_intro}) holds, an asymptotically equivalent formulation of the previously mentioned finite-sample LRT is the following LRT using acceptance region determined by normal quantiles: For any $\alpha \in (0,1)$, let
\begin{align}\label{def:lrt_generic}
\Psi(Y)\equiv \Psi(Y;m_{\mu_0},\sigma_{\mu_0})\equiv \bm{1}\bigg(\frac{T(Y)-m_{\mu_0}}{\sigma_{\mu_0}} \in \mathcal{A}_\alpha^c \bigg),
\end{align}
where $\mathcal{A}_\alpha$ is a possibly unbounded interval in $\R$ such that $\Prob(\mathcal{N}(0,1)\in \mathcal{A}_\alpha)=1-\alpha$. Common choices of $\mathcal{A}_\alpha$ include: (i) $(-\infty, z_\alpha]$ for the one-sided LRT, and (ii) $[-z_{\alpha/2},z_{\alpha/2}]$ for the two-sided LRT, where $z_\alpha$, for any $\alpha\in(0,1)$, is the normal quantile defined by $\Prob(\mathcal{N}(0,1)\geq z_\alpha) = \alpha$. Although $m_{\mu_0},\sigma_{\mu_0}$ do not admit general explicit formulae (some notable exceptions can be found in e.g. \cite[Table 6.1]{mccoy2014from} or \cite[Table 1]{goldstein2017gaussian}), their numeric values can be approximated by simulations. In what follows, we will focus on the LRT given by (\ref{def:lrt_generic}), and in particular its power behavior, when the normal approximation in (\ref{eq:normal_approx_intro}) holds.

\subsubsection{Power characterization}
Using the normal approximation (\ref{eq:normal_approx_intro}), our second main result (see Theorem \ref{thm:power_lrt_global_testing}) shows that under mild regularity conditions,
\begin{align}\label{eq:power_heuristic_intro}
 \E_\mu \Psi(Y;m_{\mu_0},\sigma_{\mu_0})\approx \Prob\bigg[\mathcal{N}\bigg(\frac{ m_\mu - m_{\mu_0}}{\sigma_{\mu_0}},1\bigg) \in \mathcal{A}_\alpha^c \bigg].
\end{align}
This power formula implies that for a wide class of alternatives, the LRS $T(Y)$ still has an asymptotically Gaussian shape under the alternative, but with a mean shift parameter $(m_\mu - m_{\mu_0})/\sigma_{\mu_0}$. In particular, (\ref{eq:power_heuristic_intro}) implies that
\begin{align}\label{eq:power_expansion_intro}
\mathcal{L}\bigg(\bigg\{\frac{ m_\mu - m_{\mu_0}}{\sigma_{\mu_0}}\bigg\}\bigg) \subset \Delta_{\mathcal{A}_\alpha}^{-1}(\beta)  \;\;
\Leftrightarrow \;\; \E_\mu \Psi(Y;m_{\mu_0},\sigma_{\mu_0})\to \beta \in [0,1].
\end{align}
Here $\overline{\R}\equiv \R \cup \{\pm \infty\}$, and for a sequence $\{w_n\}\subset \R$, $\mathcal{L}(\{w_n\})$ denotes the set of all limit points of $\{w_n\}$ in $\overline{\R}$, and the power function $\Delta_{\mathcal{A}_\alpha}:\overline{\R}\to [0,1]$ is defined in (\ref{def:bar_psi}) below. For instance, when $\mathcal{A}_\alpha=(-\infty,z_\alpha]$ is the acceptance region for the one-sided LRT, $\Delta_{(-\infty,z_\alpha]}(w)=\Phi(-z_\alpha+w)$. In general, $\Delta_{\mathcal{A}_\alpha}(0) = \alpha$ and $\Delta_{\mathcal{A}_\alpha}(w) = 1$ only if $w\in\{\pm\infty\}$. Hence the LRT is power consistent under $\mu$, i.e., $\E_\mu \Psi(Y;m_{\mu_0},\sigma_{\mu_0})\rightarrow 1$,  if and only if 
\begin{align}\label{eq:power_generic_intro}
\mathcal{L}\bigg(\bigg\{\frac{ m_\mu - m_{\mu_0}}{\sigma_{\mu_0}} \bigg\}\bigg) \subset \Delta_{\mathcal{A}_\alpha}^{-1}(1) \subset \{\pm \infty\}.
\end{align}
The asymptotically exact power characterization (\ref{eq:power_expansion_intro}) for the LRT is rather rare beyond the classical parametric and certain semiparametric settings under contiguous alternatives (cf.~\cite{van2000asymptotic,van2002semiparametric}). The setting in (\ref{eq:power_expansion_intro}) can therefore be viewed as a general nonparametric analogue of contiguous alternatives for the LRT in the Gaussian sequence model (\ref{model:sequence}).

A notable implication of (\ref{eq:power_generic_intro}) is that for any alternative $\mu\in K$, the power characterization of the LRT depends on the quantity $m_\mu - m_{\mu_0}$, which cannot in general be equivalently reduced to a usual lower bound condition on $\pnorm{\mu-\mu_0}{}$. This indicates the non-uniform power behavior of the LRT with respect to the Euclidean norm $\pnorm{\cdot}{}$. As the LRT (with an optimal calibration) is known to be minimax optimal in terms of uniform separation under $\pnorm{\cdot}{}$ in several examples (cf.~\cite{wei2019geometry}), the non-uniform characterization (\ref{eq:power_generic_intro}) hints that the minimax optimality criteria can be too conservative and non-informative for evaluating the power behavior of the LRT.

Another implication of (\ref{eq:power_generic_intro}) is that it is possible in certain cases that the one-sided LRT (i.e., $\mathcal{A}_\alpha = (-\infty, z_\alpha]$) has an asymptotically vanishing power, whereas the two-sided LRT (i.e., $\mathcal{A}_\alpha = [-z_{\alpha/2}, z_{\alpha/2}]$) is power consistent. This phenomenon occurs when the limit point $-\infty$ in (\ref{eq:power_generic_intro}) is achieved for certain alternatives $\mu \in K$ in the high dimensional limit. See Remark \ref{rmk:one_two_sided_LRT} ahead for a detailed discussion.

\subsection{Testing subspace versus closed convex cone}\label{subsec:intro_subspace}
A particularly important special setting for (\ref{eq:testing_generic}) is the case of testing $H_0: \mu = 0$ versus $H_1: \mu \in K$, where $K$ is assumed to be a closed convex cone in $\R^n$. We perform a detailed case study of the following slightly more general testing problem: 
\begin{align}\label{eq:testing_special}
H_0: \mu \in K_0\qquad \textrm{versus}\qquad H_1: \mu \in K,
\end{align}
where $K_0\subset K\subset \R^n$ is a subspace, and $K$ is a closed convex cone. The primary motivation to study~\eqref{eq:testing_special} arises from the problem of testing a global polynomial structure versus its shape-constrained generalization; concrete examples include constancy versus monotonicity, linearity versus convexity, etc.; see Section \ref{subsec:example_subspace} for details. The LRS for (\ref{eq:testing_special}) takes the slightly modified form:
\begin{align}\label{def:TY_subspace}
T(Y)&\equiv T_{K_0}(Y)\equiv \pnorm{Y-\hat{\mu}_{K_0}}{}^2 - \pnorm{Y-\hat{\mu}_K}{}^2 \nonumber \\
&= \pnorm{\mu+\xi-\Pi_{K_0}(\mu+\xi)}{}^2 - \pnorm{\mu+\xi-\Pi_K(\mu+\xi)}{}^2.
\end{align}
The dependence in the notation of the LRS $T(Y)$ on $K_0$ will usually be suppressed when no confusion could arise.

Specializing our first main result to this testing problem, we show in Theorem \ref{thm:lrt_clt_pivotal} that normal approximation of $T(Y)$ under $H_0$ holds essentially under the minimal growth condition that $\delta_K - \dim(K_0)\to \infty$, where $\delta_K$ is the statistical dimension of $K$ (formally defined in Definition \ref{def:stat_dim}). Similar to (\ref{eq:power_expansion_intro}), the normal approximation makes possible the following precise characterization of the power behavior of the LRT under the prescribed growth condition (see Theorem~\ref{thm:power_lrt_subspace}):
\begin{align}\label{eq:power_subspace_intro_0}
&\mathcal{L}\bigg(\bigg\{\frac{ \E \pnorm{\Pi_K\big(\mu-\Pi_{K_0}(\mu)+\xi\big)}{}^2-\E \pnorm{\Pi_K(\xi)}{}^2}{\sigma_0}\bigg\}\bigg)\subset \Delta_{\mathcal{A}_\alpha}^{-1}(\beta) \cap [0,+\infty]\nonumber\\
& \Leftrightarrow \quad  \E_\mu \Psi(Y;m_{0},\sigma_{0})\to \beta \in [0,1].
\end{align}
As $\sigma_0^2=\var(T(\xi))\asymp \delta_K - \dim(K_0)$ (cf. Lemma \ref{lem:var_proj}) for the modified LRS $T(Y)$ in (\ref{def:TY_subspace}), the LRT is power consistent under $\mu$ if and only if 
\begin{align}\label{eq:power_subspace_intro}
\frac{ \E \pnorm{\Pi_K\big(\mu-\Pi_{K_0}(\mu)+\xi\big)}{}^2-\E \pnorm{\Pi_K(\xi)}{}^2 }{\big(\delta_K-\dim(K_0)\big)^{1/2}}\to +\infty.
\end{align}
Formula (\ref{eq:power_subspace_intro}) shows that power consistency of the LRT is determined completely by (a complicated expression involving) the `distance' of the alternative $\mu \in K$ to its projection onto $K_0$ in the problem (\ref{eq:testing_special}).
Compared to the uniform $\|\cdot\|$-separation rate derived in the recent work~\cite{wei2019geometry} (cf.~(\ref{cond:wwg}) below),  (\ref{eq:power_subspace_intro_0})-(\ref{eq:power_subspace_intro}) provide asymptotically precise power characterizations of the LRT for a sequence of point alternatives. This difference is indeed crucial as (\ref{eq:power_subspace_intro}), similar to (\ref{eq:power_generic_intro}), cannot be equivalently inverted into a lower bound on $\pnorm{\mu-\Pi_{K_0}(\mu)}{}$ alone. This illustrates that the non-uniform power behavior of the LRT is not an aberration in certain artificial testing problems, but is rather a fundamental property of the LRT in the high dimensional regime that already appears in the special yet important setting of testing subspace versus a cone.

\subsection{Examples}\label{subsec:intro_example}

As an illustration of the scope of our theoretical results, we validate the normal approximation of the LRT and exemplify its power behavior in two classes of problems:
\begin{enumerate}
	\item Testing in orthant/circular cone, isotonic regression and Lasso;
	\item Testing parametric assumptions versus shape-constrained alternatives, e.g.,~constancy versus monotonicity, linearity versus convexity, and generalizations thereof.
\end{enumerate}

\subsubsection{Non-uniform power of the LRT}\label{subsec:non_uniform_lrt}
Some of the above problems give clear examples of the aforementioned non-uniform power behavior of the LRT: In the problem of testing $\mu = 0$ versus the orthant and (product) circular cone, the LRT is indeed powerful against most alternatives in the region where the uniform separation in $\pnorm{\cdot}{}$ is not informative. More concretely:
\begin{itemize}
	\item In the case of the orthant cone, the LRT is known to be minimax optimal (cf.~\cite{wei2019geometry}) in terms of the uniform $\|\cdot\|$-separation of the order $n^{1/4}$. Our results show that the LRT is actually powerful for `most' alternatives $\mu$ with $\pnorm{\mu}{}=\mathcal{O}(n^{1/4})$, including some with $\|\cdot\|$-separation of the order $n^{\delta}$ for any $\delta > 0$. This showcases the conservative nature of the minimax optimality criteria. See Section \ref{subsec:example_orthant} for details.
	\item In the case of (product) circular cone, the LRT is known to be minimax sub-optimal (cf.~\cite{wei2019geometry}) with $\|\cdot\|$-separation of the order $n^{1/4}$ while the minimax separation rate is of the constant order. Our results show the minimax sub-optimality is witnessed only by a few unfortunate alternatives and the LRT is powerful within a large cylindrical set including many points of constant $\|\cdot\|$-separation order. This also identifies the minimax framework as too pessimistic for the sub-optimality results of the LRT; see Section \ref{subsec:example_circular} for details.
\end{itemize}

\subsection{Related literature}
The results in this paper are related to the vast literature on nonparametric testing in the Gaussian sequence model, or more general Gaussian models, under a minimax framework. We refer the readers to the monographs \cite{ingster2003nonparametric,gine2015mathematical} for a comprehensive treatment on this topic, and \cite{baraud2002nonasymptotic,juditsky2002nonparametric,donoho2004higher,baraud2005testing,ingster2010detection,arias2011global,verzelen2012minimax,comminges2013minimax,carpentier2015testing,collier2017minimax,carpentier2019minimax,carpentier2019optimal,carpentier2020estimation,mukherjee2020minimax} and references therein for some recent papers on a variety of testing problems. Many results in these references establish minimax separation rates under a pre-specified metric, with the Euclidean metric $\pnorm{\cdot}{}$ being a popular choice in the Gaussian sequence model. In particular, for the testing problem (\ref{eq:testing_special}), this minimax approach with $\pnorm{\cdot}{}$ metric is adopted in the recent work \cite{wei2019geometry}, which derived minimax lower bounds for the separation rates and the uniform separation rate of the LRT. These results show that the LRT is minimax rate-optimal in a number of examples, while being sub-optimal in some other examples. 

Our results are of a rather different flavor and give a precise distributional description of the LRT. Such a description is made possible by the central limit theorems of the LRS under the null proved in Theorems \ref{thm:lrt_clt_global_testing} and \ref{thm:lrt_clt_pivotal}. It also allows us to make two significant further steps beyond the work \cite{wei2019geometry}:
\begin{enumerate}
	\item For the testing problem (\ref{eq:testing_special}), we provide asymptotically \emph{exact} power formula of the LRT in (\ref{eq:power_subspace_intro_0}) for \emph{each and every alternative}, as opposed to \emph{lower bounds} for the uniform separation rates of the LRT in $\pnorm{\cdot}{}$ as in \cite{wei2019geometry}. As a result, the main results for the separation rates of the LRT in \cite{wei2019geometry} follow as a corollary of our main results (see Corollary \ref{cor:wwg} for a formal statement). 
	\item Our theory applies to the general testing problem (\ref{eq:testing_generic}) which allows for a general pair $(\mu_0,K)$ without a cone structure. This level of generality goes much beyond the scope of \cite{wei2019geometry} and covers several significant examples, including testing in isotonic regression and Lasso problems. 
\end{enumerate}
The precise power characterization we derive in (\ref{eq:power_subspace_intro_0}) has interesting implications when compared to the minimax results derived in \cite{wei2019geometry}. In particular, as discussed in Section \ref{subsec:non_uniform_lrt}, (i) it is possible that the LRT beats substantially the minimax separation rates in $\pnorm{\cdot}{}$ metric for individual alternatives, and (ii) the sub-optimality of the LRT in \cite{wei2019geometry} is actually only witnessed at alternatives along some `bad' directions. In this sense, our results not only give precise understanding of the power behavior of the canonical LRT in this testing problem, but also highlights some intrinsic limitations of the popular minimax framework under the Euclidean metric in the Gaussian sequence model, both in terms of its optimality and sub-optimality criteria.

From a technical perspective, our proof technique differs significantly from the one adopted in \cite{wei2019geometry}. Indeed, the proofs of the central limit theorems and the precise power formulae in this paper are inspired by the second-order Poincar\'e inequality due to \cite{chatterjee2009fluctuations} and related normal approximation results in \cite{goldstein2017gaussian}. These technical developments are of independent interest and have broader applicability; see for instance \cite{han2021general} for further developments in the context of testing high dimensional covariance matrices.

\subsection{Organization}

The rest of the paper is organized as follows. Section \ref{subsec:preliminary} reviews some basic facts on metric projection and conic geometry. Section \ref{section:clt_LRT} studies normal approximation for the LRS $T(Y)$ and the power characterizations of the LRT both in the general setting (\ref{eq:testing_generic}) and the more structured setting (\ref{eq:testing_special}). Applications of the abstract theory to the examples mentioned above are detailed in Section \ref{section:example}. Proofs are collected in Sections \ref{section:proof_clt} and \ref{section:proof_example} and the appendix.

\subsection{Notation}\label{section:notation}

For any positive integer $n$, let $[1:n]$ denote the set $\{1,\ldots,n\}$. For $a,b \in \R$, $a\vee b\equiv \max\{a,b\}$ and $a\wedge b\equiv\min\{a,b\}$. For $a \in \R$, let $a_\pm \equiv (\pm a)\vee 0$. For $x \in \R^n$, let $\pnorm{x}{p}$ denote its $p$-norm $(0\leq p\leq \infty)$, and $B_p(r;x)\equiv \{z \in \R^n: \pnorm{z-x}{p}\leq r\}$. We simply write $\pnorm{x}{}\equiv\pnorm{x}{2}$, $B(r;x)\equiv B_2(r;x)$, and $B(r) \equiv B(r;0)$ for notational convenience. By $\bm{1}_n$ we denote the vector of all ones in $\R^n$. For a matrix $M \in \R^{n\times n}$, let $\pnorm{M}{}$ and $\pnorm{M}{F}$ denote the spectral and Frobenius norms of $M$ respectively.

For a multi-index $\bm{k} = (k_1,\ldots,k_n)\in \mathbb{Z}_{\geq 0}^n$, let $\abs{\bm{k}} \equiv \sum_{i=1}^n k_i$. For $f:\R^n\rightarrow\R$, and $\bm{k} = (k_1,\ldots,k_n)\in \mathbb{Z}_{\geq 0}^n$, let $\partial_{\bm{k}} f (z) \equiv \frac{\partial^{\abs{\bm{k}} } f(z)}{\partial_{k_1} z_1 \cdots\partial_{k_n} z_n}$ for $z \in \R^n$ whenever definable. A vector-valued map $f:\R^n\rightarrow\R^m$ is said to have sub-exponential growth at $\infty$ if $\lim_{\pnorm{x}{}\to \infty} \pnorm{f(x)e^{-\pnorm{x}{}}}{}=0$. For $f=(f_1,\ldots,f_n):\R^n \to \R^n$, let $J_f(z) \equiv (\partial f_i(z)/\partial z_j)_{i,j=1}^n$ denote the Jacobian of $f$ and 
\begin{align*}
\dv f(z) \equiv \sum_{i=1}^n \frac{\partial}{\partial z_i} f_i (z)=\tr(J_f(z))
\end{align*}
for $z \in \R^n$ whenever definable.

We use $C_{x}$ to denote a generic constant that depends only on $x$, whose numeric value may change from line to line unless otherwise specified. $a\lesssim_{x} b$ and $a\gtrsim_x b$ mean $a\leq C_x b$ and $a\geq C_x b$ respectively, and $a\asymp_x b$ means $a\lesssim_{x} b$ and $a\gtrsim_x b$ ($a\lesssim b$ means $a\leq Cb$ for some absolute constant $C$). For two nonnegative sequences $\{a_n\}$ and $\{b_n\}$, we write $a_n\ll b_n$ (respectively~$a_n\gg b_n$) if $\lim_{n\rightarrow\infty} (a_n/b_n) = 0$ (respectively~$\lim_{n\rightarrow\infty} (a_n/b_n) = \infty$). We follow the convention that $0/0 = 0$. $\mathcal{O}_{\mathbf{P}}$ and $\mathfrak{o}_{\mathbf{P}}$ denote the usual big and small O notation in probability. 

We reserve the notation $\xi=(\xi_1,\ldots,\xi_n)$ for an $n$-dimensional standard normal random vector, and $\varphi,\Phi$ for the density and the cumulative distribution function of a standard normal random variable. For any $\alpha \in (0,1)$, let $z_\alpha$ be the normal quantile defined by $\Prob(\mathcal{N}(0,1)\geq z_\alpha) = \alpha$. For two random variables $X,Y$ on $\R$, we use $d_{\mathrm{TV}}(X,Y)$ and $d_{\mathrm{Kol}}(X,Y)$  to denote their total variation distance and Kolmogorov distance defined respectively as
\begin{align*}
d_{\mathrm{TV}}(X,Y)&\equiv \sup_{B \in \mathcal{B}(\R)}\bigabs{\Prob\big(X \in B\big)-\Prob\big(Y \in B\big)},\\
d_{\mathrm{Kol}}(X,Y)&\equiv \sup_{t \in \R}\bigabs{\Prob\big(X \leq t\big)-\Prob\big(Y \leq t\big)}.
\end{align*}
Here $\mathcal{B}(\R)$ denotes all Borel measurable sets in $\R$.

\section{Preliminaries: metric projection and conic geometry}\label{subsec:preliminary}

In this section, we review some basic facts on metric projection and conic geometry. For any $x\in\R^n$, the metric projection of $x$ onto a closed convex set $K\subset \R^n$ is defined by
\begin{align*}
\Pi_K(x) \equiv  \argmin\limits_{y\in K} \pnorm{x-y}{}^2.
\end{align*}
It is a standard fact that the map $\Pi_K$ is well-defined, $1$-Lipschitz and hence absolutely continuous. The Jacobian $J_{\Pi_K}$ is therefore almost everywhere (a.e.) well-defined.

Let $G:\R^n \rightarrow\R$ be defined by 
\begin{align*}
G(y)\equiv \pnorm{y-\Pi_K(y)}{}^2.
\end{align*}
We summarize some useful properties of $G$ and $J_{\Pi_K}$ in the following lemma. 
\begin{lemma}\label{lem:proj_basic}
	The following statements hold.
	\begin{enumerate}
		\item $G$ is absolutely continuous and its gradient $\nabla G(y) = 2(y-\Pi_K(y))$ has sub-exponential growth at $\infty$.
		\item For a.e. $y\in\R^n$, $\pnorm{J_{\Pi_K }(y) }{}\vee \pnorm{I-J_{\Pi_K }(y)}{}\leq 1$ and $J_{\Pi_K}(y)^\top \Pi_K(y)= J_{\Pi_K}(y)^\top y$.
	\end{enumerate}
\end{lemma}
\begin{proof}
	(1) follows from~\cite[Lemma 2.2]{goldstein2017gaussian} and the proof of \cite[Lemma A.2]{goldstein2017gaussian}. The first claim of (2) is proved in \cite[Lemma 2.1]{goldstein2017gaussian}. For the second claim of (2), note that $\nabla G(y) = 2(I-J_{\Pi_K}(y))^\top (y-\Pi_K(y))$. By (1), $\nabla G(y) = 2(y-\Pi_K(y))$, so $J_{\Pi_K}(y)^\top(y-\Pi_K(y))=0$, proving the claim.
\end{proof}

Recall that a closed and convex cone $K\subset \R^n$ is polyhedral if it is a finite intersection of closed half-spaces, and a face of $K$ is a set of the form $K\cap H$, where $H$ is a supporting hyperplane of $K$ in $\R^n$. Let $\text{lin}(F)$ denote the linear span of $F$. The dimension of a face $F$ is $\dim F \equiv \dim(\text{lin}(F))$, and the relative interior of $F$ is the interior of $F$ in $\text{lin}(F)$.

The complexity of a closed convex cone $K$ can be described by its statistical dimension defined as follows.

\begin{definition}\label{def:stat_dim}
	The \emph{statistical dimension} $\delta_{K}$ of a closed convex cone $K$ is defined as $\delta_K\equiv \E \pnorm{\Pi_K(\xi)}{}^2$. 
\end{definition}

The statistical dimension $\delta_K$ has several equivalent definitions; see e.g. \cite[Proposition 3.1]{amelunxen2014living}. In particular, $\delta_K = \E \sup_{\nu \in K\cap B(1)}(\iprod{\nu}{\xi})^2$. For any polyhedral cone $K\subset\R^n$ and $ j \in \{0,\ldots,n\}$, the $j$-th intrinsic volume of $K$ is defined as
\begin{align}\label{def:v_j}
&v_j(K) \equiv \Prob\big(\Pi_{K}(\xi) \in \hbox{relative interior of a $j$-dimensional face of $K$}\big).
\end{align}
More generally, the intrinsic volumes $\{v_j(K)\}_{j=0}^n$ for a closed convex cone $K\subset\R^n$ are defined as the limit of (\ref{def:v_j}) using polyhedral approximation; see \cite[Section 7.3]{mccoy2014from}. These quantities are well-defined and have been investigated in considerable depth; see e.g., \cite{amelunxen2014living,mccoy2014from,goldstein2017gaussian}.  

\begin{definition}\label{def:V_K}
	For any closed convex cone $K\subset\R^n$, let $V_K$ be a random variable taking values in $\{0,\ldots,n\}$ such that $\Prob(V_K = j) = v_j(K)$. 
\end{definition}

We summarize some useful properties of $\delta_K$ and $V_K$ in the following lemma. An elementary and self-contained proof is given in Appendix \ref{section:proof_preliminary}.
\begin{lemma}\label{lem:var_proj}
	Let $K$ be a convex closed cone. Then 
	\begin{enumerate}
		\item  $\delta_K = \E V_K$; 
		\item $\var(\pnorm{\Pi_K(\xi)}{}^2) =  \var(V_K)+2\delta_K$;
		\item $2\delta_K\leq \var( \pnorm{\Pi_K(\xi)}{}^2)\leq 2\delta_K+2 \pnorm{\E \Pi_K(\xi)}{}^2\leq 4\delta_K$.
	\end{enumerate}
\end{lemma}

For any closed convex cone $K\subset\R^n$, its polar cone is defined as
\begin{align}\label{def:polar_cone}
K^*\equiv \left\{v\in\R^n: \iprod{v}{u}\leq 0, \text{ for all }u\in K\right\}.
\end{align}
With $\Pi_{K^*}$ denoting the metric projection onto $K^*$, Moreau's theorem \cite[Theorem 31.5]{rockafellar1997convex} states that for any $v\in\R^n$,
\begin{align*}
v = \Pi_K(v) + \Pi_{K^*}(v) \qquad \hbox{ with } \iprod{\Pi_K(v)}{\Pi_{K^*}(v)} = 0.
\end{align*}

\section{Theory}\label{section:clt_LRT}

\subsection{Normal approximation for $T(Y)$ and power characterizations}

We start by presenting the normal approximation result for $T(Y)$ in (\ref{def:TY}) under the null hypothesis~\eqref{eq:testing_generic}; see Section~\ref{pf:lrt_clt_global_testing} for a proof. This will serve as the basis for the size and power analysis of the LRT (\ref{def:lrt_generic}) in the testing problem (\ref{eq:testing_generic}).

\begin{theorem}\label{thm:lrt_clt_global_testing}
	Let $K\subset\R^n$ be a closed convex set and $\mu_0 \in K$. Then under $H_0$,
	\begin{align}\label{ineq:lrt_clt_global_testing}
	d_{\mathrm{TV}}\bigg( \frac{T(Y)-m_{\mu_0}}{ \sigma_{\mu_0} },\mathcal{N}(0,1)\bigg)\leq \frac{8 \sqrt{\E_{\mu_0}\pnorm{\hat{\mu}_{K}-\mu_0}{}^2}}{2\pnorm{\E_{\mu_0}\hat{\mu}_K-\mu_0}{}^2+ \pnorm{\E_{\mu_0} J_{\hat{\mu}_K }}{F}^2 }.
	\end{align}
	Here $J_{\hat{\mu}_K}\equiv J_{\hat{\mu}_K}(\xi)\equiv J_{\Pi_K}(\mu_0+\xi)$, and $m_{\mu_0}, \sigma_{\mu_0}$ are as defined in (\ref{eq:lrt_mean_var}).
\end{theorem}

The bound (\ref{ineq:lrt_clt_global_testing}) is obtained by a generalization of \cite[Theorem 2.1]{goldstein2017gaussian} using the second-order Poincar\'e inequality \cite{chatterjee2009fluctuations}, together with a lower bound for $\sigma_{\mu_0}^2$ using Fourier analysis in the Gaussian space \cite[Section 1.5]{nourdin2012normal}. The Fourier expansion can be performed up to the second order thanks to the absolute continuity of the first-order partial derivatives of  $T(Y)$ (cf.~Lemma \ref{lem:proj_basic}).

We now comment on the structure of (\ref{ineq:lrt_clt_global_testing}). The first term $\pnorm{\E_{\mu_0}\hat{\mu}_K-\mu_0}{}^2$ in the denominator is the squared bias of the projection estimator $\hat{\mu}_K$, while the second term $\pnorm{\E_{\mu_0} J_{\hat{\mu}_K }}{F}^2$, which depends on the magnitudes of the first-order partial derivatives of $\hat{\mu}_K$, can be roughly understood as the `variance' of $\hat{\mu}_K$. Consequently, one may expect that the denominator is of the order $\E_{\mu_0}\pnorm{\hat{\mu}_{K}-\mu_0}{}^2$, so the overall bound scales as $\mathcal{O}\big(1/\sqrt{\E_{\mu_0}\pnorm{\hat{\mu}_{K}-\mu_0}{}^2}\big)$. As will be clear in Section \ref{section:example}, this is indeed the case in all the examples worked out, and the major step in applying (\ref{ineq:lrt_clt_global_testing}) to concrete problems typically depends on obtaining sharp lower bounds for the `variance' term $\pnorm{\E_{\mu_0} J_{\hat{\mu}_K }}{F}^2$, which may require non-trivial problem-specific techniques.

Using Theorem \ref{thm:lrt_clt_global_testing}, we may characterize the size and power behavior of the LRT. For a possibly unbounded interval $I \subset \R$, let $\Delta_I : \overline{\R}\to [0,1]$ be defined as follows: For $w \in \R$,
\begin{align}\label{def:bar_psi}
\Delta_I(w) \equiv 1-\Prob\big(\mathcal{N}(0,1) \in I-w\big)=\Prob\big(\mathcal{N}(0,1) \in I^c-w\big),
\end{align}
and $\Delta_{I}(\pm \infty)\equiv \lim_{w \to \pm \infty}\Delta_I(w)$, which is clearly well-defined. $\Delta_I$ is either monotonic or unimodal, so $\Delta_I^{-1}(\beta)$ contains at most two elements for any $\beta \in [0,1]$. Two primary examples of $I$ are $\mathcal{A}^{\textrm{os}}_\alpha \equiv (-\infty,z_\alpha]$ and $\mathcal{A}^{\textrm{ts}}_\alpha \equiv [-z_{\alpha/2},z_{\alpha/2}]$ --- the acceptance regions for the one- and two-sided LRTs respectively, where we have
\begin{align}\label{eq:delta_example}
\Delta_{\mathcal{A}^{\textrm{os}}_\alpha}(w) = \Phi(-z_{\alpha}+w),\;\; \Delta_{\mathcal{A}^{\textrm{ts}}_\alpha}(w) = \Phi(-z_{\alpha/2}+w) + \Phi(-z_{\alpha/2} - w). 
\end{align}
It is clear that $\Delta_{\mathcal{A}^{\textrm{os}}_\alpha}(0)  = \Delta_{\mathcal{A}^{\textrm{ts}}_\alpha}(0)  = \alpha$, $\Delta^{-1}_{\mathcal{A}^{\textrm{os}}_\alpha}(1) = \{+\infty\}$, and $\Delta^{-1}_{\mathcal{A}^{\textrm{ts}}_\alpha}(1) = \{\pm\infty\}$. Recall the definitions of $m_\mu$ and $\sigma^2_\mu$ for general $\mu\in K$ in (\ref{eq:lrt_mean_var}). The following result (see Section~\ref{section:proof_master_thm} for a proof) characterizes the power behavior of the LRT.

\begin{theorem}\label{thm:power_lrt_global_testing}
	Consider testing (\ref{eq:testing_generic}) using the LRT as in~\eqref{def:lrt_generic}. There exists some constant $C_{ \mathcal{A}_\alpha}>0$ such that
	\begin{align}\label{ineq:power_lrt_global_testing_expansion}
	\notag&\biggabs{ \E_\mu\Psi(Y;m_{\mu_0},\sigma_{\mu_0})- \Delta_{\mathcal{A}_\alpha}\bigg(\frac{m_\mu-m_{\mu_0}}{\sigma_{\mu_0}}\bigg)} \\
	&\qquad \leq 2\cdot \mathrm{err}_{\mu_0}+C_{\mathcal{A}_\alpha } \cdot  \mathscr{L}\bigg( 1\bigwedge\frac{\pnorm{\mu-\mu_0}{}}{  \abs{m_\mu-m_{\mu_0}} \vee \sigma_{\mu_0} }   \bigg).
	\end{align}
	Here 
	\begin{align*}
	\mathrm{err}_{\mu_0} \equiv d_{\mathrm{Kol}}\bigg( \frac{T(\mu_0+\xi)-m_{\mu_0}}{ \sigma_{\mu_0} },\mathcal{N}(0,1)\bigg)\leq \hbox{right hand side of (\ref{ineq:lrt_clt_global_testing})},
	\end{align*}
	and $\mathscr{L}(x) \equiv x\sqrt{1\vee \log(1/x)} $ for $x > 0$ and $\mathscr{L}(0)\equiv 0$. Consequently:	
	\begin{enumerate}
		\item The LRT in~\eqref{def:lrt_generic} has size
		\begin{align*}
		\bigabs{\E_{\mu_0} \Psi(Y;m_{\mu_0},\sigma_{\mu_0})-\alpha}\leq 2\cdot\mathrm{err}_{\mu_0}.
		\end{align*}
		\item 
		Suppose the normal approximation of $T(Y)$ holds under $H_0$, i.e., $\mathrm{err}_{\mu_0}\to 0$.  Then, for any $\mu\in K$ such that 
		\begin{align}\label{cond:power_lrt_global_testing_0}
		\pnorm{\mu-\mu_0}{}\ll \abs{ m_\mu - m_{\mu_0}} \vee \sigma_{\mu_0},
		\end{align}
		we have
		\begin{align}\label{eq:power_lrt_global_testing}
		\mathcal{L}\bigg(\bigg\{\frac{ m_\mu - m_{\mu_0}}{\sigma_{\mu_0}}\bigg\}\bigg) \subset \Delta_{\mathcal{A}_\alpha}^{-1}(\beta) \; \Leftrightarrow \; \E_\mu \Psi(Y;m_{\mu_0},\sigma_{\mu_0})\to \beta \in [0,1].
		\end{align}
		\sloppy Hence under (\ref{cond:power_lrt_global_testing_0}), the LRT is power consistent under $\mu$, i.e., $\E_\mu \Psi(Y;m_{\mu_0},\sigma_{\mu_0})\to 1$, if and only if 
		\begin{align}\label{cond:power_lrt_global_testing_1}
		\mathcal{L}\bigg(\bigg\{\frac{ m_\mu - m_{\mu_0}}{\sigma_{\mu_0}}\bigg\}\bigg) \subset \Delta_{\mathcal{A}_\alpha}^{-1}(1) \subset \{\pm \infty\}.
		\end{align}
	\end{enumerate}
\end{theorem}

\begin{remark}\label{remark:power_normal_cond}
	The validity of the normal approximation in Theorem \ref{thm:power_lrt_global_testing}-(2) is imposed to express the exact power behavior (\ref{eq:power_lrt_global_testing}) with the normal quantile. More generally, as long as the normalized LRS $\big(T(Y)-m_{\mu_0}\big)/\sigma_{\mu_0}$ has a distributional limit under $H_0$, (\ref{eq:power_lrt_global_testing}) can be obtained accordingly with the corresponding quantiles. 
\end{remark}

\begin{remark}
Some comments on conditions (\ref{cond:power_lrt_global_testing_0}) and (\ref{cond:power_lrt_global_testing_1}) in detail. 
\begin{enumerate}
	\item
	Condition (\ref{cond:power_lrt_global_testing_0}) centers around the key deviation quantity
	\begin{align}\label{def:deviation}
	\Delta T_{\mu,\mu_0}(\xi)\equiv T(\mu+\xi) - T(\mu_0 + \xi),
	\end{align}
	which can be shown to satisfy
	\begin{align*}
	\E(\Delta T_{\mu,\mu_0}) = m_\mu - m_{\mu_0}, \qquad \var(\Delta T_{\mu,\mu_0}) \leq \|\mu - \mu_0\|^2.
	\end{align*} 
	Moreover, it can be shown that $\Delta T_{\mu,\mu_0}$ concentrates around its mean $m_\mu - m_{\mu_0}$ with sub-Gaussian tails (see Proposition \ref{prop:concentration_mean_shift}). This concentration result allows us to connect the normal approximation under the null in Theorem \ref{thm:lrt_clt_global_testing} to the power behavior of the LRT under the alternative.
	
	\item The condition (\ref{cond:power_lrt_global_testing_0}) cannot be removed in general for the validity of the power characterization (\ref{eq:power_lrt_global_testing}). In fact, in the small separation regime $\pnorm{\mu-\mu_0}{}\ll \sigma_{\mu_0}$, (\ref{cond:power_lrt_global_testing_0})  is automatically fulfilled; in the large separation regime where $\pnorm{\mu-\mu_0}{}\gg \sigma_{\mu_0}$, (\ref{cond:power_lrt_global_testing_0}) can typically be verified by establishing a \emph{quadratic lower bound} $\abs{m_\mu-m_{\mu_0}}\gtrsim \pnorm{\mu-\mu_0}{}^2$. In this sense (\ref{cond:power_lrt_global_testing_0}) excludes possibly ill-behaved alternatives that violate the prescribed quadratic lower bound in the critical regime $\pnorm{\mu-\mu_0}{} \asymp \sigma_{\mu_0}$. Such ill-behaved alternatives do exist; see e.g., Example \ref{example:necessity_reg_cond} ahead for more details. 

	\item To verify (\ref{cond:power_lrt_global_testing_1}), some problem specific understanding for $m_\mu$ and $\sigma_{\mu_0}$ is needed. As $\E_\mu\iprod{\xi}{\hat{\mu}_K} = \E_\mu\dv \hat{\mu}_K$ by Stein's identity, we have
	\begin{align}\label{def:m_mu}
	m_\mu  = \|\mu - \mu_0\|^2 + 2\E_\mu \dv \hat{\mu}_K- \E_\mu \pnorm{\hat{\mu}_K-\mu}{}^2,
	\end{align}
	hence the numerator of (\ref{cond:power_lrt_global_testing_1}) requires sharp estimates of the expected `degrees of freedom' $\E_\mu \dv \hat{\mu}_K$ (cf.~\cite{meyer2000degrees}), and the estimation error $\E_\mu \pnorm{\hat{\mu}_K-\mu}{}^2$. A (near) matching upper and lower bound for $\sigma_{\mu_0}$ will also be required to obtain necessary and sufficient characterizations. We mention that (\ref{cond:power_lrt_global_testing_1}) cannot in general be equivalently inverted into a lower bound on $\pnorm{\mu-\mu_0}{}$ only; see the remarks after Theorem \ref{thm:power_lrt_subspace} for a more detailed discussion.
\end{enumerate}
\end{remark}

\begin{remark}\label{rmk:one_two_sided_LRT}
	The LRT defined in (\ref{def:lrt_generic}) depends on the choice of the acceptance region $\mathcal{A}_\alpha$.  Two obvious choices are:
	\begin{enumerate}
		\item (\emph{One-sided LRT}).  Let $\mathcal{A}_\alpha\equiv \mathcal{A}^{\textrm{os}}_\alpha = (-\infty, z_\alpha]$. This leads to the following one-sided LRT:
		\begin{align}\label{def:LRT_one_sided}
		\Psi_{\textrm{os}}(Y)\equiv \Psi_{\textrm{os}}(Y;m_{\mu_0},\sigma_{\mu_0})\equiv \bm{1}\bigg(\frac{T(Y)-m_{\mu_0}}{\sigma_{\mu_0}}>z_\alpha\bigg).
		\end{align}
		\item (\emph{Two-sided LRT}). Let $\mathcal{A}_\alpha\equiv \mathcal{A}^{\textrm{ts}}_\alpha = [-z_{\alpha/2},z_{\alpha/2}]$. This leads to the following two-sided LRT:
		\begin{align}\label{def:LRT_two_sided}
		\Psi_{\textrm{ts}}(Y)\equiv \Psi_{\textrm{ts}}(Y;m_{\mu_0},\sigma_{\mu_0})\equiv \bm{1}\bigg(\biggabs{\frac{T(Y)-m_{\mu_0}}{\sigma_{\mu_0}}}>z_{\alpha/2}\bigg).
		\end{align}
	\end{enumerate}
	In the classical case where $K$ is a subspace of fixed dimension, the one-sided LRT is power consistent (under $\mu \in K$) if and only if the two-sided LRT is power consistent, so one can simply use the standard one-sided LRT. The situation can be rather different for certain high dimensional instances of $K$.  Under the setting of Theorem \ref{thm:power_lrt_global_testing}-(2), as  $\Delta_{\mathcal{A}^{\textrm{os}}_\alpha}^{-1}(1) = \{+\infty\}$  while $\Delta_{\mathcal{A}^{\textrm{ts}}_\alpha}^{-1}(1) = \{\pm\infty\}$, power consistency under $\mu$ for the one-sided LRT implies that for the two-sided LRT, but the converse fails when the $-\infty$ limit in (\ref{cond:power_lrt_global_testing_1}) is achieved. See Example \ref{example:LRT_two_sided} ahead for a concrete example. However, in the special case where $\mu_0=0$ and $K$ is a closed convex cone, $(m_\mu-m_{\mu_0})/\sigma_{\mu_0}$ can only diverge to $+\infty$ under mild growth condition on $K$, so in this case power consistency is equivalent for one- and two-sided LRTs. Also see Remark \ref{rmk:subspace_one_sided_LRT}-(1).
\end{remark}

As a simple toy example of Theorem \ref{thm:power_lrt_global_testing}, we consider the testing problem (\ref{eq:testing_generic}) in the linear regression case, where $K\equiv K_X\equiv \{X\theta: \theta \in \R^p\}$ for some fixed design matrix $X \in \R^{n\times p}$, with $p\leq n$. We will be interested in the high dimensional regime $\rank(X)\to \infty$ where the normal approximation for the LRT holds under the null.
\begin{proposition}\label{prop:global_testing_subspace_toy}
	Consider testing (\ref{eq:testing_generic}) with $K=K_X$. Suppose that $\rank(X)\to \infty$. Let $\Psi \in \{\Psi_{\mathrm{os}},\Psi_{\mathrm{ts}}\}$. 
	\begin{enumerate}
		\item If $\mu_0 \in K_X$, then
		\begin{align}\label{ineq:lrt_clt_global_testing_linreg}
		d_{\mathrm{TV}}\bigg( \frac{T(Y)-m_{\mu_0}}{ \sigma_{\mu_0} },\mathcal{N}(0,1)\bigg)&\leq \frac{8}{\sqrt{\rank (X)} }.
		\end{align}
		Consequently the LRT is asymptotically size $\alpha$ with $\E_{\mu_0} \Psi(Y) = \alpha + \mathcal{O}(1/\sqrt{\rank (X)})$.
		\item For any $\mu \in K_X$, $m_\mu - m_{\mu_0} = \|\mu - \mu_0\|^2$, and the LRT is power consistent under $\mu$, i.e., $\E_{\mu} \Psi(Y)\to 1$, if and only if $\pnorm{\mu-\mu_0}{}\gg (\rank(X))^{1/4}$.
	\end{enumerate}	
\end{proposition}
\begin{proof}
	\noindent (1). Note that $\hat{\mu}_{K_X} = \Pi_{K_X}(Y)=X(X^\top X)^{-}X^\top Y \equiv PY$, where $A^-$ denotes the pseudo-inverse for $A$. Then $\E_{\mu_0} \hat{\mu}_{K_X} = P \mu_0 = \mu_0$, $J_{\hat{\mu}_{K_X}} = P$ and
	\begin{align*}
	\E_{\mu_0} \pnorm{\hat{\mu}_{K_X} -\mu_0 }{}^2 = \pnorm{\E_{\mu_0} J_{\hat{\mu}_{K_X}}}{F}^2 &= \tr(PP^\top) = \dim(K_X) =\rank(X). 
	\end{align*}
	The claim (1) now follows from Theorem \ref{thm:lrt_clt_global_testing}.
	
	\noindent (2).	By (\ref{def:m_mu}), for any $\mu \in K_X$,
	\begin{align*}
	m_\mu
	& = \|\mu - \mu_0\|^2 + \E\big[ 2\iprod{\xi}{P\xi}-\pnorm{P\xi}{}^2 \big]\\
	&= \|\mu - \mu_0\|^2 + \E[ \pnorm{P\xi}{}^2 ] =\|\mu - \mu_0\|^2 + \rank(X),
	\end{align*}
	and with $\mu_0 \in K_X$, 
	\begin{align*}
	\sigma_{\mu_0}^2 = \var\big(\pnorm{P\xi}{}^2\big)=\rank(X). 
	\end{align*}
	As $\pnorm{\mu-\mu_0}{}\ll \pnorm{\mu-\mu_0}{}^2 \vee \sqrt{\rank(X)}$ always holds, the claim follows from Theorem \ref{thm:power_lrt_global_testing}-(2).
\end{proof}

More examples on testing in orthant/circular cone, isotonic regression and Lasso are worked out in Section \ref{section:example}.

\subsection{Subspace versus closed convex cone}\label{subsec:space_cone}

In this subsection, we study in detail the testing problem (\ref{eq:testing_special}) as an important special case of (\ref{eq:testing_generic}). The additional subspace and cone structure will allow us to give more explicit characterizations of the size and the power of the LRT; note that here the LRS $T(Y)$ takes the modified form (\ref{def:TY_subspace}). We start with the following simple observation.
\begin{lemma}\label{lem:invariance_lrt}
	Let $K$ be a closed convex set in $\R^n$. Then for $\mu$ such that $K-\mu \subset K$, we have
	\begin{align*}
	\Pi_K(\mu+\xi) = \mu+\Pi_K(\xi),\quad \forall \xi \in \R^n.
	\end{align*}
	Consequently,
	\begin{align*}
	\pnorm{\mu+\xi-\Pi_K(\mu+\xi)}{}^2 = \pnorm{\xi-\Pi_K(\xi)}{}^2.
	\end{align*}
\end{lemma}
\begin{proof}
	By the definition of projection, we want to verify
	\begin{align*}
	\iprod{\mu+\xi-(\mu+\Pi_K(\xi))}{\nu-(\mu+\Pi_K(\xi))}\leq 0,\quad \forall \nu \; \in K.
	\end{align*}
	This amounts to verifying that
	\begin{align*}
	\iprod{\xi-\Pi_K(\xi)}{(\nu-\mu)-\Pi_K(\xi)}\leq 0, \quad \forall \nu \in K.
	\end{align*}
	As $\nu - \mu \in K$ by the condition $K-\mu \subset K$, the above inequality holds by the projection property for $\Pi_K(\xi)$.
\end{proof}
Recall the definition of the statistical dimension $\delta_K$ in Definition \ref{def:stat_dim}. The above lemma provides us with simplifications of $m_\mu$ and $\sigma^2_\mu$ as defined in (\ref{eq:lrt_mean_var}): under the setting of (\ref{eq:testing_special}), for any $\mu\in K_0$,
\begin{align}\label{eq:lrt_mean_var_special}
m_{\mu} \equiv m_0 = \delta_{K}-\delta_{K_0}, \quad \;\;\sigma_\mu^2 \equiv \sigma_0^2 =  \var\big(\pnorm{\Pi_{K}(\xi)}{}^2-\pnorm{\Pi_{K_0}(\xi)}{}^2\big).
\end{align}
Moreover, as $K_0$ is a subspace, we have $\delta_{K_0} = \dim(K_0)$. The following result (proved in Section~\ref{pf:lrt_clt_pivotal}) derives the normal approximation of $T(Y)$ with an explicit error bound.

\begin{theorem}\label{thm:lrt_clt_pivotal}
	Suppose $K_0\subset K\subset \R^n$ are such that $K_0$ is a subspace and $K$ is a closed convex cone. Then for $\mu \in K_0$,  
	\begin{align*}
	d_{\mathrm{TV}}\bigg( \frac{T(Y)-m_0}{\sigma_0 },\mathcal{N}(0,1)\bigg)
	&\leq \frac{8}{\sqrt{\delta_K-\delta_{K_0}}} .
	\end{align*}
\end{theorem}

It is easy to see from the above bound that under the growth condition $\delta_{K}-\delta_{K_0}\to \infty$, normal approximation of $T(Y)$ holds under the null. This growth condition cannot be improved in general: for a subspace $K$, $T(Y)$ follows a chi-squared distribution with $\delta_K-\delta_{K_0}$ degrees of freedom under the null, so normal approximation holds if and only if $\delta_{K}-\delta_{K_0}\to \infty$. The above theorem extends \cite[Theorem 2.1]{goldstein2017gaussian} in which the case $K_0 = \{0\}$ is treated. Compared to classical results on the chi-bar squared distribution \cite[Corollary 2.2]{dykstra1991asymptotic}, the growth condition here does not require exact knowledge for the mixing weights, and can be easily checked using Gaussian process techniques; see Section \ref{subsec:example_subspace} for examples.

Using Theorem \ref{thm:lrt_clt_pivotal}, we can prove sharp size and power behavior of the LRT; see Theorem~\ref{thm:power_lrt_subspace} below (proved in Section~\ref{pf:power_lrt_subspace}). For $p\geq 1$, let
\begin{align*}
\Gamma_{K,p}(\nu)\equiv \E \pnorm{\Pi_{K}\big(\nu+\xi\big)}{}^p - \E \pnorm{\Pi_{K}(\xi)}{}^p, \quad \nu \in \R^n.
\end{align*}
We simply shorthand $\Gamma_{K,1}$ as $\Gamma_K$ for notational convenience. Recall the definition of $V_K$ in Definition \ref{def:V_K} and that of the polar cone $K^*$ in (\ref{def:polar_cone}).

\begin{theorem}\label{thm:power_lrt_subspace}
	Consider testing (\ref{eq:testing_special}) using the LRT $\Psi(Y;m_0,\sigma_0)$ with the modified LRS $T(Y)$ in (\ref{def:TY_subspace}). There exist constants $C_{\mathcal{A}_\alpha},C_{\mathcal{A}_\alpha}'>0$  such that
	\begin{align}
	\notag&\biggabs{ \E_\mu\Psi(Y;m_0,\sigma_0)-  \Delta_{\mathcal{A}_\alpha}\bigg(\frac{ \Gamma_{K,2}\big(\mu-\Pi_{K_0}(\mu)\big) }{\sigma_0 }\bigg)  }\\
	&\quad \qquad \leq 2\cdot \mathrm{err}_0
	+C_{\mathcal{A}_\alpha} \cdot  \mathscr{L}\bigg( 1\bigwedge\frac{\pnorm{\mu-\Pi_{K_0}(\mu) }{}}{ \bigabs{\Gamma_{K,2}\big(\mu-\Pi_{K_0}(\mu)\big)} \vee \sigma_0  } \bigg) \label{ineq:normal_power_expansion_subspace}\\
	&\quad \qquad \leq C_{\mathcal{A}_\alpha}'\cdot \mathscr{L}\Big( \big(\delta_K-\delta_{K_0}\big)^{-1/4} \Big). \label{ineq:normal_power_expansion_subspace_1}
	\end{align}
	Here $\mathrm{err}_0, \mathscr{L}(\cdot)$ are defined in Theorem \ref{thm:power_lrt_global_testing}. Consequently:
	\begin{enumerate}
		\item For $\mu \in K_0$, the LRT has size $\E_0 \Psi(Y;m_0,\sigma_0)$, where
		\begin{align*}
		\bigabs{\E_0 \Psi(Y;m_0,\sigma_0)-\alpha}\leq \frac{16}{\sqrt{\delta_K-\delta_{K_0}}}.
		\end{align*}
		\item Suppose further $\delta_K - \delta_{K_0}\to \infty$.  Then for $\mu \in K$,
		\begin{align}\label{ineq:normal_power_expansion_subspace_2}
		&\mathcal{L}\bigg(\bigg\{\frac{\Gamma_{K,2}\big(\mu-\Pi_{K_0}(\mu)\big)}{\sigma_0}\bigg\}\bigg) \subset \Delta_{\mathcal{A}_\alpha}^{-1}(\beta) \cap [0,+\infty] \nonumber\\
		\Leftrightarrow \quad & \mathcal{L}\bigg(\bigg\{\frac{2\Gamma_K\big(\mu-\Pi_{K_0}(\mu)\big) }{ \sqrt{2+r(K,K_0)} \sqrt{1-\delta_{K_0}/\delta_K}}\bigg\}\bigg) \subset \Delta_{\mathcal{A}_\alpha}^{-1}(\beta) \cap [0,+\infty] \nonumber\\
		\Leftrightarrow \quad  &\E_\mu \Psi(Y;m_{0},\sigma_{0})\to \beta \in [0,1],
		\end{align}
		where  $r(K,K_0)\equiv \var(V_{K\cap K_0^*})/\delta_{K\cap K_0^*} \in [0,2]$. Hence the LRT is power consistent under $\mu$, i.e., $\E_\mu \Psi(Y;m_0,\sigma_0)\to 1$, if and only if 
		\begin{align}\label{cond:power_lrt_subspace_2}
		&\frac{ \Gamma_{K,2}\big(\mu-\Pi_{K_0}(\mu)\big) }{\big(\delta_K-\delta_{K_0}\big)^{1/2}}\to +\infty\quad \Leftrightarrow \quad \frac{\Gamma_{K}\big(\mu-\Pi_{K_0}(\mu)\big)}{\sqrt{1-\delta_{K_0}/\delta_K}} \to +\infty. 
		\end{align}
	\end{enumerate}		
\end{theorem}

\begin{remark}\label{rmk:subspace_one_sided_LRT}
	\begin{enumerate}
		\item By the proof of~\cite[Lemma E.1]{wei2019geometry}, $\Gamma_{K,2}(\nu)\geq \pnorm{\nu}{}^2\geq 0$ for all $\nu \in K$, so  all the limit points in (\ref{ineq:normal_power_expansion_subspace_2}) are nonnegative. This leads to the equivalence of the power consistency property for the one-sided LRT (\ref{def:LRT_one_sided}) and the two-sided LRT (\ref{def:LRT_two_sided}).
		\item With the help of Lemma \ref{lem:invariance_lrt} and (\ref{eq:lrt_mean_var_special}), which holds for any $\mu\in K_0$, some calculations yield that
		\begin{align}\label{ineq:subspace_quad_lower_bound}
		m_\mu - m_0 = \Gamma_{K,2}(\mu-\Pi_{K_0}(\mu))\geq \pnorm{\mu-\Pi_{K_0}(\mu)}{}^2.
		\end{align} 
		Therefore, the counterpart of the generic condition (\ref{cond:power_lrt_global_testing_0}) under (\ref{eq:testing_special})
		\begin{align*}
		\pnorm{\mu - \Pi_{K_0}(\mu)}{} \ll \abs{m_{\mu} - m_0} \vee \sigma_0
		\end{align*}
		is automatically satisfied due to the global quadratic lower bound (\ref{ineq:subspace_quad_lower_bound}). In particular, (\ref{ineq:normal_power_expansion_subspace_1}) vanishes under the growth condition $\delta_K - \delta_{K_0}\to \infty$.
	\end{enumerate}
\end{remark}

The power behavior of the LRT is characterized using $\Gamma_{K,2}$ and $\Gamma_K$ in Theorem \ref{thm:power_lrt_subspace}. The function $\Gamma_{K,2}$ is usually more amenable to explicit calculations in concrete examples, while the formulation using $\Gamma_K$ allows us to recover the separation rate in $\pnorm{\cdot}{}$ for the LRT derived in \cite{wei2019geometry} in the setting (\ref{eq:testing_special}). We formally state this result below; see Section~\ref{pf:wwg} for a proof.

\begin{corollary}\label{cor:wwg}
	For $\Psi \in \{\Psi_{\mathrm{os}},\Psi_{\mathrm{ts}}\}$, (\ref{cond:power_lrt_subspace_2}) is satisfied for any $\mu \in K$ such that 
	\begin{align}\label{cond:wwg}
	\pnorm{\mu-\Pi_{K_0}(\mu)}{} \gg \delta_{K}^{1/4} \bigwedge \bigg(\frac{ \delta_{K}^{1/2} }{ 0 \bigvee \inf_{\eta \in K\cap B(1)} \iprod{\eta}{\E \Pi_{K}(\xi)} }\bigg).
	\end{align}
\end{corollary}
Below we give a detailed comparison of (\ref{cond:power_lrt_subspace_2}) and its sufficient condition (\ref{cond:wwg}) due to \cite{wei2019geometry}:

\begin{itemize}
	\item (\textit{Optimality}) By \cite{wei2019geometry}, condition (\ref{cond:wwg}) cannot be further improved in the \emph{worst} case in the sense that for every fixed pair $(K_0, K)$, there exists some $\mu\in K$ violating (\ref{cond:wwg}) that invalidates (\ref{cond:power_lrt_subspace_2}). Furthermore, the same work also shows that the uniform $\|\cdot\|$-separation rate in (\ref{cond:wwg}) is minimax optimal in many cone testing problems.
	\item (\textit{Non-uniform power}) On the other hand, it is important to mention that (\ref{cond:power_lrt_subspace_2}) is not equivalent to (\ref{cond:wwg}). In fact, as we will see in the example of testing $0$ versus the orthant cone $K_+$ and the product circular cone $K_{\times,\alpha}$ (to be detailed in Corollary \ref{cor:lrt_clt_orthant} and Theorem \ref{thm:lrt_clt_global_testing_circular}), the worst case condition (\ref{cond:wwg}) in terms of a separation in $\pnorm{\cdot}{}$ is too conservative: condition (\ref{cond:power_lrt_subspace_2}) allows natural configurations of $\mu \in \{K_+,K_{\times,\alpha}\}$ whose separation rate in $\pnorm{\cdot}{}$ can be $n^\delta$ for any $\delta \in (0,1/4)$, while (\ref{cond:wwg}) necessarily requires a separation rate in $\pnorm{\cdot}{}$ of order at least $n^{1/4}$. Therefore, although (\ref{cond:wwg}) gives the best possible inversion of (\ref{cond:power_lrt_subspace_2}) in terms of uniform separation in $\pnorm{\cdot}{}$, condition (\ref{cond:power_lrt_subspace_2}) can be much weaker than (\ref{cond:wwg}), and characterizes the non-uniform power behavior of the LRT. 
\end{itemize}

To give a better sense of the results in Theorem \ref{thm:power_lrt_subspace}, we consider a toy example where $K$ is also a subspace. 
\begin{proposition}\label{prop:power_lrt_subspace_toy}
	Let $\Psi \in \{\Psi_{\mathrm{os}},\Psi_{\mathrm{ts}}\}$. Suppose $\delta_{K} - \delta_{K_0} \to \infty$. 
	\begin{enumerate}
		\item If $\mu \in K_0$, the LRT is asymptotically size $\alpha$ with $\E_\mu \Psi(Y;m_0,\sigma_0) = \alpha+\mathcal{O}\big(\big(\delta_{K} - \delta_{K_0}\big)^{-1/2}\big)$.
		\item For $\mu \in K$, the LRT is power consistent under $\mu$, i.e., $\E_\mu \Psi(Y;m_0,\sigma_0)\to 1$, if and only if $\pnorm{\mu-\Pi_{K_0}(\mu)}{}\gg \big(\delta_{K} - \delta_{K_0}\big)^{1/4}$.
	\end{enumerate}
\end{proposition}
\begin{proof}
	(1) is a direct consequence of Theorem \ref{thm:power_lrt_subspace}-(1). (2) follows from Theorem \ref{thm:power_lrt_subspace}-(2) upon noting that
	\begin{align*}
	&\Gamma_{K,2}(\mu-\Pi_{K_0}(\mu))=\E \pnorm{\Pi_{K}\big(\mu-\Pi_{K_0}(\mu)+\xi\big)}{}^2 - \E \pnorm{\Pi_{K}(\xi)}{}^2 = \pnorm{\mu-\Pi_{K_0}(\mu)}{}^2,\\
	&\sigma_0^2 = \var \big( \pnorm{\Pi_{K\cap K_0^\ast}(\xi)}{}^2\big)=2\delta_{K\cap K_0^\ast} = 2(\delta_{K}-\delta_{K_0}).
	\end{align*}
	The second line of the above display uses Lemma \ref{lem:var_proj}-(3).
\end{proof}

More examples on testing parametric assumptions versus shape-constrained alternatives will be detailed in Section \ref{section:example}.

\section{Examples}\label{section:example}

This section is organized as follows. Sections \ref{subsec:example_orthant}-\ref{subsec:example_lasso} study the generic testing problem (\ref{eq:testing_generic}) in the context of orthant/circular cones, isotonic regression, and Lasso, respectively. Section \ref{subsec:example_subspace} specializes the subspace versus cone testing problem (\ref{eq:testing_special}) to the setting of testing parametric assumptions versus shape-constrained alternatives. For simplicity of presentation, we will focus on the two-sided LRT (\ref{def:LRT_two_sided}), and simply call it the LRT unless otherwise specified.

\subsection{Testing in orthant cone}\label{subsec:example_orthant}

Consider the orthant cone $$K_{+}\equiv \left\{\nu =(\nu_1,\ldots, \nu_n) \in \R^n: \nu_i\geq 0, i \in [1:n]\right\}.$$ We are interested in the testing problem (\ref{eq:testing_generic}) with $K = K_+$. Testing in the orthant cone has previously been studied by \cite{kudo1963multivariate,raubertas1986hypothesis,wei2019geometry}. The following result (see Section~\ref{pf:lrt_clt_global_testing_orthant} for a proof) gives the limiting distribution of the LRS and characterizes the power behavior of the LRT in this example.

\begin{theorem}\label{thm:lrt_clt_global_testing_orthant}
	\begin{enumerate}
		\item There exists a universal constant $C>0$ such that for $\mu_0 \in K_+$, 
		\begin{align*}
		d_{\mathrm{TV}}\bigg( \frac{T(Y)-m_{\mu_0}}{ \sigma_{\mu_0} },\mathcal{N}(0,1)\bigg)&\leq \frac{C}{\sqrt{n}}.
		\end{align*}
		Consequently the  LRT is asymptotically size $\alpha$ with $\E_{\mu_0} \Psi_{\mathrm{ts}}(Y;m_{\mu_0},\sigma_{\mu_0}) =\alpha + \mathcal{O}(n^{-1/2})$.
		\item \sloppy For any $\mu \in K_+$, the LRT is power consistent under $\mu$, i.e., $\E_{\mu} \Psi_{\mathrm{ts}}(Y;m_{\mu_0},\sigma_{\mu_0}) \to 1$, if and only if 
		\begin{align*}
		\biggabs{ \sum_{i=1}^n \big\{\bar{S}_+(\mu_i)-\bar{S}_+((\mu_0)_i)\big\}+\pnorm{\mu-\mu_0}{}^2}\gg n^{1/2}.
		\end{align*}
		Here, $\bar{S}_+$ is an increasing, concave and bounded function on $[0,\infty)$ with $\bar{S}_+(0)=0$ and defined as
		\begin{align}\label{def:S_+}
		\bar{S}_+(x)\equiv\Phi(x)+x\varphi(x)-x^2(1- \Phi(x)) -\frac{1}{2},\quad x\geq 0.
		\end{align}
	\end{enumerate}
\end{theorem}

Let us further investigate the special case $\mu_0=0$ to illustrate the non-uniform power behavior of the LRT mentioned after Theorem \ref{thm:power_lrt_subspace}. In other words, we consider testing $\mu = 0$ versus the orthant cone $K_+$. Let 
\begin{align*}
S_+(x)\equiv \bar{S}_+(x)+x^2 = \Phi(x)+x \varphi(x)+x^2 \Phi(x) - \frac{1}{2},\quad x\geq 0.
\end{align*}
As $S_+'(x) = 2\big[\varphi(x)+x \Phi(x)\big]$, $S_+'(0)=2\varphi(0)>0$, and $S_+''(x)= 2\Phi(x)\geq 0$, $S_+$ is a strictly increasing and convex function on $[0,\infty)$ with $S_+(0)=0$. Furthermore, it can be verified via direct calculation that uniformly over $x \geq 0$, $S_+(x) \asymp x \vee x^2$. Theorem \ref{thm:lrt_clt_global_testing_orthant} immediately yields the following corollary.

\begin{corollary}\label{cor:lrt_clt_orthant}
	\begin{enumerate}
		\item \sloppy For $\mu=0$, the  LRT is asymptotically size $\alpha$ with $\E_0 \Psi_{\mathrm{ts}}(Y;{m}_0,{\sigma}_0) = \alpha+\mathcal{O}(n^{-1/2})$. 
		\item  For $\mu \in K_+$, the  LRT is power consistent under $\mu$, i.e., $\E_\mu \Psi_{\mathrm{ts}}(Y;{m}_0,{\sigma}_0)\to 1$, if and only if $\pnorm{\mu}{1}\vee \pnorm{\mu}{}^2\gg n^{1/2}$.
	\end{enumerate}
\end{corollary}

The results in \cite[Section 3.1.5]{wei2019geometry}, or equivalently, condition (\ref{cond:wwg}) show that the type II error of an optimally calibrated LRT vanishes uniformly for $\mu \in K_+$ such that $\pnorm{\mu}{}\gg n^{1/4}$. Our results above indicate that the regime where the LRT has asymptotic power 1, for the orthant cone $K_+$, is actually characterized by the condition $\|\mu\|_1\vee \|\mu\|^2 \gg n^{1/2}$ and is hence non-uniform with respect to $\pnorm{\cdot}{}$. We give two concrete examples below.
\begin{example}\label{example:orthant}
	Let $q \in (0,1/2)$ and $\tau_1,\tau_2>0$ be two fixed positive constants. Consider the following alternatives: (1) $\mu=(\tau_1 n^{-q})  \bm{1}_n\in K_+$, and (2) $\mu = (\tau_2 i^{-q})_{i=1}^n \in K_+$.  In both cases, $\pnorm{\mu}{1}\asymp n^{1-q}$ and $\pnorm{\mu}{}^2 \asymp n^{1-2q}$. The above corollary then yields that the LRT is power consistent under $\mu$ if and only if $q \in (0,1/2)$, while the characterization of \cite{wei2019geometry} guarantees power consistency of the LRT only for $q \in (0,1/4)$. In particular, as $q\rightarrow 1/2$, the LRT is power consistent for certain alternative $\mu$ with $\pnorm{\mu}{}\asymp n^\delta$ for any $\delta > 0$. See Section \ref{subsubsec:simulation_orthant} ahead for some simulation evidence.
\end{example}

One may further wonder whether the above examples only highlight `exceptional' alternatives in the regime where the uniform separation in $\pnorm{\cdot}{}$ fails to be informative, i.e., with $M_n \equiv \{ \mu \in K_+: \pnorm{\mu}{}^2\leq Cn^{1/2}\}$ for some large enough absolute constant $C>0$, whether the above examples only constitute a small fraction of $M_n$. To this end, let $A_n\equiv \{\mu \in M_n: \pnorm{\mu}{1}\vee \pnorm{\mu}{}^2\geq Cn^{1/2}\}$ be the region in $M_n$ in which the LRT is indeed powerful. By a standard volumetric calculation, it is easy to see that $A_n/M_n \to 1$. In other words, the LRT is indeed powerful for `most' alternatives in the region where the uniform separation in $\pnorm{\cdot}{}$ is not informative as $n \to \infty$. Hence the non-uniform characterization in Corollary \ref{cor:lrt_clt_orthant}-(2) is essential for determining whether the LRT is powerful for a given alternative $\mu \in K_+$ in the regime $\pnorm{\mu}{}=\mathcal{O}(n^{1/4})$.

As the separation rate $n^{1/4}$ in $\pnorm{\cdot}{}$ is minimax optimal for testing $0$ versus $K_+$ (cf.~\cite[Proposition 1]{wei2019geometry}), the discussion above also illustrates the conservative nature of the minimax formulation in this testing problem.

\subsubsection{An illustrative simulation study}\label{subsubsec:simulation_orthant}

\begin{figure}[t]
	\begin{minipage}[b]{0.45\textwidth}
		\includegraphics[width=\textwidth]{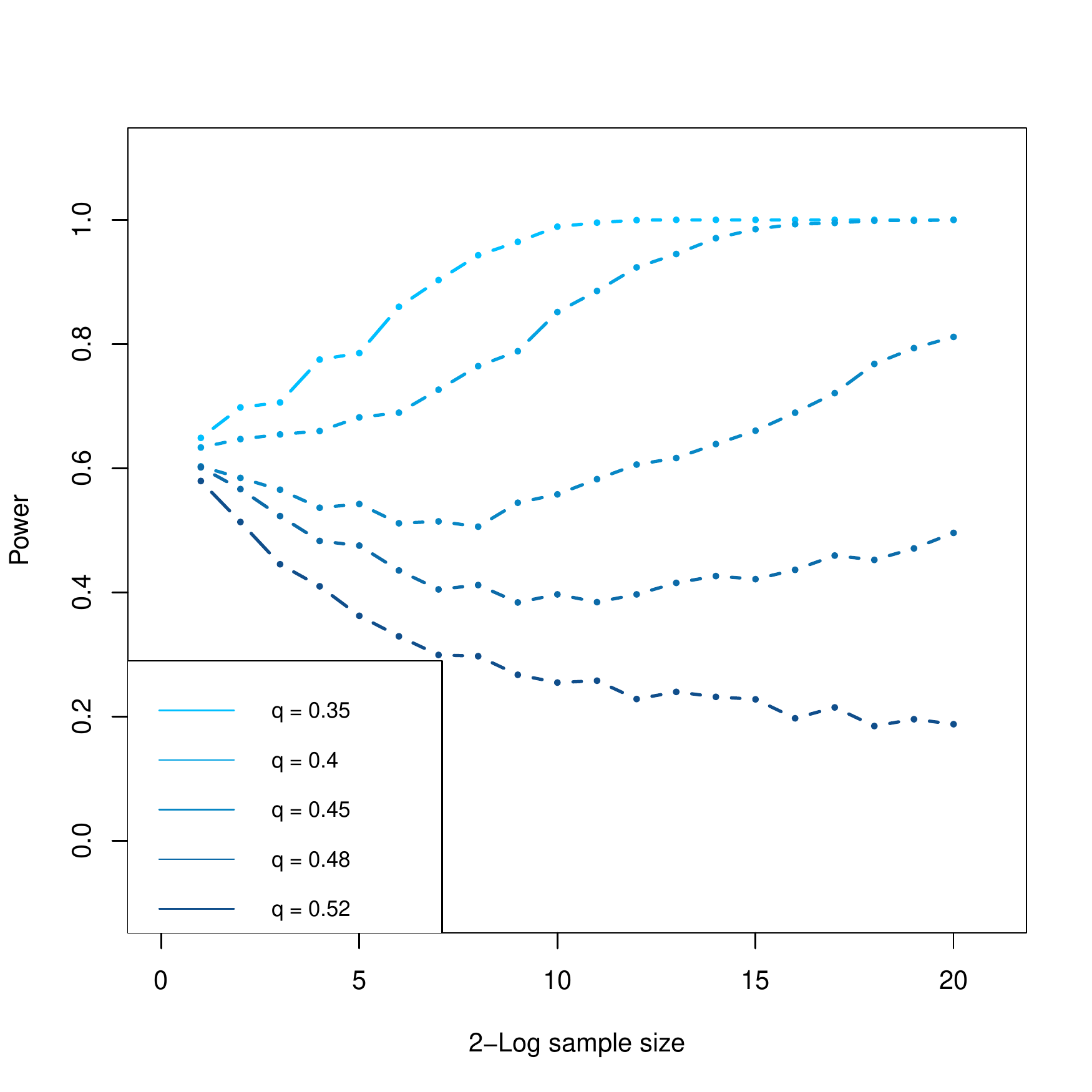}
	\end{minipage}
	\begin{minipage}[b]{0.45\textwidth}
		\includegraphics[width=\textwidth]{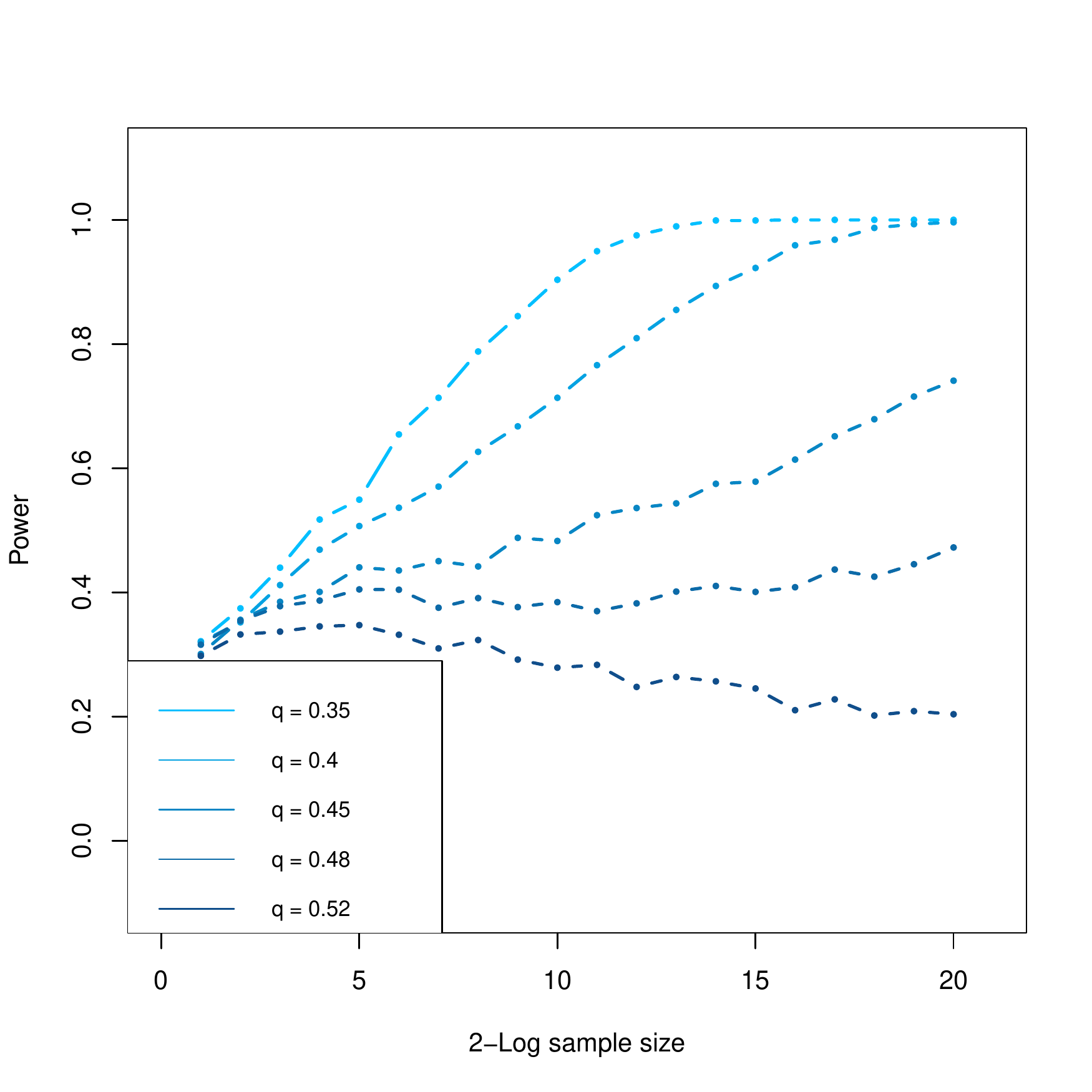}
	\end{minipage}
	\caption{The power curves for the alternatives $\mu = (2n^{-q})_{i=1}^n$ (in the left panel) and $\mu = (i^{-q})_{i=1}^n$ (in the right panel) as $q$ varies, for sample sizes $n\in \{2^\ell,\ell\in[1:20]\}$. The plots illustrate that the LRT has power  in the range $q \in (0,1/2)$ in both the examples.}
	\label{fig:power_orthant}
\end{figure}

\begin{figure}[t]
	\begin{minipage}[b]{0.45\textwidth}
		\includegraphics[width=\textwidth]{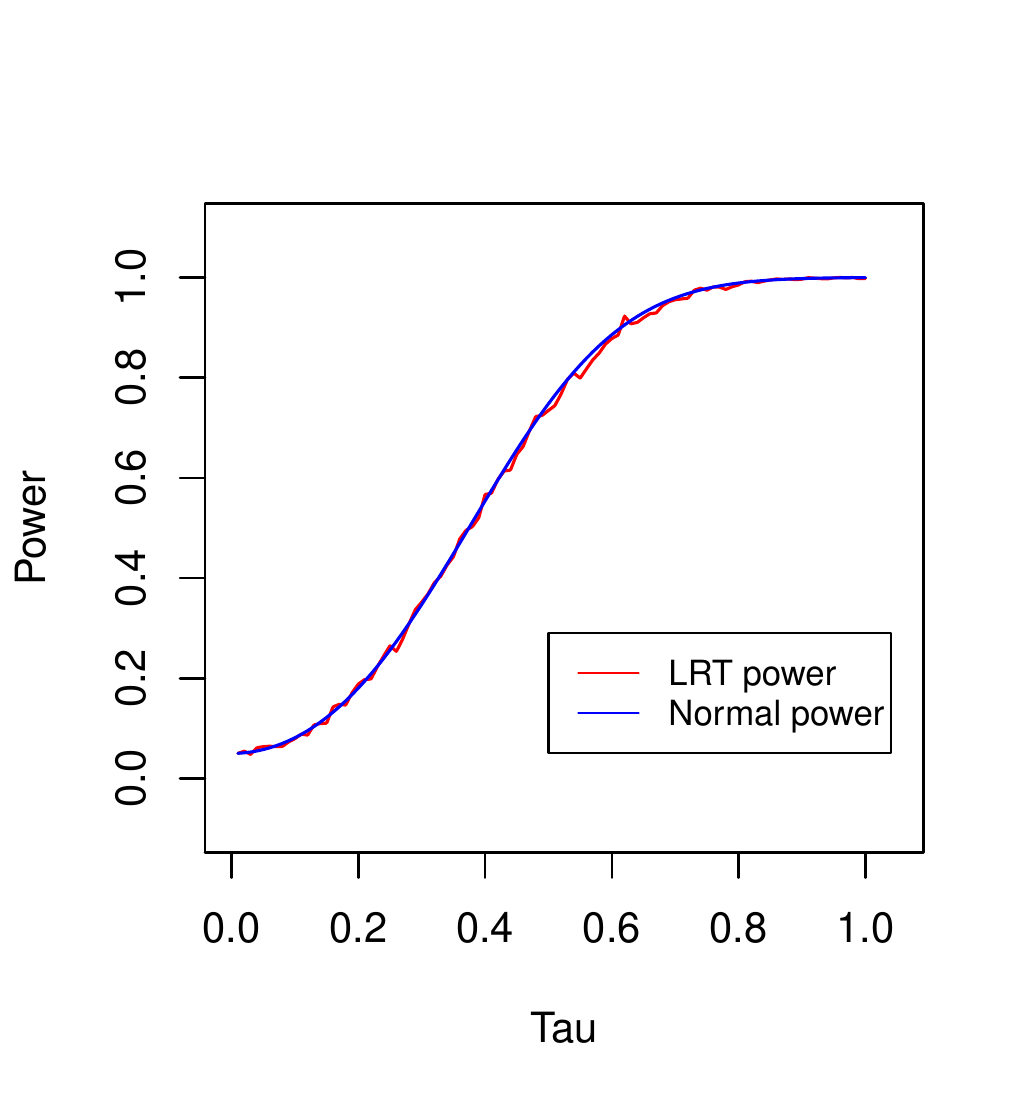}
	\end{minipage}
	\begin{minipage}[b]{0.45\textwidth}
		\includegraphics[width=\textwidth]{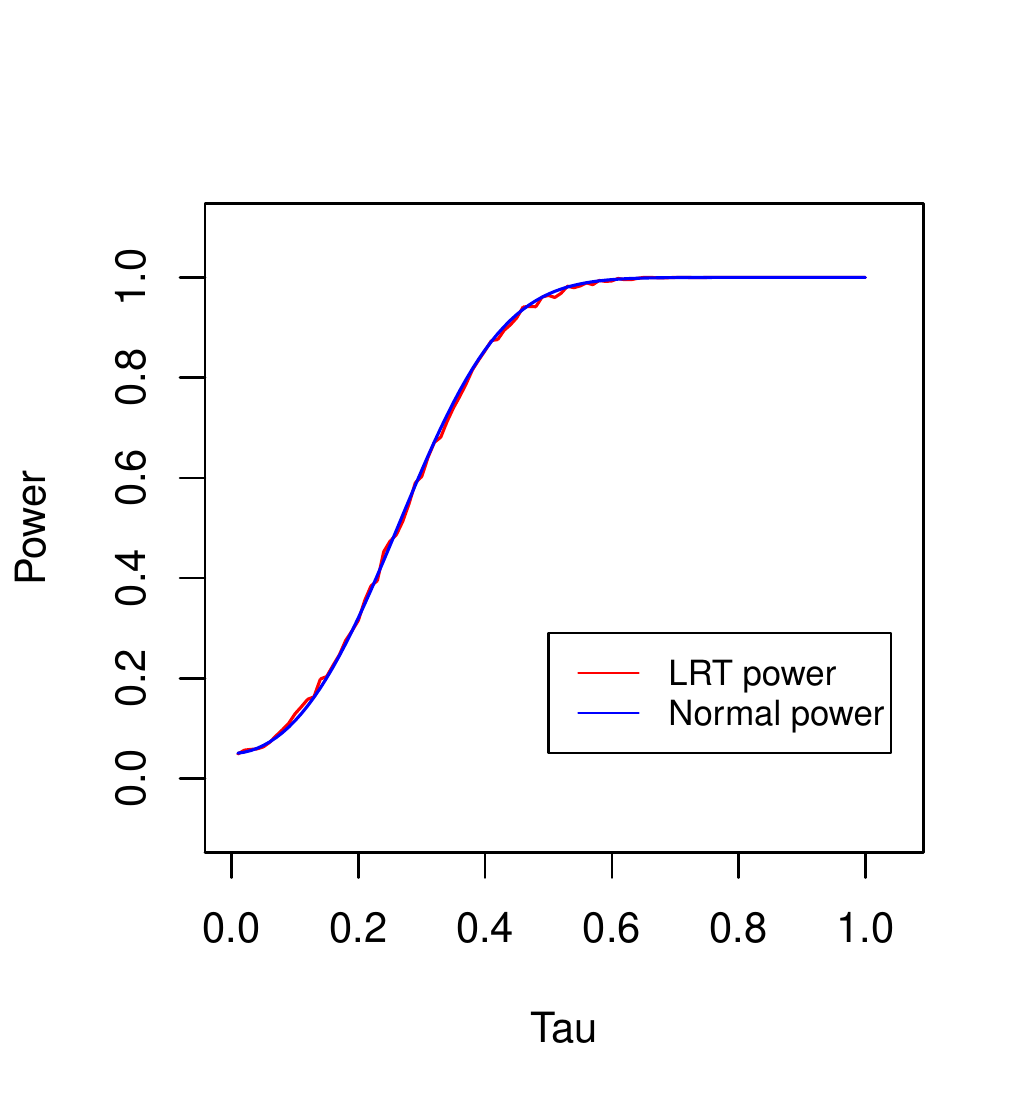}
	\end{minipage}
	\caption{Fixed $q = 0.3$ and $n = 20000$. The alternatives are $\mu = (\tau_1 n^{-0.3})$ in left panel and $\mu = (\tau_2 i^{-0.3})$ in the right panel with $\tau_1,\tau_2 \in \{0.01,0.02,\ldots,1\}$. The red line denotes the power curve of LRT, i.e., $\{\E_\mu \Psi_{\mathrm{ts}}(Y;m_0,\sigma_0):\tau_1 \}$, while the blue line denotes the theoretical power curve via normal approximation, i.e., $\{\Delta_{\mathcal{A}_\alpha}\big(\Gamma_{K,2}(\mu)/\sigma_0\big):\tau_2\}$. }
	\label{fig:power_orthant_fixed_direction}
\end{figure}

Below we present simulation results under the two settings considered in Example \ref{example:orthant}. The confidence level will be taken as $\alpha=0.05$. The power of the LRT in both the simulations below is calculated using an average of $2000$ replications.
\begin{itemize}
	\item In Figure \ref{fig:power_orthant}, we take $\tau_1=2,\tau_2=1$ and examine the sharpness of the power characterization $q \in (0,1/2)$ predicted by Corollary \ref{cor:lrt_clt_orthant}-(2).	Clearly, Figure \ref{fig:power_orthant} shows that $q \in (0,1/2)$ is the correct range where the LRT is powerful in both the settings of Example \ref{example:orthant}, rather than $q \in (0,1/4)$ as predicted by \cite{wei2019geometry}. 
	\item In Figure \ref{fig:power_orthant_fixed_direction}, we fix $q=0.3$,  $n=20000$, and examine the validity of the normal power expansion (\ref{ineq:normal_power_expansion_subspace}) in Theorem \ref{thm:power_lrt_subspace} along the alternatives considered in Example \ref{example:orthant} with $\tau_1,\tau_2 \in \{0.01,0.02,\ldots,1\}$. Formally, we consider two power curves: (i) the power of the LRT, i.e., $\E_\mu \Psi_{\mathrm{ts}}(Y;m_0,\sigma_0)$, (ii) theoretical power given by the normal approximation, i.e., $\Delta_{\mathcal{A}_\alpha}\big(\Gamma_{K,2}(\mu)/\sigma_0\big)$, for alternatives of the form $\mu = (\tau_1 n^{-0.3})_{i=1}^n$ and $\mu = (\tau_2 i^{-0.3})_{i=1}^n$ with the prescribed $\tau_1,\tau_2$'s. Figure \ref{fig:power_orthant_fixed_direction} clearly shows that the two power curves are very close to each other.
\end{itemize}

\subsubsection{Counter-examples}

Let $\mu_0 = \bm{1}_n \in K_+$, and $\mu = c \bm{1}_n$ for some fixed $c > 0$ to be determined. As long as $c\neq 1$, we have $\|\mu - \mu_0\|^2 = n(c-1)^2 \asymp n$. We also have $\sigma_\mu^2 = n\cdot\var\big[(c+\xi_1-1)^2 - (c+\xi_1)_-^2\big]\equiv n \rho^2(c) \asymp n$, and
\begin{align*}
m_\mu - m_{\mu_0} = \|\mu - \mu_0\|^2 + \sum_{i=1}^n\big(\bar{S}_+(c) - \bar{S}_+(1)\big) = n\big\{(c-1)^2 + \bar{S}_+(c) - \bar{S}_+(1)\big\},
\end{align*}
where $\bar{S}_+$ is defined in (\ref{def:S_+}). 
Let $F(c)\equiv (c-1)^2 + \bar{S}_+(c) - \bar{S}_+(1)$. Then $F(1) = 0$, $F(0) = 0.5753...$, and $F^\prime(1) = \bar{S}_+^\prime(1) = 0.1666... > 0$.

We first present a choice of $c$ that leads to an example showing the necessity of (\ref{cond:power_lrt_global_testing_0}) for the power characterization  (\ref{eq:power_lrt_global_testing}).

\begin{example}\label{example:necessity_reg_cond}
	By the previous discussion,  $F$ must admit a zero in the open interval $(0,1)$, which we denote as $c_0$. With $c = c_0$, we then have $m_\mu = m_{\mu_0}$. Moreover, as $\sigma_\mu^2 = n\rho^2(c_0)\neq n \rho^2(1)= \sigma_{\mu_0}^2$, so by Theorem \ref{thm:lrt_clt_global_testing_orthant}-(1),
	\begin{align*}
	\frac{T(\mu+\xi)-m_{\mu_0}}{\sigma_{\mu_0}} = \frac{T(\mu+\xi)-m_{\mu}}{\sigma_{\mu}}\cdot \frac{\sigma_\mu}{\sigma_{\mu_0}} \rightarrow_d \mathcal{N}\bigg(0,\frac{\rho^2(c_0)}{\rho^2(1)}\bigg)\neq_d \mathcal{N}(0,1).
	\end{align*}
	This means  (\ref{eq:power_lrt_global_testing}) fails.
\end{example}

Next we present a choice of $c$ that leads to an example showing the necessity of considering two-sided LRT. 

\begin{example}\label{example:LRT_two_sided}
	By the previous discussion, $F(c)<0$ for $c \in (0,1)$ near 1. Pick any $c_1 \in (0,1)$ such that $F(c_1)<0$ and consider $c=c_1$. Let $\mu= c_1 \bm{1}_n$. As $\sigma_{\mu_0}\asymp n^{1/2}$,  $(m_\mu-m_{\mu_0})/\sigma_{\mu_0} \asymp -n^{1/2}$, so by Theorem \ref{thm:lrt_clt_global_testing_orthant}-(1),
	\begin{align*}
	\frac{T(\mu+\xi)-m_{\mu_0}}{\sigma_{\mu_0}} = \frac{T(\mu+\xi)-m_{\mu}}{\sigma_{\mu}}\cdot \frac{\sigma_\mu}{\sigma_{\mu_0}}+ \frac{m_\mu-m_{\mu_0}}{\sigma_{\mu_0}} \to -\infty
	\end{align*}
	in probability. This means that the two-sided LRT in (\ref{def:LRT_two_sided}) is powerful under $\mu= c_1 \bm{1}_n$, i.e., $\E_\mu \Psi_{\mathrm{ts}}(Y;m_{\mu_0},\sigma_{\mu_0})\to 1$, but the one-sided LRT in (\ref{def:LRT_one_sided}) is not powerful under $\mu$, i.e., $\E_\mu \Psi_{\textrm{os}}(Y;m_{\mu_0},\sigma_{\mu_0})\to 0$.
\end{example}

\subsection{Testing in circular cone}\label{subsec:example_circular}

For any $\alpha\in(0,\pi/2)$, let the $\alpha$-circular cone be defined by $$K_\alpha\equiv \left\{\nu\in \R^{n-1}: \nu_1\geq \|\nu\|\cos(\alpha) \right\},$$ and let $ K_{\times,\alpha} \equiv K_\alpha\times \R\subset\R^n$. Consider the testing problem (\ref{eq:testing_generic}) with $\mu_0 = 0$ and $K\in \{K_\alpha, K_{\times,\alpha}\}$. The circular cone has recently been used in modeling by \cite{besson2006adaptive,greco2008radar}. The following result (see Section~\ref{pf:lrt_clt_global_testing_circular}  for a proof) gives the limiting distribution of the LRS and characterizes the power behavior of the LRT in this example.

\begin{theorem}\label{thm:lrt_clt_global_testing_circular}
	\begin{enumerate}
		\item Let $K \in \{K_\alpha,K_{\times,\alpha}\}$. There exists some universal constant $C>0$ such that, 
		\begin{align*}
		d_{\mathrm{TV}}\bigg( \frac{T(Y)-m_{0}}{ \sigma_{0} },\mathcal{N}(0,1)\bigg)&\leq \frac{C}{\sqrt{n} }.
		\end{align*}
		Consequently the LRT is asymptotically size $\alpha$ with $\E_{\mu_0} \Psi_{\mathrm{ts}}(Y;m_0,\sigma_0) =\alpha + \mathcal{O}(n^{-1/2})$.
		\item 
		\begin{enumerate}
			\item \sloppy For any $\mu \in K_\alpha$, the LRT is power consistent under $\mu$, i.e., $\E_{\mu} \Psi_{\mathrm{ts}}(Y;m_0,\sigma_0) \to 1$, if and only if $\pnorm{\mu}{}\gg 1$.
			\item \sloppy For any $\mu \in K_{\times,\alpha}$, the LRT is power consistent under $\mu$, i.e., $\E_{\mu} \Psi_{\mathrm{ts}} (Y;m_0,\sigma_0) \to 1$, if and only if $\|\mu^1\|\gg 1$ or $\abs{\mu^2} \gg n^{1/4}$.
		\end{enumerate}
		Here for any $\mu\in \R^n$, $\mu = (\mu^1,\mu^2) \in \R^{n-1}\times \R$ with $\mu^1 \in \R^{n-1}$ denoting the first $n-1$ components of $\mu$ and $\mu^2 \in \R$ denoting the last. 
	\end{enumerate}
\end{theorem}

Regarding the two cones $\{K_\alpha,K_{\times,\alpha}\}$, \cite{wei2019geometry} showed the following:
\begin{itemize}
	\item For $K_\alpha$,  an optimally calibrated LRT is powerful for $\mu \in K_\alpha$ such that $\pnorm{\mu}{}\gg 1$. The minimax $\|\cdot\|$-separation rate is of the same constant order, so the LRT is minimax optimal.
	\item For $K_{\times,\alpha}$,  an optimally calibrated LRT is powerful for $\mu \in K_{\times,\alpha}$ such that $\pnorm{\mu}{}\gg n^{1/4}$, while the minimax $\|\cdot\|$-separation rate is of constant order, so the LRT is strictly minimax sub-optimal.
\end{itemize}

Theorem \ref{thm:lrt_clt_global_testing_circular}-(2) is rather interesting compared to the above results of \cite{wei2019geometry}:
\begin{itemize}
	\item For $K_\alpha$, Theorem \ref{thm:lrt_clt_global_testing_circular}-(2)(a) shows that the power behavior of LRT is uniform with respect to $\pnorm{\cdot}{}$ for $K_\alpha$. In other words, for any $\mu\in K_\alpha$ with $\|\mu\| = \mathcal{O}(1)$ the LRT is necessarily not powerful. 
	\item For $K_{\times,\alpha}$, Theorem \ref{thm:lrt_clt_global_testing_circular}-(2)(b) shows that the only bad alternatives that drive the uniform separation rate $n^{1/4}$ in $\pnorm{\cdot}{}$ are those $\mu = (\mu^1,\mu^2) \in K_{\times,\alpha}$ lying in the narrow cylinder $\pnorm{\mu^1}{}=\mathcal{O}(1)$ and $\abs{\mu^2}=\mathcal{O}(n^{1/4})$, and the LRT will be powerful for points of the form, e.g. $(\mu^1, 0)$ as soon as $\|\mu^1\|\gg 1$. This is in line with the result of Theorem \ref{thm:lrt_clt_global_testing_circular}-(2)(a), and provides another example where the LRT exhibits non-uniform power behavior with respect to $\pnorm{\cdot}{}$. 
\end{itemize}
Similar to the LRT in the orthant cone, one may easily see that the conservative uniform separation rate (i.e., $\pnorm{\mu}{}\gg n^{1/4}$) in $\pnorm{\cdot}{}$ for $K_{\times,\alpha}$ fails to detect `most' alternatives where the LRT is powerful, as $n \to \infty$. In this sense, the minimax sub-optimality of LRT for testing $0$ versus $K_{\times,\alpha}$ is also conservative as the LRT behaves badly for only a few alternatives with large separation rate in $\pnorm{\cdot}{}$.

The phenomenon observed above for the product circular cone can be easily extended as follows. For some positive integer $m$ and generic closed convex cones $K_i \subset \R^{n_i}$, $i=1,\ldots,m$, let $K_{\times} \equiv \times_{i=1}^m K_i \subset \R^{\sum_{i=1}^m n_i}$ be the associated product cone. Then the LRT for testing $0$ versus $K_\times$ is power consistent under $\mu=(\mu^i)_{i=1}^m \in \times_{i=1}^m K_i = K$ if and only if 
\begin{align*}
\frac{ \sum_{i=1}^m \Gamma_{K_i,2}\big(\mu^i\big) }{ \big(\sum_{i=1}^m\delta_{K_i}\big)^{1/2} }\to \infty.
\end{align*} 
The proof is largely similar to Theorem \ref{thm:lrt_clt_global_testing_circular}-(2)(b) so we omit the details.

\subsection{Testing in isotonic regression}\label{subsec:example_isotonic}

Let the monotone cone be defined by $$K_{\uparrow}\equiv K_{\uparrow,0}= \{\nu =(\nu_1,\ldots, \nu_n) \in \R^n: \nu_1\leq \ldots\leq \nu_n\}.$$ We consider the testing problem (\ref{eq:testing_generic}) with $K = K_\uparrow$ using the two-sided LRT (as  in~\eqref{def:LRT_two_sided}). The following result (see Section~\ref{pf:lrt_clt_global_testing_iso}  for a proof) gives the limiting distribution of the LRS and characterizes the power behavior of the LRT in this example. 

\begin{theorem}\label{thm:lrt_clt_global_testing_iso}
	\begin{enumerate}
		\item Suppose $\mu_0 \in K_{\uparrow}$, and for a universal constant $L>1$,
		{\small \begin{align}\label{cond:lrt_clt_iso}
			\frac{1}{L} \leq \min_{1\leq i\leq n-1} n\big((\mu_0)_{i+1}-(\mu_0)_{i}\big) \leq \max_{1\leq i\leq n-1} n\big((\mu_0)_{i+1}-(\mu_0)_{i}\big)\leq L.
			\end{align}}
		Then
		\begin{align*}
		d_{\mathrm{TV}}\bigg( \frac{T(Y)-m_{\mu_0}}{ \sigma_{\mu_0} },\mathcal{N}(0,1)\bigg)\leq \frac{C}{n^{1/6}}.
		\end{align*}
		Here $C>0$ is a constant depending on $L$ only. Consequently the LRT is asymptotically size $\alpha$ with $\E_{\mu_0} \Psi_{\mathrm{ts}} (Y;m_{\mu_0},\sigma_{\mu_0}) =\alpha + \mathcal{O}(n^{-1/6})$.
		\item Let $\mu_0 = \big(f_0(i/n)\big)_{i=1}^n$ and $\mu = \big(f(i/n)\big)_{i=1}^n$, where $f,f_0:[0,1]\to \R$ are $C^2$ monotone functions related by $f_0 = f+\rho_n \delta$ for some $C^1$ function $\delta:[0,1]\to \R$ with $\int \delta^2 =1$ and $\delta'$ is bounded away from $0$ and $\infty$. Suppose the first derivatives $f',f_0'$ are bounded away from $0$ and $\infty$ and second derivatives $f'',f_0''$ are bounded away from $\infty$. Then the LRT is power consistent under $\mu$, i.e., $\E_\mu \Psi_{\mathrm{ts}}(Y;m_{\mu_0},\sigma_{\mu_0}) \to 1$, if and only if $\rho_n \gg n^{-5/12}$, if and only if $\pnorm{\mu-\mu_0}{}\gg n^{1/12}$.
	\end{enumerate}
\end{theorem}

A few remarks are in order.
\begin{itemize}
	\item (\textit{Normal approximation}) The normal approximation  in Theorem~\ref{thm:lrt_clt_global_testing_iso}-(1) settles the problem of the limiting distribution for the LRT used in the simulation in \cite[Section 4]{durot2001goodness}. There the LRT is compared to a goodness-of-fit test based on the central limit theorem for the $\ell_1$ estimation error of isotonic LSE (cf. \cite{groeneboom1985estimating,groeneboom1999asymptotic,durot2007error}). We note that  condition~(\ref{cond:lrt_clt_iso}) on the sequence $\mu_0$ is equivalent to a bounded first derivative away from $0$ and $\infty$ at the function level. This condition is commonly adopted in global CLTs for $\ell_p$ type losses of isotonic LSEs, cf.~\cite{durot2007error}. In fact, the condition in \cite{durot2007error} is stronger than (\ref{cond:lrt_clt_iso}) to guarantee a CLT for $\ell_p$ estimation error of the isotonic LSE.
	
	\item (\textit{Rate of normal approximation}) We conjecture that the error rate $\mathcal{O}(n^{-1/6})$ in the above normal approximation is optimal based on the following heuristics. Writing $\hat{\mu}$ as a shorthand for $\hat{\mu}_{K_\uparrow}$, the LRS $T(Y)$ can be written, under $H_0$, as
	\begin{align*}
	T(Y)& = 2\iprod{\xi}{\hat{\mu}-\mu_0}-\pnorm{\hat{\mu}-\mu_0}{}^2\\
	& = \sum_{i=1}^n \Big(2\xi_i (\hat{\mu}_i-(\mu_0)_i)- (\hat{\mu}_i-(\mu_0)_i)^2\Big).
	\end{align*}
	Under the regularity condition (\ref{cond:lrt_clt_iso}), the isotonic LSE $\hat{\mu}$ is localized in the sense that each $\hat{\mu}_i$ roughly depends on $\mu_0$ and $\xi$ only via indices in a local neighborhood of $i$ that contains $\mathcal{O}(n^{2/3})$ many points. So one may naturally view $T(Y)$ as roughly a summation of $\mathcal{O}(n^{1/3})$ `independent' blocks, each of which roughly has variance of constant order. This naturally leads to the $ \mathcal{O}(1/\sqrt{n^{1/3}})=\mathcal{O}(n^{-1/6})$ rate in the Berry-Esseen bound of Theorem \ref{thm:lrt_clt_global_testing_iso}-(1). Our Theorem \ref{thm:lrt_clt_global_testing_iso}-(1) formalizes this intuition, but the proof is along a completely different line. 
	
	\item (\textit{Local power analysis}) The `local alternative' setting in Theorem \ref{thm:lrt_clt_global_testing_iso}-(2) follows that of \cite{durot2001goodness}. In particular, the separation rate in Theorem \ref{thm:lrt_clt_global_testing_iso}-(2) is reminiscent of \cite[Theorem 3.1]{durot2001goodness}. \cite{durot2001goodness} obtained, under similar configurations and regularity conditions, a separation rate for a goodness-of-fit test based on the CLT for the $\ell_1$ estimation error of the isotonic LSE of order $\rho_n\gg n^{-5/12}\vee n^{-1/2}\delta_n^{-1/2}$, where $\delta_n$ is the length of the support of the function $\delta$. Our results here show that the LRT has a sharp separation rate $\rho_n\gg n^{-5/12}$ under the prescribed configuration, which is no worse than the one derived in \cite{durot2001goodness} based on $\ell_1$ estimation error.
\end{itemize}

In the isotonic regression example above, the main challenge in deriving the normal approximation for $T(Y)$ is to lower bound the quantity $ \pnorm{\E_{\mu_0}J_{\hat{\mu}_{K_{\uparrow}}}}{F}^2$ in (\ref{ineq:lrt_clt_global_testing}). We detail this intermediate result in the following proposition, which may be of independent interest (see Section~\ref{pf:lrt_clt_global_testing_iso} for a proof).

\begin{proposition}\label{prop:lower_bound_J_iso}
	Under the setting of Theorem \ref{thm:lrt_clt_global_testing_iso}-(1), there exists a small enough constant $\kappa>0$, depending on $L$ only, such that
	\begin{align*}
	\big(\E_{\mu_0} J_{\hat{\mu}_{K_\uparrow} }\big)_{ij} \geq \kappa n^{-2/3}
	\end{align*}
	for $\{(i,j): \abs{i-j}\leq \kappa n^{2/3}, 0.1 n\leq i,j\leq 0.9n\}$ for $n$ large enough.
\end{proposition}

The above proposition is proved via exploiting the min-max representation of the isotonic LSE, a property not shared by general shape-constrained LSEs. We conjecture that results analogous to Theorem~\ref{thm:lrt_clt_global_testing_iso} hold for the general $k$-monotone cone $K_{\uparrow,k}$, to be formally defined in Section \ref{subsec:example_subspace}, but an analogue to Proposition \ref{prop:lower_bound_J_iso} above is not yet available for general $K_{\uparrow,k}$.

\subsection{Testing in Lasso} \label{subsec:example_lasso}

Consider the linear regression model $$Y = \mu+\xi\equiv X\theta + \xi,$$ where $X\in\R^{n\times p}$ is a fixed design matrix with $p\leq n$ and full column rank. Let $\Sigma \equiv X^\top X/n$ be the Gram matrix. Let $\hat{\theta}^0 \equiv (X^\top X)^{-1}X^\top Y$ be the ordinary LSE, $\hat{\theta}\equiv\hat{\theta}(\lambda)$ be the constrained Lasso solution defined as
\begin{align}\label{def:lasso_constrained}
\hat{\theta}(\lambda) \equiv \argmin_{\theta\in\R^p} \frac{1}{2}\|Y - X\theta\|^2 \qquad \text{s.t. } \|\theta\|_1\leq \lambda,
\end{align}
and $\hat{\mu}\equiv \hat{\mu}(\lambda) \equiv X\hat{\theta}(\lambda)$. The setting here fits into our general framework by letting $$K \equiv K_{X,\lambda}\equiv \{\mu = X\theta: \|\theta\|_1\leq \lambda\}$$ and $\hat{\mu}_{K}\equiv \hat{\mu}$. Note that we do not impose sparsity of $\theta$ here. We will be interested in the testing problem (\ref{eq:testing_generic}), i.e., $H_0: \mu= \mu_0$ versus $H_1: \mu \in K_{X,\lambda}$, where $\mu_0  = X\theta_0\in K_{X,\lambda}$ with $\pnorm{\theta_0}{1}\leq\lambda$. Such a goodness-of-fit test and the related problem of constructing confidence sets for the Lasso estimator has previously been studied in \cite{verzelen2010goodness,chatterjee2011bootstrapping,nickl2013confidence,shah2018goodness}. In the following, we use  the two-sided LRT $\Psi_{\mathrm{ts}} (Y;m_0,\sigma_0)$ (as  in~\eqref{def:LRT_two_sided}) to test  (\ref{eq:testing_generic}) and study its power characterization (see Section~\ref{pf:lrt_clt_global_testing_lasso} for a proof).

\begin{theorem}\label{thm:lrt_clt_global_testing_lasso}
	Suppose $p\rightarrow\infty$. For $\mu\in K_{X,\lambda}$, let $$ \mathfrak{p}_{\lambda,\mu}\equiv \Prob_\mu\big(\pnorm{\hat{\theta}^0}{1}~\geq~\lambda\big).$$
	\begin{enumerate}
		\item There exists a universal constant $C>0$ such that, for $\mu_0 \in K_{X,\lambda}$, 
		\begin{align*}
		d_{\mathrm{TV}}\bigg( \frac{T(Y)-m_{\mu_0}}{ \sigma_{\mu_0} },\mathcal{N}(0,1)\bigg)&\leq \frac{C \sqrt{p+ n \mathfrak{p}_{\lambda,\mu_0}^{1/2} } }{ \big(p- C (n \mathfrak{p}_{\lambda,\mu_0})^2\big)_+}.
		\end{align*}
		Consequently the LRT is asymptotically size $\alpha$ with $\E_{\mu_0} \Psi_{\mathrm{ts}}(Y;m_{\mu_0},\sigma_{\mu_0}) =\alpha + \mathcal{O}(p^{-1/2})$, provided that $n \mathfrak{p}_{\lambda,\mu_0}^{1/2} =\mathfrak{o}(1)$. 
		\item Suppose $n\cdot(\mathfrak{p}_{\lambda,\mu}^{1/2}\vee \mathfrak{p}_{\lambda,\mu_0}^{1/2}) = \mathfrak{o}(1)$. For any $\mu \in K_{X,\lambda}$, the LRT is power consistent, i.e., $\E_\mu \Psi_{\mathrm{ts}}(Y;m_{\mu_0},\sigma_{\mu_0}) \to 1$, if and only if $\pnorm{\mu-\mu_0}{}\gg p^{1/4}$. 
	\end{enumerate}
\end{theorem}

The proof of Theorem \ref{thm:lrt_clt_global_testing_lasso} relies on the following proposition, which may be of independent interest (see Section~\ref{pf:lrt_clt_global_testing_lasso} for a proof).

\begin{proposition}\label{prop:lasso_estimates}
	The following hold:
	\begin{enumerate}
		\item $\pnorm{\E_{\mu_0} J_{\hat{\mu}_{K_{X,\lambda}}}}{F}^2 \geq p/2-4 \big(n \mathfrak{p}_{\lambda,\mu_0}\big)^2$.
		\item For any $\mu\in K_{X,\lambda}$, $\bigabs{\E_\mu \dv \hat{\mu}_{K_{X,\lambda}}-p }\leq 2p\cdot  \mathfrak{p}_{\lambda,\mu}$.
		\item For any $\mu\in K_{X,\lambda}$, $\bigabs{\E_\mu\pnorm{\hat{\mu}_{K_{X,\lambda}}-\mu}{}^2-p}\leq C n \mathfrak{p}_{\lambda,\mu}^{1/2}$.
	\end{enumerate}
	Here $C>0$ is an absolute constant.
\end{proposition}

The proof of the above proposition makes essential use of an explicit representation of the Jacobian $J_{\hat{\mu}_{K_{X,\lambda}}}$ derived in \cite{kato2009degrees}, which complements its analogues for Lasso in the penalized form derived in \cite{zou2007degrees,tibshirani2012degrees}.

\begin{remark}
	A few remarks are in order.
	\begin{enumerate}
		\item (\textit{Choice of $\lambda$}) To apply Theorem \ref{thm:lrt_clt_global_testing_lasso}, we need to control the probability term $ \mathfrak{p}_{\lambda,\mu}$ for a generic $\mu = X\theta\in K_{X,\lambda}$. This can be done via the following exponential inequality (see Lemma \ref{lem:l1_deviation}): for any $t\geq 1$,
		\begin{align*}
		\Prob_\mu\bigg(\|\hat{\theta}^0\|_1 \geq \pnorm{\theta}{1} +t \sqrt{\frac{p}{n\lambda_{\min}(\Sigma)}}\bigg) \leq e^{-t^2/C}.
		\end{align*}
		Here $C > 0$ is a universal constant and $\lambda_{\min}(\Sigma)$ is the smallest eigenvalue of $\Sigma$. Therefore, for any choice of the tuning parameter $\lambda$ satisfying
		\begin{align}\label{ineq:lambda_lasso}
		\lambda \geq \pnorm{\theta_0}{1}+ r_n,\textrm{ with } r_n\equiv C \sqrt{p \log n/\big(n\lambda_{\min}(\Sigma)\big)}
		\end{align}
	    for a large enough constant $C>0$, we have $n\cdot(\mathfrak{p}_{\lambda,\mu}^{1/2}\vee \mathfrak{p}_{\lambda,\mu_0}^{1/2}) = \mathfrak{o}(1)$ uniformly in $\mu \in K_{X,\lambda-r_n }$.  Hence Theorem \ref{thm:lrt_clt_global_testing_lasso} yields that the LRT is asymptotically size $\alpha$ and power consistent for all such prescribed $\mu$'s if and only if $\pnorm{\mu-\mu_0}{}\gg p^{1/4}$. To get some intuition for this result, for the tuning parameters $\lambda$ satisfying (\ref{ineq:lambda_lasso}) above, the proof of Proposition \ref{prop:lasso_estimates} shows that the Jacobian $J_{\hat{\mu}_{K_{X,\lambda}}}$ of $\hat{\mu}_{K_{X,\lambda}}$ (cf. Equation (\ref{ineq:jacobian_lasso})) is close to that of the least squares estimator $\hat{\theta}^0$ with high probability. From this perspective, the separation rate $p^{1/4}$ is quite natural under (\ref{ineq:lambda_lasso})  in view of Proposition \ref{prop:global_testing_subspace_toy} with $\rank(X)=p$ in the current setting.
		\item (\textit{Lasso in penalized form}) Theorem \ref{thm:lrt_clt_global_testing_lasso} is applicable for Lasso in its constrained form as defined in (\ref{def:lasso_constrained}). The penalized form of Lasso
		\begin{align}\label{def:lasso_pen}
		\hat{\theta}_{\mathrm{pen}}(\tau)\equiv \argmin_{\theta \in \R^p} \bigg[\frac{1}{2}\pnorm{Y-X\theta}{}^2+\tau\pnorm{\theta}{1}\bigg],
		\end{align}			
		however, does not fit into our general testing framework (\ref{eq:testing_generic}), and therefore there is no natural associated `likelihood ratio test'. An interesting problem is to study the behavior of the statistic $T(Y)$ defined in (\ref{def:TY}) with $\hat{\mu}_K$ replaced by the penalized Lasso estimator $\hat{\mu}_{\textrm{pen}}(\tau)\equiv X \hat{\theta}_{\textrm{pen}}(\tau)$. The major hurdle here is to, as in Proposition \ref{prop:lasso_estimates}-(1), evaluate a lower bound for the Frobenius norm of the the expected Jacobian $\E J_{\hat{\mu}_{\mathrm{pen}}(\tau)}=\E X_{\hat{S}(\tau)}(X_{\hat{S}(\tau)}^\top X_{\hat{S}(\tau)})^{-1} X_{\hat{S}(\tau)}$ (see e.g. \cite[Proposition 3.10]{bellec2018second}), where $\hat{S}(\tau)$ is the (random) support of $\hat{\theta}_{\mathrm{pen}}(\tau)$. Although the penalized form (\ref{def:lasso_pen}) is known to be `equivalent' to the constrained form (\ref{def:lasso_constrained}) in that for each given $\tau>0$, there exists some data-dependent $\lambda=\lambda(\tau,X,Y)>0$ such that $\hat{\theta}(\lambda)= \hat{\theta}_{\mathrm{pen}}(\tau)$, due to the random correspondence of $\tau$ and $\lambda$, the techniques used to prove Proposition \ref{prop:lasso_estimates} do not translate to a lower bound for $\pnorm{\E J_{\hat{\mu}_{\mathrm{pen}}(\tau)}}{F}^2$.  We leave this for a future study.

	\end{enumerate}
\end{remark}

\subsection{Testing parametric assumptions versus shape-constrained alternatives}\label{subsec:example_subspace}

For fixed $k\in \mathbb{Z}_{\geq 0}$ and $n\geq k+2$, and consider the testing problem
\begin{align}\label{eq:testing_k_monotone}
H_0: \mu \in K_{0,k}\quad \textrm{versus}\quad H_1: \mu \in K_{\uparrow,k}.
\end{align}
Here  $K_{\uparrow,k}\equiv \{\mu \in \R^n: \nabla^{k+1}\mu\geq 0\}$ and $K_{0,k}\equiv \{\mu \in \R^n: \nabla^{k+1}\mu= 0\}$, with $\nabla:\R^n\to \R^{n-1}$ denoting the difference operator defined by $\nabla (\mu_i)_{i=1}^n \equiv (\mu_{i+1}-\mu_i)_{i=1}^{n-1}$, and $\nabla^{k+1} \equiv \nabla\circ\cdots \circ\nabla: \R^n \to \R^{n-k-1}$ with $k+1$ compositions. It can be readily verified that $K_{0,k}$ is a subspace of dimension $k+1$, $K_{\uparrow,k}$ is a closed and convex cone, and $K_{0,k}\subset K_{\uparrow,k}\subset \R^n$. Hence (\ref{eq:testing_k_monotone}) is a special case of the general testing problem (\ref{eq:testing_special}).

Testing a parametric model against a nonparametric alternative has previously been studied in \cite{cox1988testing,eubank1990testing,azzalini1993use,hardle1993comparing,stute1997nonparametric,fan2001goodness,guerre2005data,christensen2010alternative,neumeyer2010estimating,sen2017testing} among which the shape-constrained alternatives in (\ref{eq:testing_k_monotone}) are sometimes preferred since the model fits therein usually do not involve the choice of tuning parameters. In particular: 
\begin{enumerate}
	\item When $k=0$, (\ref{eq:testing_k_monotone}) becomes:
	\begin{align*}
	H_0: \mu \hbox{ is `constant'}, \qquad \textrm{versus}\qquad H_1: \mu \hbox{ is `monotone'}.
	\end{align*}
	\item When $k=1$, (\ref{eq:testing_k_monotone}) becomes:
	\begin{align*}
	H_0: \mu \hbox{ is `linear'}, \qquad \textrm{versus}\qquad H_1: \mu \hbox{ is `convex'}.
	\end{align*}
\end{enumerate}
The above two settings have previously been considered in~\cite{bartholomew1959test,bartholomew1959test2,robertson1988order,sen2017testing}.

\begin{theorem}\label{thm:clt_subspace_1d}
	Fix $k\in \mathbb{Z}_{\geq 0}$. Consider testing~\eqref{eq:testing_k_monotone} using the two-sided LRT $\Psi_{\mathrm{ts}} (Y;m_0,\sigma_0)$, as  in~\eqref{def:LRT_two_sided}.
	\begin{enumerate}
		\item There exists a constant $C>0$, depending on $k$ only, such that for $\mu \in K_{0,k}$, 
		\begin{align*}
		d_{\mathrm{TV}}\bigg( \frac{T(Y)-m_0}{\sigma_0 },\mathcal{N}(0,1)\bigg)
		&\leq \frac{C}{\bm{1}_{k=0}\sqrt{\log (en)}+\bm{1}_{k\geq 1}\sqrt{\log \log (16n)}}.
		\end{align*}
		\sloppy Consequently for $\mu \in K_{0,k}$, the LRT is asymptotically size $\alpha$ with $\E_{\mu} \Psi_{\mathrm{ts}} (Y;m_0,\sigma_0)=\alpha+\mathcal{O}\big(\bm{1}_{k=0}(\log (en))^{-1/2}+\bm{1}_{k\geq 1} (\log \log (16n))^{-1/2}\big)$.
		\item For $\mu \in K_{\uparrow,k}$ with $\pnorm{\mu-\Pi_{K_{0,k}}(\mu)}{}\gg \log^{1/4} (en)$, the LRT is power consistent under $\mu$, i.e., $\E_{\mu} \Psi_{\mathrm{ts}} (Y;m_0,\sigma_0)\to 1$.
	\end{enumerate}
\end{theorem}

The key step in the proof of Theorem \ref{thm:clt_subspace_1d} (proved in Section~\ref{pf:clt_subspace_1d}) is to obtain the correct order of the statistical dimension $\delta_{K_{\uparrow,k}}$. The discrepancy between $k = 0$ and $k\geq 1$ in claim (1) is due to the fact that while a universal upper bound of the order $\log (en)$ can be proved for any fixed $k\geq 0$, only a lower bound of the order $\log\log(16 n)$ can be proved for $k\geq 1$. We conjecture that the correct order of $\delta_{K_{\uparrow,k}}$ should be $\log (en)$ for all fixed $k\geq 0$.

The above theorem can be easily extended to the multi-dimensional analogue of (\ref{eq:testing_k_monotone}) in the context of, e.g., testing constancy versus coordinate-wise monotonicity, linearity versus multi-dimensional convexity, by using results of \cite{han2017isotonic,kur2020convex}; we omit the details here.

\section{Proofs of results in Section~\ref{section:clt_LRT}}\label{section:proof_clt}

\subsection{Proof of Theorem \ref{thm:lrt_clt_global_testing}}\label{pf:lrt_clt_global_testing}

We need the following proposition, which can be proved using techniques similar to \cite[Theorem 2.1]{goldstein2017gaussian}. We provide the details of its proof in Appendix \ref{section:proof_lrt_clt} for the convenience of the reader.

\begin{proposition}\label{prop:lrt_clt}
	Suppose $K_0,K$ are two non-empty closed convex sets in $\R^n$. Let
	\begin{align*}
	T_{K_0,K}(y)\equiv \pnorm{y - \Pi_{K_0}(y)}{}^2-\pnorm{y - \Pi_{K}(y)}{}^2.
	\end{align*}
	Then for any $\mu\in \R^n$, under the model (\ref{model:sequence}),
	\begin{align*}
	d_{\mathrm{TV}}\bigg( \frac{T_{K_0,K}(Y)-\E_\mu T_{K_0,K}(Y) }{ \sqrt{\var_\mu (T_{K_0,K}(Y))} },\mathcal{N}(0,1)\bigg)\leq \frac{16 \sqrt{\E_\mu \pnorm{\hat{\mu}_{K}-\hat{\mu}_{K_0}}{}^2} }{\var_\mu(T_{K_0,K}(Y))}.
	\end{align*}
\end{proposition}

The next lemma provides a lower bound for the variance of $F(\xi)$, where the absolute continuity of $F:\R^n \to \R$ is valid up to its first derivatives. The proof is based on Fourier analysis in the Gaussian space in the spirit of~\cite[Proposition 1.5.1]{nourdin2012normal}. 

\begin{lemma}\label{lem:var_lower_bound_fourier}
	Let $F:\R^n\to \R$ be such that $\{\partial_{\bm{k}} F: \abs{\bm{k}}\leq 1\}$ are absolutely continuous and $\{\partial_{\bm{k}} F: \abs{\bm{k}}\leq 2\}$ have sub-exponential growth at $\infty$. Then  
	\begin{align*}
	\var(F(\xi))\geq \sum_i\big(\E \partial_i F(\xi)\big)^2+\sum_{i\neq j} \big(\E \partial_{ij} F(\xi)\big)^2+\frac{1}{2}\sum_i \big(\E \partial_{ii} F(\xi)\big)^2.
	\end{align*}
\end{lemma}
\begin{proof}
	We only need to verify the above claimed inequality for $\E F(\xi)=0$. 
	Let $H_k(x) = (-1)^k e^{x^2/2} \frac{\mathrm{d}^k}{\mathrm{d}{x^k}}e^{-x^2/2}$ be the Hermite polynomial of order $k$. For a multi-index $\bm{k}=(k_1,\ldots,k_n)$ and $y \in \R^n$, let $H_{\bm{k}}(y)\equiv \prod_{i=1}^n H_{k_i}(y_i)$. Then $\{H_{\bm{k}}: \bm{k} \in \mathbb{Z}_{\geq 0}^n\}$ is a complete orthogonal basis of $L_2(\gamma_n)$, where $\gamma_n$ is the standard Gaussian measure on $\R^n$. On the other hand, the absolute continuity and growth condition on $F$ ensures the validity of the following Gaussian integration-by-parts: For all multi-indices $\bm{k}$ such that $\abs{\bm{k}}\leq 2$,
	\begin{align*}
	\E\big[F(\xi)H_{\bm{k}}(\xi)\big] = \E \partial_{\bm{k}} F(\xi).
	\end{align*}
	As $\E \abs{H_{\bm{k}}(\xi)}^2=\bm{k}!$, it follows by Plancherel's theorem that
	\begin{align*}
	\var(F(\xi)) = \E F^2(\xi) \geq \sum_{ \bm{k}: \abs{\bm{k}}\leq 2} \frac{ \big(\E F(\xi) H_{\bm{k}}(\xi) \big)^2 }{\E \abs{H_{\bm{k}}(\xi)}^2 },
	\end{align*}
	which equals the right hand side of the claimed inequality.
\end{proof}

\begin{proof}[Proof of Theorem \ref{thm:lrt_clt_global_testing}]
	Let 
	\begin{align*}
	F(\xi)\equiv T(\mu_0+\xi) = \pnorm{\mu_0+\xi-\mu_0}{}^2 - \pnorm{\mu_0+\xi-\Pi_K(\mu_0+\xi)}{}^2.
	\end{align*}
	By Lemma \ref{lem:proj_basic}-(1),
	\begin{align*}
	\nabla F(\xi)=\nabla T (\mu_0+\xi) = 2\big(\Pi_K(\mu_0+\xi)-\mu_0\big).
	\end{align*}
	Hence
	\begin{align*}
	\partial_{ij} F(\xi)=\partial_{ij} T(\mu_0+\xi) =2 (J_{\Pi_K}(\mu_0+\xi))_{ji}.
	\end{align*}
	We verify that $F$ satisfies the condition of Lemma \ref{lem:var_lower_bound_fourier}. By the above closed-form expression of $F$ and $\nabla F$, the absolute continuity for $\{\partial_{\bm{k}} F: \abs{\bm{k}}\leq 1\}$ holds by noting that $\nabla F$ is $2$-Lipschitz. On the other hand, as
	\begin{align*}
	\abs{F(\xi)}&=\big\lvert \pnorm{\xi}{}^2 - \pnorm{\mu_0+\xi-\Pi_K(\mu_0+\xi)}{}^2 \big\lvert \leq C\cdot \big(\pnorm{\xi}{}^2\vee \pnorm{\mu_0}{}^2 \big),\\
	\pnorm{\nabla F(\xi)}{}& \leq C\cdot \big(\pnorm{\mu_0}{}\vee \pnorm{\xi}{}\big),\qquad \pnorm{\nabla^2 F(\xi)}{} = 2\pnorm{J_{\Pi_K}(\mu_0+\xi)^\top}{}\leq 2,
	\end{align*}
	it follows that $\{\partial_{\bm{k}} F: \abs{\bm{k}}\leq 2\}$ have sub-exponential growth at $\infty$. Now we may apply Lemma \ref{lem:var_lower_bound_fourier} to see that
	\begin{align*}
	\sigma_{\mu_0}^2=\var(T(Y))&\geq \sum_i \big(\E \partial_i F(\xi)\big)^2+\frac{1}{2}\sum_{i,j} \big(\E \partial_{ij} F(\xi)\big)^2\\
	& = 4 \pnorm{\E \Pi_K(\mu_0+\xi)-\mu_0}{}^2+ 2\sum_{i,j} \big(\E J_{\Pi_K}(\mu_0+\xi)\big)_{ij}^2,
	\end{align*}
	as desired. The claim of the theorem now follows from Proposition \ref{prop:lrt_clt}.
\end{proof}

\subsection{Proof of Theorem  \ref{thm:power_lrt_global_testing}}\label{section:proof_master_thm}

A simple but important observation in the proof of Theorem \ref{thm:power_lrt_global_testing} is the following.
\begin{proposition}\label{prop:concentration_mean_shift}
	Let
	\begin{align*}
	Z(\mu,\mu_0)&\equiv \Delta T_{\mu,\mu_0}(\xi)- \E(\Delta T_{\mu,\mu_0})\\
	&=T(\mu+\xi)-T(\mu_0+\xi)-(m_\mu-m_{\mu_0}).
	\end{align*}
	Then for any $t\geq 0$,
	\begin{align*}
	\Prob\big(Z(\mu,\mu_0)>t\big)\vee \Prob\big(Z(\mu,\mu_0)<-t\big)\leq  \exp\bigg(-\frac{t^2}{8\pnorm{\mu-\mu_0}{}^2}\bigg).
	\end{align*}
\end{proposition}
\begin{proof}
	As 
	\begin{align*}
	\nabla T(y) = \nabla \big(\pnorm{y-\mu_0}{}^2-\pnorm{y-\Pi_K(y)}{}^2\big)=2(\Pi_K(y)-\mu_0)
	\end{align*}
	by Lemma \ref{lem:proj_basic}-(1), it follows that 
	\begin{align*}
	\pnorm{\nabla_\xi Z(\mu,\mu_0)}{}& = 2\pnorm{\Pi_K(\mu+\xi)-\Pi_K(\mu_0+\xi)}{}\leq 2\pnorm{\mu-\mu_0}{}. 
	\end{align*}
	The claim now follows by Gaussian concentration inequality for Lipschitz functions, cf. \cite[Theorem 5.6]{bousquet2003concentration}.
\end{proof}

\begin{lemma}\label{lem:normal_mean_multi}
	For any $t \in \R$, there exists some $C_t>0$ such that for all $u \in \R,\eta \in [-1/2,1/2]$, 
	\begin{align*}
	\bigabs{ \Prob\big( \mathcal{N}(u,1)\leq t\big)-\Prob\big( \mathcal{N}((1+\eta)u,1)\leq t\big)}\leq C_t\cdot |\eta|.
	\end{align*}
	Furthermore, $\sup_{t\in M}C_t < \infty$ for any compact subset $M$ of $\R$.
\end{lemma}
\begin{proof}
	We assume $\eta\geq0$ without loss of generality. Note that with $\varphi$ denoting the d.f. for standard normal,
	\begin{align*}
	\Prob\big( \mathcal{N}(0,1)\leq t-(1+\eta)u\big) & \leq \Prob\big(\mathcal{N}(0,1)\leq t-u\big)+\eta\cdot \sup_{v \in [(t-u)\pm \eta \abs{u}]}\varphi(v)  \abs{u}\\
	&\leq \Prob\big(\mathcal{N}(0,1)\leq t-u\big) + \eta\cdot C_t,
	\end{align*}
	where $C_t\equiv \sup_u \sup_{v \in [(t-u)\pm (\abs{u}/2)]}\varphi(v)  \abs{u}<\infty$ depends on $t$ only.
\end{proof}

\begin{proof}[Proof of Theorem \ref{thm:power_lrt_global_testing}]
	First note that under the model (\ref{model:sequence}), the normalized LRS $T(Y)$ satisfies the decomposition
	\begin{align}\label{ineq:decomposition_power_lrt_generic}
	\frac{T (Y)-m_{\mu_0}}{ \sigma_{\mu_0} } & = \frac{T(\mu+\xi)-T(\mu_0+\xi)}{\sigma_{\mu_0}}+\frac{T(\mu_0+\xi)-m_{\mu_0}}{\sigma_{\mu_0}}.
	\end{align}
	Using Proposition \ref{prop:concentration_mean_shift}, on an event $E_u$ with $\Prob(E_u)\geq 1-2e^{-u^2}$, we have $\abs{Z(\mu,\mu_0)}\leq 3u\pnorm{\mu-\mu_0}{}$ with $Z(\mu,\mu_0)$ defined therein. Then for any $t \in \R$, 
	\begin{align}\label{ineq:power_expansion_generic_1}
	&\Prob\bigg(\frac{T (\mu+\xi)-m_{\mu_0}}{ \sigma_{\mu_0} } \leq t\bigg) \nonumber\\
	&=\Prob\bigg(\frac{m_\mu-m_{\mu_0}+Z(\mu,\mu_0)}{\sigma_{\mu_0}}+\frac{T(\mu_0+\xi)-m_{\mu_0}}{\sigma_{\mu_0}} \leq t\bigg)\nonumber\\
	& \leq \Prob\bigg(\frac{m_\mu-m_{\mu_0}-3u\pnorm{\mu-\mu_0}{}}{\sigma_{\mu_0}}+\mathcal{N}(0,1) \leq t\bigg) +2e^{-u^2}+\textrm{err}_{\mu_0}\\
	&= \Prob\bigg(\frac{m_\mu-m_{\mu_0}}{\sigma_{\mu_0}}\big(1+\eta(u)\big)+\mathcal{N}(0,1) \leq t\bigg) +2e^{-u^2}+\textrm{err}_{\mu_0},\nonumber
	\end{align}
	where
	\begin{align*}
	\eta(u) \equiv  -3u\cdot \frac{ \pnorm{\mu-\mu_0}{} }{ m_\mu-m_{\mu_0}}. 
	\end{align*}
	By choosing $u \leq \abs{m_\mu-m_{\mu_0}}/\big(6\pnorm{\mu-\mu_0}{}\big) $, we have $\abs{\eta(u)}\leq 1/2$, so we may apply Lemma \ref{lem:normal_mean_multi} to see that,
	\begin{align*}
	\Delta^\ast &\equiv \Prob\bigg(\frac{T (\mu+\xi)-m_{\mu_0}}{ \sigma_{\mu_0} } \leq t\bigg)-\Prob\bigg(\frac{m_\mu-m_{\mu_0}}{\sigma_{\mu_0}}+\mathcal{N}(0,1) \leq t\bigg) \\
	&\leq 2e^{-u^2}+ C_t u\cdot \frac{ \pnorm{\mu-\mu_0}{} }{ \abs{m_\mu-m_{\mu_0}} }+\textrm{err}_{\mu_0}.
	\end{align*}
	Optimizing $u \leq  \abs{m_\mu-m_{\mu_0}}/\big(6\pnorm{\mu-\mu_0}{})$, the first two terms in the error bound above can be bounded, up to an absolute constant, by
	\begin{align*}
	(1\vee C_t)\cdot \mathscr{L}\bigg( 1\bigwedge\frac{\pnorm{\mu-\mu_0}{}}{\abs{m_\mu-m_{\mu_0}}} \bigg).
	\end{align*}
	Next we will obtain a similar upper bound for $\Delta^\ast$, but replacing $\abs{m_\mu-m_{\mu_0}}$ in the above display by $\sigma_{\mu_0}$. To see this, (\ref{ineq:power_expansion_generic_1}) along with
	\begin{align*}
	&\Prob\bigg(\frac{m_\mu-m_{\mu_0}-3u\pnorm{\mu-\mu_0}{}}{\sigma_{\mu_0}}+\mathcal{N}(0,1) \leq t\bigg)\\
	&\quad \leq \Prob\bigg(\frac{m_\mu-m_{\mu_0}}{\sigma_{\mu_0}}+\mathcal{N}(0,1) \leq t\bigg) + \|\varphi\|_\infty\cdot\frac{3u\pnorm{\mu-\mu_0}{}}{\sigma_{\mu_0}}
	\end{align*}
	yields that
	\begin{align*}
	\Delta^\ast&\leq \inf_{u>0} \bigg\{2e^{-u^2}+ 3 u\cdot \frac{ \pnorm{\mu-\mu_0}{} }{\sigma_{\mu_0} }\bigg\}+\textrm{err}_{\mu_0}\leq C\cdot \mathscr{L}\bigg( 1\bigwedge\frac{\pnorm{\mu-\mu_0}{}}{\sigma_{\mu_0} } \bigg)+ \textrm{err}_{\mu_0}.
	\end{align*}
	Similar lower bounds can be derived. Applying the above arguments to the (at most 2) end point(s) of $\mathcal{A}_\alpha$ proves the inequality (\ref{ineq:power_lrt_global_testing_expansion}). Now (1) is a direct consequence of (\ref{ineq:power_lrt_global_testing_expansion}), while (2) follows by further noting $\Delta_{A_\alpha}(w_n)\to \beta$ if and only if all limit points of the sequence $\{w_n\}$ are contained in $\Delta_{A_\alpha}^{-1}(\beta)$. 
\end{proof}

\subsection{Proof of Theorem \ref{thm:lrt_clt_pivotal}}\label{pf:lrt_clt_pivotal}
By Lemma \ref{lem:invariance_lrt}, we only need to consider $\mu=0$. Note that: (i) $\pnorm{\xi-\Pi_{K'}(\xi)}{}^2 =\pnorm{\xi}{}^2- \pnorm{\Pi_{K'}(\xi)}{}^2$ for $K' \in \{K_0,K\}$, (ii) $(K_0,K)$ is a non-oblique pair of closed convex cones in that $\Pi_{K_0} = \Pi_{K_0}\circ \Pi_K$, so $\Pi_K(\xi) = \Pi_{K_0}(\xi)+ \Pi_{K\cap K_0^\ast} (\xi)$ with $\iprod{\Pi_{K_0}(\xi)}{\Pi_{K\cap K_0^\ast} (\xi)}=0$ (cf.~\cite[Equation (25)]{wei2019geometry}), and hence $\pnorm{\Pi_{K}(\xi)}{}^2 = \pnorm{\Pi_{K_0}(\xi)}{}^2+ \pnorm{\Pi_{K\cap K_0^\ast} (\xi)}{}^2$. Thus,
\begin{align*}
\E \pnorm{\Pi_K(\xi)-\Pi_{K_0}(\xi)}{}^2& = \E \pnorm{ \Pi_{K\cap K_0^\ast} (\xi)}{}^2\\
& = \E \big[\pnorm{\Pi_{K}(\xi)}{}^2 -\pnorm{\Pi_{K_0}(\xi)}{}^2 \big] = \delta_K-\delta_{K_0},
\end{align*}
and
\begin{align*}
\sigma_0^2 &= \var\big(\pnorm{\Pi_{K_0}(\xi)}{}^2-\pnorm{\Pi_{K}(\xi)}{}^2\big) \\
&= \var \big(\pnorm{\Pi_{K\cap K_0^\ast} (\xi)}{}^2 \big)\stackrel{(\ast)}{\geq} 2 \delta_{K\cap K_0^\ast} = 2\big(\delta_K-\delta_{K_0}\big).
\end{align*}
Here the inequality $(\ast)$ follows by Lemma \ref{lem:var_proj}-(2). The claim now follows from Proposition \ref{prop:lrt_clt}. \qed

\subsection{Proof of Theorem~\ref{thm:power_lrt_subspace}}\label{pf:power_lrt_subspace} 
First note that we have the decomposition
\begin{align}\label{ineq:decomposition_power_lrt}
\frac{T (\mu+\xi)-m_0}{ \sigma_0 } & = \frac{T(\mu+\xi)-T(\xi)}{\sigma_{0}}+\frac{T(\xi)-m_0}{\sigma_0}. 
\end{align}
As
\begin{align*}
m_\mu& = \E \big[\pnorm{\mu+\xi-\Pi_{K_0}(\mu+\xi)}{}^2-\pnorm{\mu+\xi-\Pi_K(\mu+\xi)}{}^2\big]\\
& = \E \bigg[ \pnorm{\Pi_{K_0}(\mu+\xi)-\mu}{}^2-2\iprod{\xi}{\Pi_{K_0}(\mu+\xi)}+\pnorm{\xi}{}^2\\
&\qquad\qquad  -\bigg(\pnorm{\Pi_{K}(\mu+\xi)-\mu}{}^2-2\iprod{\xi}{\Pi_{K}(\mu+\xi)}+\pnorm{\xi}{}^2\bigg) \bigg]\\
& = \bigg\{\pnorm{\mu-\Pi_{K_0}(\mu)}{}^2 + 2\E\iprod{\xi}{\Pi_{K}(\mu+\xi)}-\E \pnorm{\Pi_{K}(\mu+\xi)-\mu}{}^2\bigg\} - \delta_{K_0},
\end{align*}
we have (as $\delta_K = \E \pnorm{\Pi_K(\xi)}{}^2 = \E\iprod{\xi}{\Pi_K(\xi)}$)
\begin{align*}
m_\mu- m_0 & = \E \bigg[2\iprod{\xi}{\Pi_{K}(\mu+\xi)}- \pnorm{\Pi_{K}(\mu+\xi)-\mu}{}^2-\iprod{\xi}{\Pi_{K}(\xi)}\bigg] +\pnorm{\mu-\Pi_{K_0}(\mu)}{}^2\\
& = \E \bigg[2\iprod{\xi}{\Pi_{K}(\mu-\Pi_{K_0}(\mu)+\xi)}\\
&\qquad\qquad - \pnorm{\Pi_{K}(\mu-\Pi_{K_0}(\mu)+\xi)-\big(\mu-\Pi_{K_0}(\mu)\big)}{}^2-\iprod{\xi}{\Pi_{K}(\xi)}\bigg]\\
&\qquad\qquad +\pnorm{\mu-\Pi_{K_0}(\mu)}{}^2\qquad\qquad\qquad\qquad \hbox{ (by Lemma \ref{lem:invariance_lrt})}\\
& = \E \bigg[2\iprod{\mu-\Pi_{K_0}(\mu)+\xi}{\Pi_{K}(\mu-\Pi_{K_0}(\mu)+\xi)}\\
&\qquad\qquad - \pnorm{\Pi_{K}\big(\mu-\Pi_{K_0}(\mu)+\xi\big)}{}^2-\pnorm{\mu-\Pi_{K_0}(\mu)}{}^2-\iprod{\xi}{\Pi_{K}(\xi)}\bigg]\\
&\qquad\qquad +\pnorm{\mu-\Pi_{K_0}(\mu)}{}^2\\
& = \E \pnorm{\Pi_{K}\big(\mu-\Pi_{K_0}(\mu)+\xi\big)}{}^2 - \E \pnorm{\Pi_{K}(\xi)}{}^2 = \Gamma_{K,2}(\mu-\Pi_{K_0}(\mu)).
\end{align*}
Here in the last line of the above display we used that 
\begin{align*}
\E \iprod{\mu-\Pi_{K_0}(\mu)+\xi}{\Pi_{K}(\mu-\Pi_{K_0}(\mu)+\xi)}&= \E \pnorm{\Pi_{K}\big(\mu-\Pi_{K_0}(\mu)+\xi\big)}{}^2,\\
\E \iprod{\xi}{\Pi_{K}(\xi)} &= \E \pnorm{\Pi_{K}(\xi)}{}^2.
\end{align*}
Let
\begin{align*}
Z_0(\mu)\equiv T(\mu+\xi)-T(\xi)-(m_\mu-m_0).
\end{align*}
As $\nabla T(y) = \nabla\big(\pnorm{y-\Pi_{K_0}(y)}{}^2-\pnorm{y-\Pi_K(y)}{}^2\big)=2(\Pi_K(y)-\Pi_{K_0}(y))$ by Lemma \ref{lem:proj_basic}-(1),
\begin{align*}
\nabla_\xi Z_0(\mu)
&= 2\big(\Pi_K(\mu+\xi)-\Pi_K(\xi)\big) - 2\big(\Pi_{K_0}(\mu+\xi)-\Pi_{K_0}(\xi)\big)\\
&= 2\big(\Pi_K(\mu-\Pi_{K_0}(\mu)+\xi)-\Pi_K(\xi)\big) \\
&\qquad- 2\big(\Pi_{K_0}(\mu-\Pi_{K_0}(\mu)+\xi)-\Pi_{K_0}(\xi)\big),\qquad \hbox{(by Lemma \ref{lem:invariance_lrt})}
\end{align*}
and hence
\begin{align*}
\pnorm{\nabla_\xi Z_0(\mu)}{}\leq 4\pnorm{\mu-\Pi_{K_0}(\mu)}{}.
\end{align*}
Now using the Gaussian concentration inequality for Lipschitz functions, cf. \cite[Theorem 5.6]{boucheron2013concentration}, it holds for any $t> 0$ that
\begin{align*}
\Prob\big(Z_0(\mu)>t\big)\vee \Prob\big(Z_0(\mu)<-t\big)\leq  \exp\bigg(-\frac{t^2}{32\pnorm{\mu-\Pi_{K_0}(\mu)}{}^2}\bigg).
\end{align*}
From here we may conclude (\ref{ineq:normal_power_expansion_subspace}) by using similar arguments as in the proof of Theorem \ref{thm:power_lrt_global_testing}. Furthermore, by the proof of \cite[Lemma E.1]{wei2019geometry}, $\Gamma_{K,2}(\nu)\geq \pnorm{\nu}{}^2\geq 0$ for any $\nu\in K$, so
\begin{align*}
\frac{\pnorm{\mu-\Pi_{K_0}(\mu) }{}}{ \bigabs{\Gamma_{K,2}\big(\mu-\Pi_{K_0}(\mu)\big)} \vee \sigma_0  } \stackrel{(\ast)}{\leq} \sup_{\nu \in K} \frac{\pnorm{\nu }{}}{ \Gamma_{K,2}(\nu)\vee \sigma_0  } \leq \sup_{\nu \in K} \frac{1}{ \pnorm{\nu}{} \vee (\sigma_0/\pnorm{\nu}{}) } \stackrel{(\ast\ast)}{\leq} \frac{1}{\sigma_0^{1/2}}.
\end{align*}
The inequality $(\ast)$ follows as $\mu- \Pi_{K_0}(\mu) \in K$ for $\mu \in K$, and $(\ast\ast)$ follows as $\inf_{\nu \in K} \big\{\pnorm{\nu}{} \vee (\sigma_0/\pnorm{\nu}{}) \big\} \geq \inf_{t\geq 0} \big\{t \vee (\sigma_0/t)\} =\sigma_0^{1/2}$. As $\sigma_0^2 \asymp \delta_K-\delta_{K_0}$, the second inequality (\ref{ineq:normal_power_expansion_subspace_1}) follows by the bound $\textrm{err}_0\leq 8\big(\delta_K - \delta_{K_0} \big)^{-1/2}$ via Theorem \ref{thm:lrt_clt_pivotal}.

Note that (1) is a direct consequence of (\ref{ineq:normal_power_expansion_subspace}) (as $\mathrm{err}_0$ can be bounded above by Theorem~\ref{thm:lrt_clt_pivotal} and $\mathscr{L}(0)=0$) so we prove (2) below. To see the claimed power characterization, note that
\begin{align*}
&\frac{ \E \pnorm{\Pi_{K}\big(\mu-\Pi_{K_0}(\mu)+\xi\big)}{}^2 - \E \pnorm{\Pi_{K}(\xi)}{}^2}{\sigma_{0}} \\
& = \frac{ \big(\E \pnorm{\Pi_{K}\big(\mu-\Pi_{K_0}(\mu)+\xi\big)}{}\big)^2-\big(\E \pnorm{\Pi_{K}(\xi)}{}\big)^2  }{\sigma_0 } + \mathcal{O}(\sigma_0^{-1})\\
& = \big(\E \pnorm{\Pi_{K}\big(\mu-\Pi_{K_0}(\mu)+\xi\big)}{} -\E \pnorm{\Pi_{K}(\xi)}{}  \big)\\
&\qquad\qquad \times \bigg[ \frac{  2\E \pnorm{\Pi_K(\xi)}{} }{\sigma_0} +\frac{\E \pnorm{\Pi_{K}\big(\mu-\Pi_{K_0}(\mu)+\xi\big)}{}-\E \pnorm{\Pi_K(\xi)}{} }{ \sigma_0}\bigg] + \mathcal{O}(\sigma_0^{-1})\\
& =  \Gamma_K\big(\mu-\Pi_{K_0}(\mu)\big) \bigg[2\sqrt{ \frac{\delta_K+\mathcal{O}(1)}{2\big(\delta_K-\delta_{K_0}\big) + \var\big(V_{K\cap K_0^*}\big) } }\\
&\qquad\qquad\qquad\qquad\qquad\qquad+\frac{\Gamma_K\big(\mu-\Pi_{K_0}(\mu)\big) }{\sigma_0}\bigg] + \mathcal{O}(\sigma_0^{-1})\\
& =  \Gamma_K\big(\mu-\Pi_{K_0}(\mu)\big)  \bigg[2\sqrt{ \frac{1+\mathcal{O}(\delta_K^{-1})}{ \big(2+ \var\big(V_{K\cap K_0^*}\big)/\delta_{K\cap K_0^\ast}\big)\cdot \big(1-\delta_{K_0}/\delta_K\big) } }\\
&\qquad\qquad\qquad\qquad\qquad\qquad +\frac{\Gamma_K\big(\mu-\Pi_{K_0}(\mu)\big) }{\sigma_0}\bigg] + \mathcal{O}(\sigma_0^{-1}).
\end{align*}
Under the growth condition $\sigma_0\to\infty$, direct calculation now entails that
\begin{align*}
&\frac{ \E \pnorm{\Pi_{K}\big(\mu-\Pi_{K_0}(\mu)+\xi\big)}{}^2 - \E \pnorm{\Pi_{K}(\xi)}{}^2}{\sigma_{0}}  \to w^\ast \in [0,+\infty] \\
\Leftrightarrow\quad &\frac{2\Gamma_K\big(\mu-\Pi_{K_0}(\mu)\big) }{ \sqrt{2+\var(V_{K\cap K_0^*})/\delta_{K\cap K_0^*}} \sqrt{1-\delta_{K_0}/\delta_K}}\to w^\ast \in [0,+\infty].
\end{align*}
The proof is now complete. \qed

\subsection{Proof of Corollary~\ref{cor:wwg}}\label{pf:wwg}
We will prove a slightly stronger (than (\ref{cond:power_lrt_subspace_2})) claim that condition (\ref{cond:wwg}) implies
\begin{align}\label{ineq:subspace_1}
\Gamma_K\big(\mu - \Pi_{K_0}(\mu)\big)\to \infty.
\end{align}
Suppose $\pnorm{\mu-\Pi_{K_0}(\mu)}{}$ is greater or equal than $L_n$ times the right hand side of (\ref{cond:wwg}) for some slowly growing sequence $L_n\uparrow\infty$. Then either (i) $\pnorm{\mu-\Pi_{K_0}(\mu)}{} \geq L_n \delta_{K}^{1/4}$, or (ii) $\pnorm{\mu-\Pi_{K_0}(\mu)}{}<L_n \delta_{K}^{1/4}$ and $\iprod{\mu-\Pi_{K_0}(\mu)}{\E \Pi_{K}(\xi)}\geq L_n \delta_{K}^{1/2}$. In both cases, we have $\|\mu - \Pi_{K_0}(\mu)\|\rightarrow \infty$ as there exists some universal constant $c_0 > 0$ such that the right hand side of (\ref{cond:wwg}) is bounded below by $c_0$. In case (i), using \cite[(74a)]{wei2019geometry},
\begin{align}\label{ineq:L_cal_1}
\nonumber\Gamma_{K}\big(\mu-\Pi_{K_{0}}(\mu)\big)&\geq \frac{ \pnorm{\mu-\Pi_{K_{0}}(\mu)}{}^2  }{ 2\pnorm{\mu-\Pi_{K_{0}}(\mu)}{}+8 \E \pnorm{\Pi_{K}(\xi)}{}}-2/\sqrt{e}\\
\nonumber&\geq (1/16) \pnorm{\mu-\Pi_{K_{0}}(\mu)}{} \bigwedge \Big(\pnorm{\mu-\Pi_{K_{0}}(\mu)}{}^2 /\delta_{K}^{1/2}\Big)-2/\sqrt{e}\\
&\geq (1/16) \pnorm{\mu-\Pi_{K_{0}}(\mu)}{} \bigwedge L_n^2-2/\sqrt{e}\rightarrow \infty
\end{align}
as $n\rightarrow\infty$, so (\ref{ineq:subspace_1}) is verified. 	In case (ii), we may assume without loss of generality that $\|\mu - \Pi_{K_0}(\mu)\| \leq L_n^{1/4}\delta_K^{1/4}$ because otherwise we can follow the same arguments as in the previous case. Then using \cite[(74b)]{wei2019geometry} with
\begin{align*}
\alpha\equiv \alpha\big(\mu-\Pi_{K_0}(\mu)\big)&=1-e^{-\iprod{\mu-\Pi_{K_0}(\mu)}{\E\Pi_{K}(\xi)}^2/8\pnorm{\mu-\Pi_{K_0}(\mu)}{}^2}\\
&\geq 1- e^{-\delta_K^{1/2}/8}\to 1,
\end{align*}
we have
\begin{align*}
\Gamma_{K}\big(\mu-\Pi_{K_{0}}(\mu)\big)&\geq \alpha\cdot\frac{\iprod{\mu - \Pi_{K_0}(\mu)}{\E \Pi_K \xi} - \|\mu - \Pi_{K_0}(\mu)\|^2}{\alpha\|\mu - \Pi_{K_0}(\mu)\| + 2\E\|\Pi_K(\xi)\|_2} - \frac{2}{\sqrt{e}}\\
&\gtrsim \frac{(L_n-L_n^{1/2})\delta_K^{1/2}}{L_n^{1/4} \delta_{K}^{1/4}+ \delta_{K}^{1/2} } - \mathcal{O}(1) \rightarrow\infty
\end{align*}
as $n\rightarrow\infty$, so (\ref{ineq:subspace_1}) is verified. The proof is complete. \qed

\section{Proofs of results in Section~\ref{section:example}}\label{section:proof_example}

\subsection{Proof of Theorem \ref{thm:lrt_clt_global_testing_orthant}}\label{pf:lrt_clt_global_testing_orthant}

\begin{lemma}\label{lem:gaussian_tail}
	Let $\xi_1$ be a standard normal random variable. Then for $x>0$,
	\begin{align*}
	\E[\xi_1\bm{1}_{\xi_1\geq x}]=\varphi(x), \quad \E[\xi_1^2\bm{1}_{\xi_1\geq x}]=x\varphi(x)+\int_{x}^\infty \varphi(y)\,\mathrm{d}{y}.
	\end{align*}
\end{lemma}
\begin{proof}
	The first equality follows as $\E[\xi_1\bm{1}_{\xi_1\geq x}] = \int_{x}^\infty y \varphi(y)\,\mathrm{d}{y} = \varphi(x)$. The second equality follows as $\E[\xi_1^2\bm{1}_{\xi_1 \geq x}] = \int_x^\infty y^2\varphi(y)\,\mathrm{d}{y} = -\int_x^\infty y\varphi'(y)\,\mathrm{d}{y} = x\varphi(x)+\int_x^\infty \varphi(y)\,\mathrm{d}{y}$. 
\end{proof}

\begin{proof}[Proof of Theorem \ref{thm:lrt_clt_global_testing_orthant}]
	Note that $\hat{\mu}_{K_{+} }  =\big((\mu_i+\xi_i)_+\big)$. For $\mu_0 \in K_{+}$, so
	\begin{align*}
	\E_{\mu_0} \pnorm{\hat{\mu}_{K_{+}} -\mu_0 }{}^2 &= \sum_{i=1}^n \E \Big[\big((\mu_0)_i+\xi_i\big)_+-(\mu_0)_i\Big]^2\\
	& =  \sum_{i=1}^n \Big[\E \xi_i^2 \bm{1}_{\xi_i\geq -(\mu_0)_i}+(\mu_0)_i^2 \Prob\big(\xi_i<-(\mu_0)_i\big)\Big].
	\end{align*}
	As $(\mu_0)_i\geq 0$ for $1\leq i\leq n$, and $\sup_{x>0} x^2 \Prob(\xi<-x)<\infty$, it follows that
	\begin{align*}
	\E_{\mu_0} \pnorm{\hat{\mu}_{K_{+}} -\mu_0 }{}^2 \asymp n.
	\end{align*}
	On the other hand, as under the null $J_{\hat{\mu}_{K_{+}}}=\big(\bm{1}_{i=j}\bm{1}_{\xi_i\geq -(\mu_0)_i}\big)_{ij}$,
	\begin{align}\label{ineq:orthant_global_testing_1}
	\pnorm{\E_{\mu_0} J_{\hat{\mu}_{K_{+}}}}{F}^2=\sum_{i,j} \big(\E_{\mu_0} J_{\hat{\mu}_{K_{+}}}\big)_{ij}^2 = \sum_{i=1}^n \big(\Prob(\xi_i\geq -(\mu_0)_i)\big)^2 \asymp n.
	\end{align}
	The claim (1) now follows from Theorem \ref{thm:lrt_clt_global_testing}.
	
	For (2), let for $x\geq 0$
	\begin{align*}
	Q(x)&\equiv \E \xi_1^2 \bm{1}_{\xi_1\geq -x}+2x \E \xi_1 \bm{1}_{\xi_1\geq-x}-x^2 \Prob\big(\xi_1<-x)\\
	& = 1- \E \xi_1^2 \bm{1}_{\xi_1\geq x} + 2x \E \xi_1 \bm{1}_{\xi_1\geq x} - x^2 \Prob\big(\xi_1\geq x\big)\\
	& = \int_{-\infty}^x \varphi(y)\,\mathrm{d}{y}+x\varphi(x)-x^2 \big(1-\Phi(x)\big).
	\end{align*}
	The last equality follows from Lemma \ref{lem:gaussian_tail}. Hence for all $x\geq 0$,
	\begin{align*}
	Q^\prime(x) &= 2\varphi(x)+x\varphi^\prime(x)- \big[2x(1-\Phi(x))-x^2\varphi(x)\big]\\
	& = 2 \varphi(x) - 2x\big(1-\Phi(x)\big),\\
	Q^{\prime\prime}(x) & = 2\Big[\varphi^\prime(x)-1+\Phi(x)+x\varphi(x)\Big] = 2\big(-1+\Phi(x)\big)<0.
	\end{align*}
	This means that $Q'$ is nonnegative, decreasing with $Q'(0)=2\varphi(0) = 2/\sqrt{2\pi}$ and $Q'(\infty) = 0$, and $Q$ is strictly increasing, concave and bounded on $[0,\infty)$ with $Q(0)=1/2$.

	Now note that for any $\mu \in K_{+}$, 
	\begin{align*}
	&m_\mu -\pnorm{\mu-\mu_0}{}^2\\
	& = \E \big[2\iprod{\xi}{\Pi_{K_{+}}(\mu+\xi)-\mu }-\pnorm{\Pi_{K_{+}}(\mu+\xi)-\mu}{}^2\big]\\
	& = \sum_{i=1}^n \bigg[2 \E \xi_i^2 \bm{1}_{\xi_i\geq -\mu_i}+2\mu_i \E \xi_i  \bm{1}_{\xi_i\geq -\mu_i} -\bigg(\E \xi_i^2 \bm{1}_{\xi_i\geq -\mu_i}+\mu_i^2 \Prob\big(\xi_i<-\mu_i)\big)\bigg)\bigg]\\
	& = \sum_{i=1}^n \bigg[\E \xi_i^2 \bm{1}_{\xi_i\geq -\mu_i} +2\mu_i \E \xi_i  \bm{1}_{\xi_i\geq -\mu_i} -\mu_i^2 \Prob\big(\xi_i<-\mu_i)\big)\bigg]=\sum_{i=1}^n Q(\mu_i).
	\end{align*}
	Using the lower bound (\ref{ineq:orthant_global_testing_1}) for $\sigma_{\mu_0}^2$, and an easy matching upper bound (by e.g. triangle inequality), we have $\sigma_{\mu_0}^2 \asymp n$. The condition (\ref{cond:power_lrt_global_testing_0}) reduces to
	\begin{align}\label{ineq:orthant_global_testing_2}
	\pnorm{\mu-\mu_0}{}\ll \biggabs{\sum_{i=1}^n\big\{\bar{S}_+(\mu_i)-\bar{S}_+((\mu_0)_i)\big\}+\pnorm{\mu-\mu_0}{}^2 }\bigvee n^{1/2}.
	\end{align}
	(\ref{ineq:orthant_global_testing_2}) clearly holds for $\pnorm{\mu-\mu_0}{}\ll n^{1/2}$. For $\pnorm{\mu-\mu_0}{}\gg n^{1/2}$, as
	\begin{align*}
	\biggabs{\sum_{i=1}^n\big\{\bar{S}_+(\mu_i)-\bar{S}_+((\mu_0)_i)\big\}} \leq (2/\sqrt{2\pi}) \pnorm{\mu-\mu_0}{1}\lesssim \sqrt{n} \pnorm{\mu-\mu_0}{},
	\end{align*}
	the right hand side of (\ref{ineq:orthant_global_testing_2}) is bounded from below by
	\begin{align*}
	\big(\pnorm{\mu-\mu_0}{}^2-C \sqrt{n} \pnorm{\mu-\mu_0}{}\big)_+\vee n^{1/2}\asymp \pnorm{\mu-\mu_0}{}^2\gg \pnorm{\mu-\mu_0}{},
	\end{align*}
	so (\ref{ineq:orthant_global_testing_2}) holds. Hence in these two regimes, the claim follows from Theorem \ref{thm:power_lrt_global_testing}-(2). For $\pnorm{\mu-\mu_0}{}\asymp n^{1/2}$, by the decomposition (\ref{ineq:decomposition_power_lrt_generic}), the LRT is powerful if and only if $\abs{m_\mu-m_{\mu_0}}/\sigma_{\mu_0}\to \infty$, i.e., $\bigabs{ \sum_{i=1}^n \big\{\bar{S}_+(\mu_i)-\bar{S}_+((\mu_0)_i)\big\} +\pnorm{\mu-\mu_0}{}^2}\gg n^{1/2}$. The proof is now complete.
\end{proof}

\subsection{Proof of Theorem \ref{thm:lrt_clt_global_testing_circular}}\label{pf:lrt_clt_global_testing_circular}
We write $K_{\times,\alpha}$ for $K_\times$ in the proof for notational convenience.	

\noindent (1). This claim follows from the fact that $\sigma_0 \asymp \delta_{K_\times}^{1/2} \asymp \delta_{K_\alpha}^{1/2}\asymp n^{1/2}$ (cf.~\cite[Section 6.3]{mccoy2014from}) and Theorem \ref{thm:power_lrt_subspace}-(1). 

\noindent (2)(a). We only need to prove that the LRT is not powerful for $\mu \in K_\alpha$ such that $\pnorm{\mu}{}=\mathcal{O}(1)$. Using the decomposition (\ref{ineq:decomposition_power_lrt}), it suffices to show $T(\mu+\xi)-T(\xi) = \mathcal{O}_{\mathbf{P}}(n^{1/2})$. This follows as
\begin{align*}
&T(\mu+\xi)-T(\xi)\\
& = \|\mu+\xi\|^2 - \|\mu+\xi - \Pi_{K_\alpha}(\mu+\xi)\|^2 - \big(\pnorm{\xi}{}^2-\pnorm{\xi-\Pi_{K_\alpha}(\xi)}{}^2\big)\\
& = \pnorm{\mu}{}^2+2\iprod{\mu}{\xi}- \pnorm{\Pi_{K_\alpha}(\xi)-\Pi_{K_\alpha}(\mu+\xi)+\mu }{}^2 \\
&\qquad -2\iprod{\xi- \Pi_{K_\alpha}(\xi) }{ \Pi_{K_\alpha}(\xi)-\Pi_{K_\alpha}(\mu+\xi)+\mu}\\
& = \mathcal{O}_{\mathbf{P}}\bigg( \pnorm{\mu}{}^2+ \pnorm{\mu}{}+\pnorm{\Pi_{K_\alpha}(\xi)-\Pi_{K_\alpha}(\mu+\xi) }{}^2 +\pnorm{\mu}{}^2\\
&\qquad  + \pnorm{ \xi- \Pi_{K_\alpha}(\xi) }{} \big[ \pnorm{\Pi_{K_\alpha}(\xi)-\Pi_{K_\alpha}(\mu+\xi) }{} \vee \pnorm{\mu}{}\big]\bigg)\\
& = \mathcal{O}_{\mathbf{P}}(n^{1/2}).
\end{align*} 
\noindent (2)(b). By (2)(a) and Theorem \ref{thm:power_lrt_subspace}-(2), we have $\pnorm{\mu^1}{}\gg 1$ if and only if 
\begin{align*}
\frac{ \E\|\Pi_{K_\alpha}(\mu^1 + \xi^1)\|^2 - \E\|\Pi_{K_\alpha}(\xi^1)\|^2 }{n^{1/2}} \rightarrow\infty.
\end{align*}
Now using Theorem \ref{thm:power_lrt_subspace}-(2) again for $K_\times$ to conclude by noting that
\begin{align*}
&\E\|\Pi_{K_{\times}}(\mu + \xi)\|^2 - \E\|\Pi_{K_{\times}}(\xi)\|^2 \\
&= \E\|\Pi_{K_\alpha}(\mu^1 + \xi^1)\|^2 - \E\|\Pi_{K_\alpha}(\xi^1)\|^2+ (\mu^2)^2.
\end{align*}
This completes the proof. \qed

\subsection{Proof of Theorem \ref{thm:lrt_clt_global_testing_iso}} \label{pf:lrt_clt_global_testing_iso}

We first prove Proposition \ref{prop:lower_bound_J_iso}. The following lemma will be used. We present its proof at the end of this subsection.

\begin{lemma}\label{lem:iso_localize}
	Fix $0.1n \leq i \leq 0.9n$. Let $u^\ast \leq i$ and $h_1^\ast \geq 0$ be defined through the following max-min formula for the isotonic LSE:
	\begin{align}\label{eq:max_min}
	\hat{\mu}_i &= \max_{u\leq i}\min_{v\geq i}\bar{Y}|_{[u,v]}  \equiv \min_{v\geq i}\bar{Y}|_{[u^\ast,v]} \equiv \min_{h_2\geq 0}\bar{Y}|_{[i - h_1^*n^{2/3}, i + h_2n^{2/3}]}.
	\end{align}
	Then there exists some $C = C(L)>0$ such that for any $t > 0$
	\begin{align*}
	\Prob\big(|\hat{\mu}_i - \mu_i| > n^{-1/3}t\big)\vee \Prob(h_1^\ast > t) \leq C\exp(-t^2/C).
	\end{align*} 
\end{lemma}

\begin{proof}[Proof of Proposition \ref{prop:lower_bound_J_iso}]
	We write in the proof $\hat{\mu} = \hat{\mu}_{K_\uparrow}$ and $\mu=\mu_0$ for simplicity of notation. Note that $(J_{\hat{\mu}})_{ij}=\bm{1}_{\hat{\mu}_i=\hat{\mu}_j}(1/\abs{\{k: \hat{\mu}_k=\hat{\mu}_i\}})$. Note that
	\begin{align*}
	\Prob(\hat{\mu}_i=\hat{\mu}_j) &= \E\bigg[\bm{1}_{\hat{\mu}_i=\hat{\mu}_j}\cdot \frac{1}{ \abs{\{k: \hat{\mu}_k=\hat{\mu}_i\}}^{1/2}}\cdot \abs{\{k: \hat{\mu}_k=\hat{\mu}_i\}}^{1/2} \bigg]\\
	&\leq \sqrt{ (\E J_{\hat{\mu}})_{ij} }\cdot \sqrt{\E \abs{\{k: \hat{\mu}_k=\hat{\mu}_i\}}}.
	\end{align*}
	This implies that
	\begin{align}\label{ineq:iso_exp_J_1}
	(\E J_{\hat{\mu}})_{ij} \geq \frac{\Prob^2(\hat{\mu}_i=\hat{\mu}_j)}{ \E \abs{\{k: \hat{\mu}_k=\hat{\mu}_i\}} }.
	\end{align}
	We will bound the denominator from above and the numerator from below in the above display separately in the regime $\{(i,j): \abs{i-j}\leq \kappa n^{2/3}, 0.1 n\leq i,j\leq 0.9n\}$, where $\kappa=\kappa(L)>0$ is a constant to be specified below. 
	
	Fix $0.1n\leq i\leq 0.9n$. 	First we provide an upper bound for the denominator in (\ref{ineq:iso_exp_J_1}). By Lemma \ref{lem:iso_localize} and using the notation defined therein, there exists some large $c=c(L,\epsilon)>1$ such that on an event $E_0$ with probability $1-\epsilon$,
	\begin{align}\label{ineq:iso_exp_J_2}
	\abs{\hat{\mu}_i-\mu_i}&\leq cn^{-1/3},
	\end{align}
	and
	\begin{align*}
	\Prob\big( E_1\equiv \big\{ h_1^\ast \geq c \big\}\big)\leq Ce^{-c^2/C},
	\end{align*}
	where $C=C(L)>0$ is a constant depending on $L$ only. Hence integrating the tail leads to the following: for some constant $C'=C'(L)>0$,
	\begin{align*}
	\E \abs{\{k\leq i: \hat{\mu}_k=\hat{\mu}_i\}} &\leq C' n^{2/3}.
	\end{align*}
	Similarly we can handle the case $k\geq i$, so we arrive at 
	\begin{align}\label{ineq:iso_exp_J_3}
	\E \abs{\{k: \hat{\mu}_k=\hat{\mu}_i\}} &\leq C'' n^{2/3}
	\end{align}
	for some constant $C''=C''(L)>0$.

	Next we provide a lower bound for the numerator of (\ref{ineq:iso_exp_J_1}). On the event $E_2=\{1\vee (i- c^{-100} n^{2/3})\leq u^\ast \leq i\}$ (there is nothing special about the constant $100$---a large enough value suffices), we have
	\begin{align*}
	&n^{1/3}\big(\hat{\mu}_i-\mu_i\big)= \min_{v\geq i} n^{1/3}\big(\bar{\mu}|_{[u^\ast,v]}-\mu_i+\bar{\xi}|_{[u^\ast,v]}\big)\\
	&\leq \min_{i\leq v\leq n\wedge (i+c^{-10} n^{2/3})} n^{1/3}\big(\bar{\mu}|_{[u^\ast,v]}-\mu_i+\bar{\xi}|_{[u^\ast,v]}\big)\\
	&\leq \min_{i\leq v\leq n\wedge (i+c^{-10} n^{2/3})} n^{1/3}\bar{\xi}|_{[u^\ast,v]} +  \max_{i\leq v\leq n\wedge (i+c^{-10} n^{2/3})} n^{1/3} \big(\bar{\mu}|_{[u^\ast,v]}-\mu_i\big)\\
	& \leq \min_{0\leq h_2 \leq c^{-10}} \frac{W(-h_1^\ast)+W(h_2)+R_n }{h_1^\ast+h_2+\abs{\mathcal{O}(n^{-2/3})} }+\mathcal{O}(c),
	\end{align*}
	where $R_n = \mathcal{O}_{a.s.}(\log n/n^{1/3})$ comes from Kolm\'os-Major-Tusn\'ady strong embedding, and $W$ denotes a standard two-sided Brownian motion starting from $0$. The bound $\mathcal{O}(c)$ for the bias term follows as
	\begin{align}\label{ineq:iso_exp_J_5}
	&\max_{i\leq v\leq n\wedge (i+c^{-10} n^{2/3})} n^{1/3} \big(\bar{\mu}|_{[u^\ast,v]}-\mu_i\big) \leq \max_{i\leq v\leq n\wedge (i+cn^{2/3})} n^{1/3} \big(\bar{\mu}|_{[i,v]}-\mu_i\big) \nonumber\\
	&\leq \max_{n^{-2/3}\leq h_2\leq c}\frac{ n^{1/3}\sum_{\ell \in [i,i+h_2n^{2/3}]\cap \mathbb{Z}}(\mu_\ell-\mu_i)}{\floor{h_2n^{2/3}} +1} \nonumber\\
	&\leq  \max_{n^{-2/3}\leq h_2\leq c}\frac{ n^{1/3}(L/n)\sum_{\ell \in [0,h_2n^{2/3}]\cap \mathbb{Z}}\ell }{\floor{h_2n^{2/3}}+1 } = \mathcal{O}(c).
	\end{align}
	Now on the event $E_2$, 
	\begin{align*}
	W(-h_1^\ast) \leq \sup_{0\leq h_1\leq c^{-100}}W(-h_1)\equald c^{-50} \sup_{0\leq t\leq 1} W(t)\equiv c^{-50}Z.
	\end{align*}
	By reflection principle for Brownian motion, for any $u>0$,
	\begin{align*}
	\Prob\big(\{W(-h_1^\ast)>c^{-50} u\}\cap E_2\big)\leq \Prob\big(Z> u) = 2\Prob\big(W(1)>u\big)\leq 2e^{-u^2/2}.
	\end{align*}
	Let $h_2^\circ$ be such that
	\begin{align*}
	W(h_2^\circ)\equiv \inf_{0\leq h_2\leq c^{-10}} W(h_2) \equald c^{-5} \inf_{0\leq t\leq 1}W(t) = -c^{-5} Z.
	\end{align*}
	So for $u>0$,
	\begin{align*}
	\Prob(W(h_2^\circ)<-c^{-5}u) = \Prob(Z>u) =1- \Prob\big(\abs{\mathcal{N}(0,1)}\leq u\big)\geq 1-2u.
	\end{align*}
	Hence on the event $E_2$ intersected with an event with probability at least $1-4\epsilon$, 
	\begin{align*}
	&\min_{0\leq h_2 \leq c^{-10}} \frac{W(-h_1^\ast)+W(h_2)+R_n }{h_1^\ast+h_2+\abs{\mathcal{O}(n^{-2/3})} }\leq \frac{W(-h_1^\ast)+W(h_2^\circ)+R_n}{h_1^\ast+h_2^\circ +\abs{\mathcal{O}(n^{-2/3})} }\\
	& \leq \frac{c^{-50}\sqrt{2\log(1/\epsilon)} -c^{-5}\epsilon + R_n }{h_1^\ast+h_2^\circ+\abs{\mathcal{O}(n^{-2/3})} }\leq -C_\epsilon\cdot c^5.
	\end{align*}
	where the last inequality follows by choosing for $c=c(\epsilon)$ large enough followed by $n$ large enough, and $h_1^\ast+h_2^\circ \leq c^{-100}+c^{-10}\leq 2c^{-10}$ on $E_2$. Combining the above estimates, we see that 
	\begin{align*}
	n^{1/3}\big(\hat{\mu}_i-\mu_i\big)\leq -C_\epsilon'\cdot c^5
	\end{align*}
	on the event $E_2$ intersected with an event with probability at least $1-4\epsilon$, when $c$ and $n$ are chosen large enough, depending on $L,\epsilon$. This event must occur with small enough probability for $c$ large in view of (\ref{ineq:iso_exp_J_2}), so we have proved that $\Prob(E_2)\leq 5\epsilon$ for large enough $c=c(L,\epsilon)>1$ and $n=n(L,\epsilon) \in \N$. This means that $\Prob(\hat{\mu}_i=\hat{\mu}_j)\geq 1-5\epsilon$ for $1\vee (i-c^{-100} n^{2/3})\leq j\leq i$ for large enough $c=c(L,\epsilon)>1$ and $n=n(L,\epsilon) \in \N$. Similarly one can handle the regime $i\leq j\leq (i+c^{-100} n^{2/3})\vee n$. In summary, we have proved there exists some $\kappa=\kappa(L)>0$ such that
	\begin{align}\label{ineq:iso_exp_J_4}
	\Prob(\hat{\mu}_i=\hat{\mu}_j)\geq 1/2
	\end{align}
	holds for $\{(i,j): \abs{i-j}\leq \kappa n^{2/3}, 0.1 n\leq i,j\leq 0.9n\}$ for $n$ large enough. The claim of the proposition now follows by plugging (\ref{ineq:iso_exp_J_3}) and (\ref{ineq:iso_exp_J_4}) into (\ref{ineq:iso_exp_J_1}).
\end{proof}

Now we are in position to prove Theorem \ref{thm:lrt_clt_global_testing_iso}.
\begin{proof}[Proof of Theorem \ref{thm:lrt_clt_global_testing_iso}-(1)]
	We write in the proof $\hat{\mu} = \hat{\mu}_K$ and $\mu=\mu_0$ for simplicity of notation. $\Prob_\mu, \E_\mu$ are shorthanded to $\Prob,\E$ if no confusion could arise. For $\kappa>0$, let $I_\ell\equiv I_\ell(\kappa)\equiv \{i: 0.1n+ (\ell-1)\cdot \kappa n^{2/3}\leq i\leq \big(0.1n + \ell\cdot \kappa n^{2/3}\big)\wedge 0.9n\}$ and $\ell_0$ be the maximum integer for which $I_{\ell_0} \subset [0.1n,0.9n]$. Clearly $\abs{I_\ell}\asymp \kappa n^{2/3}$ for all $1\leq \ell\leq \ell_0$ and $\ell_0\asymp n^{1/3}/\kappa$. Using the $\kappa$ specified in Proposition \ref{prop:lower_bound_J_iso}, we have
	\begin{align}\label{ineq:global_testing_iso_1}
	\pnorm{\E J_{\hat{\mu}}}{F}^2 = \sum_{i,j} \big(\E J_{\hat{\mu}}\big)_{ij}^2 & \geq \sum_{\ell=1}^{\ell_0} \sum_{(i,j)\in I_\ell\times I_\ell}(\E (J_{\hat{\mu}})_{ij}\big)^2 \nonumber\\
	&\gtrsim \sum_{\ell=1}^{\ell_0} \sum_{(i,j)\in I_\ell\times I_\ell} n^{-4/3}  \asymp n^{1/3}.
	\end{align}
	On the other hand, by e.g., \cite{zhang2002risk}, under the condition of Theorem \ref{thm:lrt_clt_global_testing_iso}, we have
	\begin{align*}
	\E \pnorm{\hat{\mu}-\mu}{}^2\lesssim n^{1/3}. 
	\end{align*}
	The claim now follows by applying Theorem \ref{thm:lrt_clt_global_testing} (by ignoring the bias term in the denominator) with the above two displays.
\end{proof}

\begin{proof}[Proof of Theorem \ref{thm:lrt_clt_global_testing_iso}-(2)]
	Following the notation used in \cite{meyer2000degrees}, let $\tilde{W}$ be the greatest convex minorant of $t\mapsto W(t)+t^2/2, t \in \R$, and $a=-\E[\tilde{W}(0)]>0, b = \E[\tilde{W}'(0)^2]>0$. 
	Using the same techniques as in \cite[Theorem 2, Corollary 4]{meyer2000degrees} but by performing Taylor expansion to the second order,  it can be shown that for all $C^2$ monotone functions $f:[0,1]\to \R$ with bounded first derivative $f'$ away from $0$ and $\infty$, and bounded second derivative $f''$ away from $\infty$,
	\begin{align}\label{ineq:global_testing_iso_1.1}
	\E_{\mu_f} \dv \hat{\mu}_{K_{\uparrow}} &= (a+b)\cdot  n^{1/3} \int_0^1 (f'(t))^{2/3}\,\mathrm{d}{t} +\mathcal{O}(1),\nonumber\\
	\E_{\mu_f} \pnorm{\hat{\mu}_{K_{\uparrow}}-\mu_f}{}^2 &= b  \cdot n^{1/3} \int_0^1 (f'(t))^{2/3}\,\mathrm{d}{t} +\mathcal{O}(1).
	\end{align}
	Here $\mu_f = (f(i/n))_{i=1}^n$ for a generic $f:[0,1]\to \R$, and the $\mathcal{O}(1)$ term in the above display depends only on the upper and lower bounds for $f'$ and the upper bound of $f''$.  Hence for the prescribed $f$,
	\begin{align*}
	m_{\mu_f} - \pnorm{\mu_f-\mu_{f_0}}{}^2& = 2\E_{\mu_f} \dv \hat{\mu}_{K_{\uparrow}} - \E_{\mu_f} \pnorm{\hat{\mu}_{K_{\uparrow}}-\mu_f}{}^2\\
	& = (2a+b) \cdot n^{1/3} \int_0^1 (f'(t))^{2/3}\,\mathrm{d}{t} +\mathcal{O}(1).
	\end{align*}
	On the other hand, for the prescribed $f$, (\ref{ineq:global_testing_iso_1}) provides a lower bound for $\sigma_{\mu_f}^2$, while the Gaussian-Poincar\'e inequality yields a matching upper bound:
	\begin{align*}
	n^{1/3}\lesssim \sigma^2_{\mu_f} \leq 4 \E_{\mu_f} \pnorm{\hat{\mu}_{K_{\uparrow}}-\mu_f}{}^2 \lesssim n^{1/3}. 
	\end{align*}
	Now with $\pnorm{\delta}{[1]}\equiv \int \delta'$, condition (\ref{cond:power_lrt_global_testing_0}) reduces to
	\begin{align}\label{ineq:global_testing_iso_2}
	&\pnorm{\mu_f-\mu_{f_0}}{}\ll \bigabs{m_{\mu_f}-m_{\mu_{f_0}}} \vee \sigma_{\mu_f}\nonumber\\
	\Leftrightarrow\quad & \sqrt{n\int (f-f_0)^2+\mathcal{O}(1)} \nonumber\\
	&\quad \ll \biggabs{ (2a+b) n^{1/3} \int \Big\{ (f')^{2/3}-(f'_0)^{2/3}\Big\} +\mathcal{O}(1)+  n\int (f-f_0)^2 }  \bigvee n^{1/6}\nonumber\\
	\Leftrightarrow\quad& \sqrt{n\rho_n^2+\mathcal{O}(1)} \ll \biggabs{-\abs{\mathcal{O}(1)}n^{1/3}\rho_n \pnorm{\delta}{[1]}+\mathcal{O}(1)+ n\rho_n^2 }\bigvee n^{1/6},
	\end{align}
	where in the last equivalence we used that
	\begin{align*}
	\int \Big\{ (f')^{2/3}-(f'_0)^{2/3}\Big\} &= \int \Big\{ (f')^{2/3}-(f'+\rho_n \delta')^{2/3}\Big\} =- \mathcal{O}(\rho_n) \int\delta'.
	\end{align*}
	By Theorem \ref{thm:power_lrt_global_testing}-(2), under (\ref{ineq:global_testing_iso_2}) the LRT is power consistent if and only if
	\begin{align}\label{ineq:global_testing_iso_3}
	\frac{\bigabs{-\abs{\mathcal{O}(1)}n^{1/3}\rho_n \pnorm{\delta}{[1]}+ n\rho_n^2}}{n^{1/6}} \to \infty.
	\end{align}
	We have two cases:
	\begin{enumerate}
		\item If $n\rho_n^2\gg n^{1/3}\rho_n \abs{ \pnorm{\delta}{[1]} }\Leftrightarrow \rho_n\gg n^{-2/3} \abs{\pnorm{\delta}{[1]}}$, then (\ref{ineq:global_testing_iso_3}) requires $\rho_n\gg n^{-5/12}$. 
		\item If $n\rho_n^2\ll n^{1/3}\rho_n \abs{ \pnorm{\delta}{[1]} }\Leftrightarrow \rho_n\ll n^{-2/3} \abs{\pnorm{\delta}{[1]}}$,  then (\ref{ineq:global_testing_iso_3}) requires $\rho_n\gg n^{-1/6} /\abs{\pnorm{\delta}{[1]}}$. This is not feasible as $\abs{\pnorm{\delta}{[1]}} = \abs{\int \delta'}=\mathcal{O}(1)$. 
	\end{enumerate}
	To summarize, (\ref{ineq:global_testing_iso_3}) is equivalent to requiring $\rho_n\gg n^{-5/12}$. In this regime (\ref{ineq:global_testing_iso_2}) also holds. The proof is complete.
\end{proof}

\begin{proof}[Proof of Lemma \ref{lem:iso_localize}]
	By the monotonicity of $\mu$, we have
	\begin{align*}
	\hat{\mu}_i - \mu_i &= \min_{h_2 \geq 0}\bar{Y}|_{[i - h_1^*n^{2/3}, i + h_2n^{2/3}]} - \mu_i \leq \bar{Y}|_{[i - h_1^*n^{2/3}, i + n^{2/3}]} - \mu_i\\
	&= \big(\bar{\mu}|_{[i - h_1^*n^{2/3}, i + n^{2/3}]} - \mu_i\big) + \bar{\xi}|_{[i - h_1^*n^{2/3}, i + n^{2/3}]}\\
	&\leq \big(\mu_{\ceil{i+n^{2/3}}} - \mu_i\big) + \max_{h_1\geq0}|\bar{\xi}|_{[i - h_1n^{2/3}, i + n^{2/3}]}|\\
	&\leq C_1 n^{-1/3} + \max_{h_1\geq0}|\bar{\xi}|_{[i - h_1n^{2/3}, i + n^{2/3}]}|,
	\end{align*}
	where the last inequality follows by (\ref{cond:lrt_clt_iso}). Note for any $t > 0$, a standard blocking argument (cf. \cite[Lemma 2]{han2019limit}) yields
	\begin{align}\label{ineq:local_iso_1}
	\Prob\bigg(\max_{h_1\geq0}|\bar{\xi}|_{[i - h_1n^{2/3}, i + n^{2/3}]}| > t n^{-1/3}\bigg) 
	\leq  Ce^{-t^2/C}.
	\end{align}
	This concludes the one-sided estimate $\Prob(\hat{\mu}_i - \mu_i >n^{-1/3}t)$. The other side is similar.
	
	Now consider $\Prob(h_1^\ast>t)$. On the event $\{h_1^* > t\}$, we have
	\begin{align*}
	\hat{\mu}_i - \mu_i &= \min_{h_2 \geq 0}\bar{Y}|_{[i - h_1^*n^{2/3}, i + h_2n^{2/3}]}\\
	&\leq \big(\bar{\mu}|_{[i - h_1^*n^{2/3}, i + n^{2/3}]} - \mu_i\big) + \bar{\xi}|_{[i - h_1^*n^{2/3}, i + n^{2/3}]}\\
	&\leq \big(\bar{\mu}|_{[i - tn^{2/3}, i + n^{2/3}]} - \mu_i\big) + \max_{h_1\geq0}\bigabs{\bar{\xi}|_{[i - h_1n^{2/3}, i + n^{2/3}]}}\\
	&\leq -C_2\cdot t n^{-1/3} + C_3\cdot n^{-1/3}+ \max_{h_1\geq0}\bigabs{\bar{\xi}|_{[i - h_1n^{2/3}, i + n^{2/3}]}},
	\end{align*}
	where the last inequality follows from calculations similar to (\ref{ineq:iso_exp_J_5}), but now using both the upper and lower bound parts of (\ref{cond:lrt_clt_iso}). Choosing $t\geq 2C_3/C_2$, and replacing $t$ in (\ref{ineq:local_iso_1}) by $C_2 t/4$, we see that
	\begin{align*}
	\Prob\big(h_1^\ast>t\big)\leq \Prob\big(\hat{\mu}_i-\mu_i\leq -(C_2/4) tn^{-1/3}\big)+ C_4 e^{-t^2/C_4}.
	\end{align*}
	The claim follows by adjusting constants.
\end{proof}

\subsection{Proof of Theorem \ref{thm:lrt_clt_global_testing_lasso}}\label{pf:lrt_clt_global_testing_lasso}
We first prove Proposition \ref{prop:lasso_estimates}. The following lemma will be useful to control the term $\mathfrak{p}_{\lambda,\mu_0}$ therein. We present its proof at the end of this subsection.
\begin{lemma}\label{lem:l1_deviation}
	For any $\theta\in\R^p$ with $\|\theta\|_1\leq \lambda$, let $\mu \equiv X\theta \in K_{X,\lambda}$. There exists some  universal constant $C > 0$ such that for $t\geq 1$,
	\begin{align*}
	\Prob_\mu\bigg(\|\hat{\theta}^0\|_1 \geq \pnorm{\theta}{1} +t \sqrt{\frac{p}{n\lambda_{\min}(\Sigma)}}\bigg) \leq e^{-t^2/C}.
	\end{align*}
\end{lemma}

\begin{proof}[Proof of Proposition \ref{prop:lasso_estimates}]
	
	\noindent (1). We will derive an explicit formula for $J_{\hat{\mu}}$ using the results of \cite{kato2009degrees}. First note by the chain rule that
	\begin{align*}
	J_{\hat{\mu}}(\xi) = \frac{\partial\hat{\mu}}{\partial \xi} = \frac{\partial\hat{\mu}}{\partial \hat{\theta}} \frac{\partial\hat{\theta} }{\partial \hat{\theta}^0} \frac{\partial \hat{\theta}^0}{\partial \xi} = X\frac{\partial \hat{\theta}}{\partial\hat{\theta}^0}(X^\top X)^{-1}X^\top.
	\end{align*}
	Let $\tilde{K}\equiv \tilde{K}_\lambda\equiv \{\theta \in\R^p: \|\theta\|_1\leq \lambda\}$. For each $m \in \{1,\ldots,p\}$, suppose there are $N_m$ faces of $\tilde{K}_\lambda$ of dimension $m$, denoted as $\{F_{m,\ell}\}_{\ell=1}^{N_m}$. Then we can partition $\partial \tilde{K}_\lambda$ as $\{F_{m,\ell}\}_{m,\ell}$. Let $\big\{E_0,\{E_{m,\ell}\}_{m,\ell}\big\}$ be a partition of $\R^p$ defined as $E_0 \equiv \tilde{K}_\lambda$, $E_{m,\ell} \equiv \{y\in \R^p: \Pi_{\tilde{K}_\lambda}(y)\in F_{m,\ell}\}$. Let $E_0^\circ, E_{m,\ell}^\circ$ be the interiors of $E_0, E_{m,\ell}$, respectively. Since $\tilde{K}_\lambda$ is a polyhedron, it follows by \cite[Equation (3.6) and Remark 3.3]{kato2009degrees} that when $\hat{\theta}^0\in E^\circ_{m,\ell}$,
	\begin{align*}
	\frac{\partial\hat{\theta}}{\partial\hat{\theta}^0} = B_{m,\ell}(B_{m,\ell}^\top X^\top XB_{m,\ell})^{-1}B_{m,\ell}^\top X^\top X, 
	\end{align*}
	where $B_{m,\ell} = [b_{1,\ell},\ldots,b_{p-m,\ell}]\in \R^{p\times (p-m)}$ whose columns are linearly independent and span the tangent space at any point in $F_{m,\ell}$. Hence on the event $\{\hat{\theta}^0 \in E^\circ_{m,\ell}\}$, 
	\begin{align*}
	J_{\hat{\mu}}(\xi) = XB_{m,\ell}(B_{m,\ell}^\top X^\top XB_{m,\ell})^{-1}(B_{m,\ell}^\top X^\top),
	\end{align*}
	which is a projection matrix onto the column space of $XB_{m,\ell}$. In other words, a.e. on $\R^n$,
	\begin{align}\label{ineq:jacobian_lasso}
	J_{\hat{\mu}}(\xi) = \bm{1}_{\hat{\theta}^0\in E_0^\circ}\cdot Z_0 + \sum_{m=1}^p\sum_{\ell=1}^{N_m}\bm{1}_{\hat{\theta}^0\in E^\circ_{m,\ell}}\cdot Z_{m,\ell},
	\end{align}
	where
	\begin{align*}
	Z_0\equiv X(X^\top X)^{-1}X^\top, \quad Z_{m,\ell}\equiv XB_{m,\ell}(B_{m,\ell}^\top X^\top XB_{m,\ell})^{-1}(B_{m,\ell}^\top X^\top).
	\end{align*}
	Hence,
	\begin{align*}
	\E_{\mu_0} J_{\hat{\mu}}(\xi) =Z_0\cdot\Prob_{\mu_0}\big(\hat{\theta}^0\in E^\circ_0\big)+\sum_{m=1}^p\sum_{\ell=1}^{N_m}\Prob_{\mu_0}\big(\hat{\theta}^0\in E^\circ_{m,\ell}\big)\cdot Z_{m,\ell},
	\end{align*}
	and therefore
	\begin{align*}
	\pnorm{\E_{\mu_0} J_{\hat{\mu}}}{F}^2&=\sum_{i,j=1}^n \big(\E_{\mu_0} J_{\hat{\mu}}\big)_{i j}^2 \\
	&= \sum_{i,j=1}^n \bigg[(Z_0)_{ij}\Prob\big(\hat{\theta}^0 \in E^\circ_0\big) + \sum_{m,\ell}\Prob\big(\hat{\theta}^0\in E^\circ_{m,\ell}\big)(Z_{m,\ell})_{ij}\bigg]^2\\
	&\geq \sum_{i,j=1}^n \bigg[(Z_0)_{ij}\Prob_{\mu_0}\big(\hat{\theta}^0 \in E^\circ_0\big) - \Prob_{\mu_0}\big(\hat{\theta}^0 \notin E^\circ_0\big)\max_{m,\ell}\abs{(Z_{m,\ell})_{ij}}\bigg]_+^2\\
	&\stackrel{(*)}{\geq} \sum_{i,j=1}^n \Big[(Z_0)_{ij}\Prob_{\mu_0}\big(\hat{\theta}^0 \in E^\circ_0\big) - \Prob_{\mu_0}\big(\hat{\theta}^0 \notin E^\circ_0\big)\Big]_+^2\\
	&\stackrel{(*)}{\geq} \sum_{i,j=1}^n \Big[(Z_0)_{ij}- 2\Prob_{\mu_0}\big(\hat{\theta}^0 \notin E^\circ_0\big)\Big]_+^2\\
	&\stackrel{(**)}{\geq} \sum_{i,j=1}^n (Z_0)_{ij}^2/2 - 4n^2\Prob_{\mu_0}\big(\hat{\theta}^0 \notin E^\circ_0\big)^2\\
	&= p/2 - 4n^2\Prob_{\mu_0}\big(\hat{\theta}^0 \notin E^\circ_0\big)^2.
	\end{align*}
	Here we have used the following:
	\begin{itemize}
		\item In $(\ast)$, we apply the estimate 
		\begin{align*}
		\pm Z_{ij} = \pm e_i^\top Z e_j \leq \sup_{u,v: \pnorm{u}{}=\pnorm{v}{}=1} u^\top Z v = \pnorm{Z}{}\leq 1, \quad Z\in\{Z_0,Z_{m,\ell}\}.
		\end{align*}
		This means $\max_{i,j}\abs{Z_{ij}}\leq 1$ for $Z\in\{Z_0,Z_{m,\ell}\}$. 
		\item In $(\ast\ast)$ we use the estimate $(a - b)_+^2\geq a^2/2 - b^2$.
		\item In the last equality we use $\sum_{i,j} (Z_0)_{ij}^2 = \tr(Z_0Z_0^\top) = p$.
	\end{itemize}
	Thus, claim (1) follows.

	\noindent (2). Note that
	\begin{align*}
	\E_\mu \dv \hat{\mu} &= \E_\mu \tr (J_{\hat{\mu}}) \\
	& = \Prob_\mu\big(\hat{\theta}^0 \in E_0^\circ\big) \tr (Z_0) + \sum_{m=1}^p\sum_{\ell=1}^{N_m} \Prob_\mu\big(\hat{\theta}^0\in E^\circ_{m,\ell}\big) \tr ( Z_{m,\ell})\\
	& = p - p \Prob_\mu\big(\hat{\theta}^0 \notin E_0^\circ\big)+ \sum_{m=1}^p\sum_{\ell=1}^{N_m} \Prob_\mu\big(\hat{\theta}^0\in E^\circ_{m,\ell}\big) (p-m).
	\end{align*}
	Hence
	\begin{align*}
	\bigabs{\E_\mu \dv \hat{\mu}-p}&\leq 2p \Prob_\mu\big(\hat{\theta}^0 \notin E_0^\circ\big),
	\end{align*}
	proving claim (2).

	\noindent (3). 
	When $\hat{\theta}^0 \in E_0^\circ$, $\hat{\theta}=\hat{\theta}^0$ as $\hat{\theta}$ is the projection of $\hat{\theta}^0$ onto $\tilde{K}$ with respect to $\|\cdot\|_X\equiv \big(\cdot^\top X^\top X \cdot\big)^{1/2}$, cf.~\cite[Equation (1.6)]{kato2009degrees}. This means
	\begin{align*}
	\E_\mu \pnorm{\hat{\mu}-\mu}{}^2 & = \E_\mu\big( \pnorm{\hat{\mu}-\mu}{}^2 \bm{1}_{ \hat{\theta}^0 \in E_0^\circ }\big)+ \E_\mu \big(\pnorm{\hat{\mu}-\mu}{}^2 \bm{1}_{ \hat{\theta}^0 \notin E_0^\circ }\big)\\
	& = \E_\mu \big(\pnorm{X\hat{\theta}^0 - X\theta}{}^2  \bm{1}_{ \hat{\theta}^0 \in E_0^\circ } \big) + \E_\mu\big( \pnorm{\hat{\mu}-\mu}{}^2 \bm{1}_{ \hat{\theta}^0 \notin E_0^\circ }\big)\\
	& = \E_\mu \pnorm{X\hat{\theta}^0 - X\theta}{}^2  + R_{n,\mu} = p+R_{n,\mu},
	\end{align*}
	where 
	\begin{align*}
	R_{n,\mu} \equiv \E_\mu\big( \pnorm{\hat{\mu}-\mu}{}^2 \bm{1}_{ \hat{\theta}^0 \notin E_0^\circ }\big)- \E_\mu \big(\pnorm{Z_0 \xi}{}^2  \bm{1}_{ \hat{\theta}^0 \notin E_0^\circ }\big).
	\end{align*}
	As $\|\theta\|_1 \leq \lambda$, $\pnorm{\hat{\mu}-\mu}{}^2 \leq 2 \pnorm{Y-\hat{\mu}}{}^2+2\pnorm{\xi}{}^2 \leq 4\pnorm{\xi}{}^2$ by using $\pnorm{Y-\hat{\mu}}{}^2\leq \pnorm{Y-\mu}{}^2=\pnorm{\xi}{}^2$ via the definition of projection, we have
	\begin{align*}
	\E_\mu\big( \pnorm{\hat{\mu}-\mu}{}^2 \bm{1}_{ \hat{\theta}^0 \notin E_0^\circ }\big)& \leq 4 \sqrt{\E \pnorm{\xi}{}^4} \sqrt{\Prob_\mu\big( \hat{\theta}^0 \notin E_0^\circ  \big)}\leq C n \sqrt{\Prob_\mu\big( \hat{\theta}^0 \notin E_0^\circ  \big)}.
	\end{align*}
	On the other hand, a similar estimate yields
	\begin{align*}
	\E_\mu \big(\pnorm{Z_0 \xi}{}^2  \bm{1}_{ \hat{\theta}^0 \notin E_0^\circ }\big)\leq Cp \sqrt{\Prob_\mu\big( \hat{\theta}^0 \notin E_0^\circ  \big)}\leq Cn \sqrt{\Prob_\mu\big( \hat{\theta}^0 \notin E_0^\circ  \big)}.
	\end{align*}
	Hence
	\begin{align*}
	\bigabs{\E_\mu \pnorm{\hat{\mu}-\mu}{}^2-p}\leq C n \sqrt{\Prob_\mu\big( \hat{\theta}^0 \notin E_0^\circ  \big)}.
	\end{align*}
	This completes the proof of claim (3). 
\end{proof}

\begin{proof}[Proof of Theorem \ref{thm:lrt_clt_global_testing_lasso}]
	The first claim follows from Proposition \ref{prop:lasso_estimates}-(1)(3) and Theorem \ref{thm:lrt_clt_global_testing} (by ignoring the bias term in the denominator). For the second claim, by Proposition \ref{prop:lasso_estimates}-(2)(3),
	\begin{align*}
	m_\mu-\big(\pnorm{\mu-\mu_0}{}^2 + p\big) = 2\E_\mu \dv \hat{\mu}_{K_{X,\lambda}}-\E_\mu \pnorm{\hat{\mu}_{K_{X,\lambda}}-\mu}{}^2 - p= \mathcal{O}(n\cdot \mathfrak{p}_{\lambda,\mu}^{1/2}).
	\end{align*}
	This entails that $m_\mu - m_{\mu_0} = \|\mu - \mu_0\|^2 + n\cdot\mathcal{O}(\mathfrak{p}_{\lambda,\mu}^{1/2}\vee \mathfrak{p}_{\lambda,\mu_0}^{1/2})$. Furthermore, using Gaussian-Poincar\'e inequality along with Proposition \ref{prop:lasso_estimates}-(1)(3), we have
	\begin{align*}
	\sigma_{\mu_0}^2 \leq 4\E_{\mu_0}\|\hat{\mu} - \mu_0\|^2 \leq 4p + C(n\mathfrak{p}_{\lambda,\mu_0}^{1/2}) = \mathcal{O}(p),
	\end{align*}
	where the last inequality follows from the condition $n\mathfrak{p}_{\lambda,\mu_0}^{1/2} = \mathfrak{o}(1)$. This, along with lower bound for $\sigma_{\mu_0}^2$ derived in Proposition \ref{prop:lasso_estimates}-(1), yields that $\sigma_{\mu_0}^2 \asymp p$. Therefore, under the condition $n\cdot(\mathfrak{p}_{\lambda,\mu}^{1/2}\vee \mathfrak{p}_{\lambda,\mu_0}^{1/2}) = \mathfrak{o}(1)$, (\ref{cond:power_lrt_global_testing_0}) is satisfied automatically, and (\ref{cond:power_lrt_global_testing_1}) is equivalent to
	\begin{align*}
	\biggabs{\frac{ n\cdot\mathcal{O}\big(\mathfrak{p}_{\lambda,\mu}^{1/2}\vee \mathfrak{p}_{\lambda,\mu_0}^{1/2}\big) + \pnorm{\mu-\mu_0}{}^2}{ p^{1/2} } } \to \infty \Leftrightarrow \pnorm{\mu-\mu_0}{}\gg p^{1/4}.
	\end{align*}
	The proof is complete.
\end{proof}

\begin{proof}[Proof of Lemma \ref{lem:l1_deviation}]
	Recall that $\Sigma = X^\top X/n$. Note that 
	\begin{align*}
	\hat{\theta}^0 = (X^\top X)^{-1}X^\top Y = \theta + (X^\top X)^{-1}X^\top \xi \stackrel{d}{=}  \theta +(X^\top X)^{-1/2}Z
	\end{align*}
	with $Z\sim \mathcal{N}(0,I_p)$. For any $b\in \R^p$, let $f_b:\R^p\rightarrow \R$ be defined as $f_b(y)\equiv f_b(y;X)\equiv\iprod{(X^\top X)^{-1/2}y}{b}$. Then $\|\hat{\theta}^0\|_1 = \sup_{b:\|b\|_\infty \leq 1}\big[f_b(Z)+b^\top \theta\big]$. Hence by Gaussian concentration (cf. \cite[Theorem 5.8]{boucheron2013concentration}), for any $t > 0$,
	\begin{align}\label{ineq:l1_concentration}
	\Prob\Big(\|\hat{\theta}^0\|_1 - \E\|\hat{\theta}^0\|_1 > t\Big) \leq \exp\big(-t^2/2\sigma^2\big),
	\end{align}
	where $\sigma^2 = \sup_{b:\|b\|_\infty\leq 1}\var\big(f_b(Z)\big)$. Next we bound $\E\|\hat{\theta}^0\|_1$ and $\sigma^2$.	For $\sigma^2$, note that
	\begin{align*}
	\sigma^2 &= \sup_{b:\|b\|_\infty\leq 1} \E\iprod{(X^\top X)^{-1/2}Z}{b}^2 = n^{-1}\sup_{b:\|b\|_\infty\leq 1} b^\top\Sigma^{-1}b\\
	&\leq (p/n)\cdot \sup_{b:\|b\|_2\leq 1} b^\top\Sigma^{-1}b = p\big/\big(n\lambda_{\min}(\Sigma)\big).
	\end{align*}
	For the mean term, we have $\E\|\hat{\theta}^0\|_1\leq \|\theta\|_1 + \E\|(X^\top X)^{-1/2}\xi\|_1 = \|\theta\|_1 + \E\sup_{b:\|b\|_\infty\leq 1}f_b(Z)$. Note that the natural metric $d$ induced by the Gaussian process $\big(f_b(Z): b \in \R^p\big)$ takes the form
	\begin{align*}
	d^2(b_1,b_2) &\equiv \E\big(f_{b_1}(Z) - f_{b_2}(Z)\big)^2\\
	& = (b_1-b_2)^\top(n\Sigma)^{-1}(b_1-b_2) \leq n^{-1}\lambda_{\min}^{-1}(\Sigma)\|b_1-b_2\|^2,
	\end{align*}
	and a simple volume estimate yields that
	\begin{align*}
	\mathcal{N}(\epsilon,\{b:\|b\|_\infty\leq 1\}, d)\lesssim \big[(n\lambda_{\min}(\Sigma))^{-1/2}/\epsilon\big]^p.
	\end{align*}
	Hence by Dudley's entropy integral (cf. \cite[Theorem 2.3.6]{gine2015mathematical}), 
	\begin{align*}
	\E\|(X^\top X)^{-1/2}\xi\|_1 &\lesssim \int_0^\infty \sqrt{\log \big(1\vee \mathcal{N}(\epsilon, \{b:\|b\|_\infty\leq 1\}, d)\big)}\, \mathrm{d} \epsilon\\
	& \lesssim \sqrt{p\big/\big(n\lambda_{\min}(\Sigma)\big)}.
	\end{align*}
	The claim now follows from (\ref{ineq:l1_concentration}).
\end{proof}

\subsection{Proof of Theorem \ref{thm:clt_subspace_1d}}\label{pf:clt_subspace_1d}

By definition of $K_{0,k}$, we have $\delta_{K_{0,k}} = \dim(K_{0,k}) = k+1$. We will now show that
\begin{align}\label{ineq:clt_subspace_1d_1}
L_k^{-1} \big(\bm{1}_{k\geq 1}\log \log (16n)+\bm{1}_{k=0}\log (en)\big)\leq \delta_{K_{\uparrow,k}}\leq L_k \log (en),
\end{align}
where $L_k > 0$ only depends on $k$.

We first prove the upper bound in (\ref{ineq:clt_subspace_1d_1}) by induction. The baseline case $k = 0$ follows by \cite[Equation (D.12)]{amelunxen2014living}. Suppose the claim holds for some $k\in \mathbb{Z}_{\geq 0}$. For $k + 1$, note that $K_{\uparrow,k+1} = \cup_{\ell=1}^n K_{\uparrow,k+1;\ell}$ where $K_{\uparrow,k+1;\ell}$ contains all $\nu \in K_{\uparrow,k+1}$ such that $-\nu|_{[1:\ell]}$ is $k$-monotone, and $\nu|_{(\ell:n]}$ is $k$-monotone. Hence for any $\ell\in[1:n]$, it follows by \cite[Proposition 3.1]{amelunxen2014living} that
\begin{align*}
\delta_{K_{\uparrow,k+1;\ell}} \leq L_k\big(\log(e\ell) + \log(e(n-\ell)) \big) \leq 2L_k\log (en), 
\end{align*}
where the second inequality follows by induction. On the other hand, let $Z_k\equiv \sup_{\nu\in K_{\uparrow,k}\cap B(1)}\iprod{\nu}{\xi} = \|\Pi_{K_{\uparrow,k}}(\xi)\|$, then Gaussian concentration (cf. \cite[Theorem 5.8]{boucheron2013concentration}) entails that for any $t > 0$, 
\begin{align*}
\Prob(Z_k \geq \E Z_k + t) \leq \exp(-t^2/2).
\end{align*}
Hence, using the induction hypothesis $\E Z_k \leq \big(\E Z_k^2\big)^{1/2} \leq L_k^{1/2}\sqrt{\log (en)}$ and the union bound, it holds w.p. at least $1-\exp(-t)$ that
\begin{align*}
Z_{k+1} \equiv \sup_{\nu\in K_{\uparrow,k+1}\cap B(1)}\iprod{\nu}{\xi} &\leq \max_{1\leq \ell\leq n}\delta^{1/2}_{K_{\uparrow,k+1;\ell}} + \sqrt{2(t+\log (en))}\\
&\leq \big(2L_k\log (en)\big)^{1/2} + \sqrt{2(t+\log (en))}.
\end{align*}
Now the bound for $\delta_{K_{\uparrow,k+1}} =\E Z_{k+1}^2$ follows by integrating the tail.

Next we prove the lower bound in (\ref{ineq:clt_subspace_1d_1}) for $k\geq 1$. By Sudakov's minorization (cf. \cite[Theorem 2.4.12]{gine2015mathematical}), we have
\begin{align*}
\delta_{K_{\uparrow,k+1}}^{1/2} \geq \E Z_{k+1} &\gtrsim \sup_{\epsilon>0}\epsilon\sqrt{\log \mathcal{N}(\epsilon, K_{\uparrow,k+1}\cap B(1), \|\cdot\|)}\\
&\geq \sup_{\epsilon>0}\epsilon\sqrt{\log \mathcal{D}(2\epsilon, K_{\uparrow,k+1}\cap B(1), \|\cdot\|)},
\end{align*}
where $\mathcal{D}(\epsilon, T, d)$ is the maximal $\epsilon$-packing number of set $T$ with respect to the metric $d$. By taking $\epsilon$ to be small enough, the construction in \cite[Theorem 3.4]{shen2020phase} yields an $(2\epsilon)$-packing set of cardinality of the order $\log (en)$. This completes the lower bound proof.

Now the claim (1) follows from Theorem \ref{thm:power_lrt_subspace}-(1) and the lower bound in (\ref{ineq:clt_subspace_1d_1}). (2) follows from the upper bound in (\ref{ineq:clt_subspace_1d_1}) and Theorem \ref{thm:power_lrt_subspace}-(2). \qed

\appendix
\section{Additional proofs}

\subsection{Proof of Lemma~\ref{lem:var_proj}}\label{section:proof_preliminary} 
We provide the proof for (1)-(2) assuming $K$ is a polyhedral cone. The claim for a general convex cone $K$ follows from polyhedral approximation \cite[Section 7.3]{mccoy2014from}.

\noindent (1) As $V_K \equald \dv \Pi_K(\xi) = \tr\big(J_{\Pi_K}(\xi)\big)$, we have $\E V_K = \E \dv \Pi_K(\xi) = \E \iprod{\xi}{\Pi_K(\xi)} = \E \pnorm{\Pi_K(\xi)}{}^2 = \delta_K$.

\noindent (2)	The claim is proved in \cite[Proposition 4.4]{mccoy2014from} using the `Master Steiner formula', cf.~\cite[Theorem 3.1]{mccoy2014from}, which is a restatement of the chi-bar squared distribution. Below we provide a simple alternative proof of this claim using Gaussian integration-by-parts only. 

By expanding $V_K\equald \iprod{\xi}{\Pi_K(\xi)} - \big(\iprod{\xi}{\Pi_K(\xi)}-\dv \Pi_K(\xi)\big)$ and noting that $\E\iprod{\xi}{\Pi_K(\xi)} = \E \dv \Pi_K(\xi)$,
\begin{align*}
\var(V_K) 
& = \var \big(\iprod{\xi}{\Pi_K(\xi)}-\dv \Pi_K(\xi)\big) + \var\big(\iprod{\xi}{\Pi_K(\xi)} \big)\\
&\qquad -2 \E\big[\big(\iprod{\xi}{\Pi_K(\xi)}-\dv \Pi_K(\xi)\big) \iprod{\xi}{\Pi_K(\xi)} \big]\\
& = \E \tr J_{\Pi_K}^2(\xi)+\E \pnorm{\Pi_K(\xi)}{}^2+ \var\big(\iprod{\xi}{\Pi_K(\xi)} \big)\\
&\qquad -2\E\big[\iprod{\Pi_K(\xi)}{ \nabla\iprod{\xi}{\Pi_K(\xi)}}\big].
\end{align*}
The last equality follows from Gaussian integration-by-parts: (i) $\var(\iprod{\xi}{f(\xi)}-\dv f(\xi))=\E \tr J_f^2(\xi)+\E \pnorm{f(\xi)}{}^2$ (see e.g., \cite[Theorem 3]{stein1981estimation}, or \cite[Theorem 2.1]{bellec2018second}) and (ii) $\E[(\iprod{\xi}{f(\xi)}-\dv f(\xi))g(\xi)]=\E[\iprod{f(\xi)}{\nabla g(\xi)}]$ \cite[Equation (2.4)]{bellec2018second}. Note that (i) $\nabla \iprod{\xi}{\Pi_K(\xi)} = \nabla \pnorm{\Pi_K(\xi)}{}^2 =\nabla \pnorm{\xi-\Pi_{K^\ast}(\xi)}{}^2 = 2(\xi-\Pi_{K^\ast}(\xi))= 2\Pi_K(\xi)$ using the fact $K$ is a cone and Lemma \ref{lem:proj_basic}-(1), and (ii) $\E \tr J_{\Pi_K}^2(\xi) = \E \tr J_{\Pi_K}(\xi) = \E \dv \Pi_K(\xi)$ using the fact that when $K$ is polyhedral, $J_{\Pi_K}$ is a.e. a projection matrix (cf.~\cite[Remark 3.3]{kato2009degrees}). Finally using that $\E \dv \Pi_K(\xi) = \E \iprod{\xi}{\Pi_K(\xi)} = \E \pnorm{\Pi_K(\xi)}{}^2$ to conclude. 

\noindent (3)	The right inequality follows by an application of the improved Gaussian-Poincar\'e inequality stated in \cite[Theorem A.2]{goldstein2017gaussian} as follows: By Lemma \ref{lem:proj_basic}-(1) again, $\nabla \pnorm{\Pi_K(\xi)}{}^2 = \nabla \pnorm{\xi-\Pi_{K^\ast}(\xi)}{}^2 = 2(\xi-\Pi_{K^\ast}(\xi)) = 2\Pi_{K}(\xi)$, so  \cite[Theorem A.2]{goldstein2017gaussian} yields that
\begin{align*}
\var( \pnorm{\Pi_K(\xi)}{}^2)&\leq \frac{1}{2} \E \bigpnorm{\nabla \pnorm{\Pi_K(\xi)}{}^2}{}^2 + \frac{1}{2} \bigpnorm{\E \nabla\pnorm{\Pi_K(\xi)}{}^2}{}^2\\
& = 2\E\pnorm{\Pi_K(\xi)}{}^2+2 \pnorm{\E \Pi_K(\xi)}{}^2.
\end{align*}
The left inequality is an immediately consequence of (2). 
\qed

\subsection{Proof of Proposition \ref{prop:lrt_clt}}\label{section:proof_lrt_clt}

Recall the following second-order Poincar\'e inequality due to \cite{chatterjee2009fluctuations}.
\begin{lemma}[Second-order Poincar\'e inequality]\label{lem:sec_poincare}
	Let $\xi$ be an $n$-dimensional standard normal random vector. Let $F: \R^n \to \R$ be absolute continuous such that $F$ and its derivatives have sub-exponential growth at $\infty$. Let $\xi'$ be an independent copy of $\xi$. Define $T:\R^n\to \R$ by
	\begin{align*}
	T(y)\equiv \int_0^1 \frac{1}{2\sqrt{t}} \iprod{\nabla F (y)}{ \E_{\xi^\prime} \nabla F(\sqrt{t}y + \sqrt{1-t}\xi^\prime)}\,\mathrm{d}{t}.
	\end{align*}
	Then with $W\equiv F(\xi)$,
	\begin{align*}
	&d_{\mathrm{TV}}\bigg( \frac{W-\E W}{\sqrt{ \mathrm{Var}(W)}}, \mathcal{N}(0,1)\bigg)\leq   \frac{2 \sqrt{\mathrm{Var}(T(\xi))} }{ \mathrm{Var}(W)}.
	\end{align*}
\end{lemma} 
\begin{proof}
	Let $W' \equiv \big(F(\xi)-\E F(\xi)\big)/ \sqrt{\mathrm{Var}(F(\xi))}$, and $T' \equiv T/ \mathrm{Var}(F(\xi))$. Then \cite[Lemma 5.3]{chatterjee2009fluctuations} says that
	\begin{align*}
	d_{\mathrm{TV}}(W',\mathcal{N}(0,1))\leq 2\sqrt{\mathrm{Var}(T'(\xi))} = 2\sqrt{\mathrm{Var}(T(\xi))}/\mathrm{Var}(F(\xi)).
	\end{align*} 
	The claim follows by the invariance of the total variation metric by translation and scaling.
\end{proof}

\begin{proof}[Proof of Proposition \ref{prop:lrt_clt}]
	For any fixed $\mu\in\R^n$, let $F(\xi)\equiv F_\mu(\xi)\equiv \pnorm{\mu+\xi-\Pi_{K_0}(\mu+\xi)}{}^2- \pnorm{\mu+\xi-\Pi_{K}(\mu+\xi)}{}^2$. By Lemma \ref{lem:proj_basic}-(1), we have
	\begin{align*}
	\nabla F(\xi)&=2(\mu+\xi-\Pi_{K_0}(\mu+\xi)) - 2(\mu+\xi-\Pi_{K}(\mu+\xi)) \\
	&= 2\big(\Pi_{K}(\mu+\xi)-\Pi_{K_0}(\mu+\xi)\big). 
	\end{align*}
	To use the second-order Poincar\'e inequality, let $\xi'$ be an independent copy of $\xi$ and $\xi_t \equiv \sqrt{t}\xi+\sqrt{1-t}\xi'$, and let
	\begin{align*}
	T(\xi)&\equiv \int_0^1 \frac{1}{2\sqrt{t}} \iprod{\nabla F(\xi)}{\E_{\xi'} \nabla F(\xi_t)}\,\mathrm{d}{t}\\
	& = 4\E_{\xi'} \int_0^1 \frac{1}{2\sqrt{t}}\iprod{ \Pi_{K}(\mu+\xi)-\Pi_{K_0}(\mu+\xi)}{ \Pi_{K}(\mu+\xi_t)-\Pi_{K_0}(\mu+\xi_t)}\,\mathrm{d}{t}.
	\end{align*}
	Hence 
	\begin{align*}
	\nabla T(\xi) &= 4\E_{\xi'} \int \frac{1}{2\sqrt{t}}\bigg[ (J_{\Pi_{K}}-J_{\Pi_{K_0}})(\mu+\xi)^\top\big(\Pi_{K}(\mu+\xi_t)-\Pi_{K_0}(\mu+\xi_t) \big)\\
	&\qquad\qquad + \sqrt{t} (J_{\Pi_{K}}-J_{\Pi_{K_0}})(\mu+\xi_t)^\top\big(\Pi_{K}(\mu+\xi)-\Pi_{K_0}(\mu+\xi)\big)\bigg]\,\mathrm{d}{t}.
	\end{align*}
	The terms involved in the integral in $T$ are all absolute continuous, so we may continue to use Gaussian-Poincar\'e inequality:
	\begin{align*}
	&\var(T(\xi)) \leq \E \pnorm{\nabla T(\xi)}{}^2\\
	&\leq 16 \int_0^1 \frac{1}{2\sqrt{t}} \E \bigg\lVert (J_{\Pi_{K}}-J_{\Pi_{K_0}})(\mu+\xi)^\top\big(\Pi_{K}(\mu+\xi_t)-\Pi_{K_0}(\mu+\xi_t) \big)\\
	&\qquad\qquad +  \sqrt{t}(J_{\Pi_{K}}-J_{\Pi_{K_0}})(\mu+\xi_t)^\top\big(\Pi_{K}(\mu+\xi)-\Pi_{K_0}(\mu+\xi)\big) \bigg\lVert^2\,\mathrm{d}{t}\\
	&\qquad\qquad\qquad\qquad\hbox{(by Jensen's inequality applied to the measure $\mathrm{d}{t}/2\sqrt{t}$)}\\
	&\leq 16\times 8 \int_0^1 \frac{1}{2\sqrt{t}} \bigg(\E \pnorm{\Pi_{K}(\mu+\xi_t)-\Pi_{K_0}(\mu+\xi_t) }{}^2\\
	&\qquad\qquad\qquad\qquad\qquad + \E \pnorm{\Pi_{K}(\mu+\xi)-\Pi_{K_0}(\mu+\xi) }{}^2\bigg)\,\mathrm{d}{t}.
	\end{align*}
	Here in the last inequality we used that $\pnorm{J_{\Pi_{K}}-J_{\Pi_{K_0}}}{}\leq \pnorm{J_{\Pi_{K}}}{}+\pnorm{J_{\Pi_{K_0}}}{}\leq 2$. Now using that $\xi_t$ has the same distribution as $\xi$ for each $t \in [0,1]$, we arrive at
	\begin{align*}
	\var(T(\xi))\leq 16^2 \E \pnorm{\hat{\mu}_{K}-\hat{\mu}_{K_0}}{}^2.
	\end{align*}
	The claim now follows from the second-order Poincar\'e inequality in Lemma \ref{lem:sec_poincare}. 
\end{proof}

\section*{Acknowledgments}
The authors would like to thank two referees and an Associate Editor for their helpful comments and suggestions that significantly improved the exposition of the paper.

\bibliographystyle{amsalpha}
\bibliography{mybib}

\end{document}